\documentclass{article}

\usepackage[english]{babel}
 
\usepackage{amsmath, latexsym, amsfonts, amssymb, amsthm, amscd}
\usepackage{mathrsfs}
\usepackage{graphics,epsf,psfrag}
\usepackage[lofdepth,lotdepth]{subfig}
\usepackage{graphicx}
\usepackage{float}
\usepackage{url}
\usepackage{hyperref}
\usepackage{enumerate}
\usepackage[yyyymmdd,hhmmss]{datetime}

\usepackage{caption}
\newcommand\fnote[1]{\captionsetup{font=footnotesize}\caption*{#1}}

\usepackage{stmaryrd} 
\usepackage{enumitem}
\usepackage{tikz}

\usepackage{xcolor}
\definecolor{lightgray}{gray}{0.85}
\definecolor{midgray}{gray}{0.4}
 \usepackage{soul}
 
\usepackage{fancybox}

\numberwithin{equation}{section}
\theoremstyle{remark} 
 
\theoremstyle{plain} 
\newtheorem{thm}{Theorem}[section] 
\newtheorem{cor}[thm]{Corollary}
\newtheorem{prop}[thm]{Proposition}
\newtheorem{lem}[thm]{Lemma}
\theoremstyle{definition} 
\newtheorem{deff}[thm]{Definition}

\newtheorem{remark}[thm]{Remark}

\title
{Stability of wandering bumps  for  Hawkes processes interacting on the circle}
\author{Zo\'e \textsc{Agathe-Nerine}\vspace{0.5cm}\\
{\small Universit\'e Paris Cit\'e, CNRS, MAP5, F-75006 Paris, France}\\
{\small zoe.agathe-nerine@u-paris.fr}}
\date{\today}

\usepackage[raggedright]{titlesec}

\begin{document}
\maketitle

\begin{abstract}
We consider a population of Hawkes processes modeling the activity of $N$ interacting neurons. The neurons are regularly positioned on the circle $[-\pi, \pi]$, and the connectivity between neurons is given by a cosine kernel. The firing rate function is a sigmoid. The large population limit admits a locally stable manifold of stationary solutions. The main result of the paper concerns the long-time proximity of the synaptic voltage of the population to this manifold in polynomial times in $N$. We show in particular that the phase of the voltage along this manifold converges towards a Brownian motion on a time scale of order $N$.
\end{abstract}


\section{Introduction}
\subsection{Hawkes Processes and Neural Field Equation} 

In the present paper we study the large time behavior of a population of interacting and spiking neurons indexed by $i=1,\cdots,N$, $N\geq 1$, as the size of the population $N$ tends to infinity.   We model the activity of a neuron by a point process where each point represents the time of a spike: for $i=1,\cdots,N$, $Z_{N,i}(t)$ counts the number of spikes during the time interval $[0,t]$ of the $i$th neuron of the population. Denoting $\lambda_{N,i}(t)$ as the conditional intensity of $Z_{N,i}$ at time $t$, that is
$$\mathbf{P}\left( Z_{N,i} \text{ jumps between} (t,t+dt) \vert \mathcal{F}_t\right)= \lambda_{N,i}(t)dt,$$
where $\mathcal{F}_t:=\sigma\left( Z_{N,i}(s), s\leq t, 1\leq i\leq N\right)$, we want to account for the dependence of the activity of a neuron on the past of the whole population : the spike of one neuron can trigger other spikes. \textit{Hawkes processes} are then a natural choice to emphasize this interdependency and we take here
\begin{equation}\label{eq:def_lambdaN_specific}
\lambda_{N,i}(t)=f_{\kappa,\varrho}\left( \rho(x_i)e^{-t}+\dfrac{2\pi}{N}\sum_{j=1}^N \cos(x_i-x_j) \int_0^{t-} e^{-(t-s)} dZ_{N,j}(s)\right), \quad i=1,\ldots,N.
\end{equation}
The neurons are located on the circle $S=(-\pi,\pi]$ with positions $(x_i)_{1\leq i \leq N}$ regularly distributed, that is $x_i=\dfrac{\pi}{N}\left( 2i - N\right)$. We subdivide $S$ into $N$ intervals of length $2\pi/ N$ denoted by
\begin{equation}\label{eq:def_Bi}
B_{N,i}=\left(x_{i-1},x_i\right] \quad \text{ for } 1\leq i \leq N, \text{with } x_0:=-\pi.
\end{equation}
The function $f_{\kappa,\varrho} ~:~  \mathbb{R}  \longrightarrow \mathbb{R}_+$ models the synaptic integration of neuron $i$ with respect to the input of the other neurons $j$ in the population, modulated by the spatial kernel $\cos(x_i-x_j)$. It is chosen as a sigmoid with parameters $(\kappa,\varrho)$, $\kappa>0$, $\varrho\in (0,1)$, that is 
\begin{equation}\label{eq:def_sigmoid}
f_{\kappa,\varrho}(u):=\left(1+e^{-(u-\varrho)/\kappa}\right)^{-1}.
\end{equation}
The function $\rho ~:~ S  \longrightarrow \mathbb{R}$ represents the initial inhomogeneous voltage of the population and leaks at rate 1. The exponential term $e^{-(t-s)}$ in the integral in \eqref{eq:def_lambdaN_specific} quantifies how a jump lying back $t-s$ time units in the past affects the present (at time $t$) intensity: each neuron tends to forget progressively its past. The main object of interest of the paper is the synaptic voltage
\begin{equation}\label{eq:def_UiN}
U_{N,i}(t)= \rho(x_i)e^{-t}+\dfrac{2\pi}{N}\sum_{j=1}^N \cos(x_i-x_j) \int_0^{t} e^{-(t-s)} dZ_{N,j}(s)=:\rho(x_i)e^{-t} + X_{N,i}(t),
\end{equation} 
(i.e. $\lambda_{N,i}(t)=f_{\kappa,\varrho}\left(U_{N,i}(t-)\right)$) and more precisely the random profile defined for all $x\in S$ by:
\begin{equation}\label{eq:def_UN}
U_N(t)(x):=\sum_{i=1}^N U_{N,i}(t) \mathbf{1}_{x\in B_{N,i}}.
\end{equation}

The specific form of \eqref{eq:def_lambdaN_specific} originates from the so-called ring model introduced by \cite{Shriki2003}, modelling the activity of neurons in the visual cortex on a mesoscopic scale. Here each position $x\in S$ represents a prefered orientation for each neuron, see the biological works of \cite{Georgopoulos1982,Bosking1997} and the mathematical works of \cite{veltzFaugeras2010,MacLaurinBressloff2020} amongst others. We are looking here at the microscopic counterpart of this model. It means that neurons that prefer close orientation tend to excitate each others, whereas neurons with opposite orientation inhibit each others. Making $\kappa\to 0$ in \eqref{eq:def_sigmoid}, we see that $f_{\kappa, \varrho}$ converges towards $H_\varrho$ the Heaviside function 
\begin{equation}
\label{eq:def_heaviside}
H_\varrho(u)= \mathbf{1}_{u\geq \varrho}.
\end{equation}
Hence for $\kappa$ small, a neuron can spike only when it has a high potential with rate approximately 1, and with rate approximately 0 otherwise. \\

This model \eqref{eq:def_lambdaN_specific} is a specific case of a larger class of mean-field  Hawkes processes for which one can write the intensity in the form
\begin{equation}\label{eq:def_lambda_generic}
\lambda_{N,i}(t)=\mu_t(x_i)+f\left( v_{t}(x_i)+\dfrac{1}{N}\sum_{j=1}^N w_{ij}^{(N)} \int_0^{t-}h(t-s) dZ_{N,j}(s)\right),  \quad i=1,\ldots,N.
\end{equation}
The current model \eqref{eq:def_lambdaN_specific} corresponds to the choice $h(t)=e^{-t}$ and $w_{ij}^{(N)}=2\pi\cos(x_i-x_j)$. In \eqref{eq:def_lambda_generic}, the neurons are placed in a spatial domain $I$ endowed with $\nu$ a probability measure that describes the macroscopic distribution of the positions. The parameter function $\mu_t~:~ I \longrightarrow \mathbb{R}_+$ represents a spontaneous activity of the neuron at time $t$, $v_t~:~ I \longrightarrow \mathbb{R}$ a past activity, $h$ is the memory kernel of the system, $f~:~ \mathbb{R} \longrightarrow \mathbb{R}^+$ and $ w_{ij}^{(N)} $ represents the interaction between neurons $i$ and $j$. For a suitable class of connectivity sequence $(w_{ij}^{(N)})$ that can be approximated by some macroscopic interaction kernel $w(x,y)$ as $N\to\infty$ (see \cite{CHEVALLIER20191,agathenerine2021multivariate} for precise statements), a usual propagation of chaos result as $N\to\infty$ (see \cite[Theorem 8]{delattre2016}, \cite[Theorem 1]{CHEVALLIER20191},  \cite[Theorem 3.10]{agathenerine2021multivariate}) may be stated as follows: for fixed $T>0$, there exists some $C(T)>0$ such that
\begin{equation}\label{eq:chaos_generic}
\sup_{1\leq i \leq N} \mathbf{E}\left(\sup_{s\in [0,T]} \left\vert Z_{N,i}(s) - \overline{Z}_{i}(s) \right\vert \right) \leq \dfrac{C(T)}{\sqrt{N}},
\end{equation}
where the limiting process $\left(\overline{Z}_{i}, i=1,\ldots, N\right)$ consists of independent copies of inhomogeneous Poisson process suitably coupled to $Z_{N,i}$ with intensity $(\lambda_t(x_i))_{t\geq 0}$ solving
\begin{equation}\label{eq:def_lambda_lim_generic}
\lambda_t(x)=\mu_t(x)+f\left( v_t(x)+\int_I w(x,y) \int_0^{t} h(t-s) \lambda_s(y)ds\nu(dy)\right)
\end{equation} 
(see the above references for details on this coupling). Moreover, for the specific choice $h(t)=e^{-t}$, denoting the macroscopic potential of a neuron (the synaptic current) with position $x$ at time $t$ by 
\begin{equation}\label{eq:def_utx}
u_t(x):=v_t(x)+\int_I w(x,y) \int_0^{t} h(t-s) \lambda_s(y)ds\nu(dy),
\end{equation}
an easy computation (see \cite{CHEVALLIER20191}) gives that, when $v_t(x)=\rho(x)e^{-t}$, $u$ solves the \emph{Neural Field Equation} (NFE) 
\begin{equation}\label{eq:NFE_gen}
\dfrac{\partial u_t(x)}{\partial t}=- u_t(x)+\int_I w(x,y)f(u_t(y))\nu(dy), \quad t\geq 0,
\end{equation}
with initial condition $u_0=\rho$. The NFE that first appears in \cite{Wilson1972} has been extensively studied in the literature, mostly from a phenomenological perspective \cite{Amari1977}, and is an important example of macroscopic neural dynamics with non-local interactions (we refer to \cite{bressloff_waves_2014} for an extensive review on the subject). Let us mention here an important point: whereas the analysis of \cite{CHEVALLIER20191} requires the measure $\nu$ in \eqref{eq:NFE_gen} to be a probability measure on $I$, the historical version of the NFE was originally studied when $\nu(dy)=dy$ is the Lebesgue measure on $\mathbb{R}$. In this last case, thanks to its translation invariance of the Lebesgue measure, one can show the existence of travelling waves solutions to \eqref{eq:NFE_gen}, see \cite{ermentrout_mcleod93,lang_stannat2016} for details. The same analysis when $\nu(dy)=dy$ is remplaced by a probability measure fails, as translation invariance of \eqref{eq:NFE_gen} is then broken. In this respect, the present choice of $I=S$ and $\nu(dy)= \frac{\mathbf{1}_{[-\pi,\pi)}}{2\pi}dy$ combines the two previous advantages: $\nu$ is a probability measure (hence the previous analysis when $N\to\infty$ applies) and translation invariance is preserved in the present periodic case. It can be shown (\cite{kilpatrick2013}) that \eqref{eq:NFE_gen} exhibits localized patterns (\emph{wandering bumps}) which are stationary pulse solutions.\\

We are interested in this paper in the long time behavior of the microscopic system  \eqref{eq:def_lambdaN_specific} and its proximity to these wandering bumps. Before focusing on the microscopic scale, we say a few words on the behavior of the macroscopic system \eqref{eq:def_lambda_lim_generic}/\eqref{eq:def_utx}. In the pure mean-field case (when $w_{ij}^{(N)}=1$ for all $i$, $j$), the spatial dependency is no longer relevant and \eqref{eq:def_lambda_lim_generic} reduces to the scalar nonlinear convolution equation $\lambda_t=\mu_t+f(v_t+\int_0^t h(t-s)\lambda_s ds)$. An easy instance concerns the so-called linear case where $f(x)=x$, $\mu_t=\mu$ and $\nu_t=0$: in this situation the behavior of $\lambda_t$ as $t\to\infty$ is well known. There is a phase transition (\cite[Theorems 10,11]{delattre2016}) depending on the memory kernel $h$: when $\Vert  h \Vert_1=\int_0^\infty h(t) dt<1$ (the \emph{subcritical case}), $\lambda_t\xrightarrow[t\to\infty]{}\dfrac{\mu}{1-\Vert h \Vert_1}$, whereas when $\Vert h \Vert_1>1$ (the \emph{supercritical case}), $\lambda_t\xrightarrow[t\to\infty]{}\infty$. This phase transition was extended to the inhomogeneous case in \cite{agathenerine2021multivariate} (and more especially where the interaction is made through the realisation of weighted random graphs), and the existence of such a phase transition now reads in terms of $\Vert h \Vert_1 r_\infty<1$ (then $\lambda_t(x)\to\ell(x)$ the unique solution of $\ell(x)=\mu(x)+ \int_I w(x,y)\Vert h \Vert_1 \ell(y)\nu(dy)$) and $\Vert h \Vert_1 r_\infty>1$ (then $\Vert \lambda_t\Vert_2\to\infty$), where $r_\infty$ is the spectral radius of the interaction operator $T_W g(x)\mapsto \int_I w(x, y)g(y)\nu(dy)$. In the fully inhomogeneous case and nonlinear case ($f$ no longer equal to $Id$), a sufficient condition for convergence of $\lambda_t$ is given in \cite{agathenerine_longtime_arxiv}: whenever
\begin{equation}\label{eq:def_subcritical}
\Vert f' \Vert_\infty \Vert h \Vert_1 r_\infty <1,
\end{equation}
$\lambda_t$ converges to $\ell$ as $t\to\infty$, $\ell$ being the unique solution to 
\begin{equation}
\label{eq:def_ell_previous}
\ell=\mu+f\left(\Vert h \Vert_1 T_W \ell\right).
\end{equation}
Note that the present model \eqref{eq:def_lambdaN_specific} obviously does not satisfy \eqref{eq:def_subcritical}, as $\Vert f'\Vert_\infty$ is very large (recall \eqref{eq:def_sigmoid}: $f$ is a sigmoid close to the Heaviside function).  Understanding the longtime behavior of $\lambda_t$ when \eqref{eq:def_subcritical} does not hold may be a difficult task for general $h$. However the present model is sufficiently simple to be analyzed rigorously: as it was originally noted by \cite{kilpatrick2013}, the stationary points of \eqref{eq:def_lambda_lim_generic} when $w$ is a cosine can be found by solving an appropriate fixed point relation (see \eqref{eq:A_self} below) and by invariance by translation, each fixed-point gives rise to a circle of stationary solutions to \eqref{eq:def_lambda_lim_generic}. One part of the proof will be to show the local stability of these circles (extending the results of \cite{kilpatrick2013} when $f$ is the Heaviside function).\\

The main concern of the paper is to analyse the microscopic system \eqref{eq:def_UN} on a long time scale. An issue common to all mean-field models (and their perturbations) is that there is, in general, no possibility to interchange the limits $N\to \infty$ and $t\to\infty$. Specifying to Hawkes processes, the constant $C(T)$ in \eqref{eq:chaos_generic} is of the form $\exp(CT)$, such that \eqref{eq:chaos_generic} remains only relevant up to $T \sim c \log N$ with $c$ sufficiently small. In the linear subcritical case, $C(T)$ is linear ($C(T)=CT$) so that the mean-field approximation remains relevant up to $T = o\left( \sqrt{N}\right)$ (\cite{delattre2016}). In a previous work \cite{agathenerine_longtime_arxiv}, we showed that, in the subcritical regime defined by \eqref{eq:def_subcritical} with $h(t)=e^{-t}$, the macroscopic intensity \eqref{eq:def_lambda_lim_generic} converges to $\ell$ defined by \eqref{eq:def_ell_previous} and the  microscopic intensity \eqref{eq:def_lambda_generic} remains close to this limit up to polynomial times in $N$. Here, the main difference is that \eqref{eq:def_utx} admits a manifold of stable stationary solutions parameterized by $S$, instead of a unique one. We show here that, with some initial condition close to this manifold, our microscopic process \eqref{eq:def_UN} stays close to the manifold up to time horizons that are polynomial in $N$, and moreover the dynamics of the microscopic current follows a Brownian motion on the manifold. 

\paragraph{Organization of the paper} The paper is organized as follows: after introducing some notations, we start in Section \ref{S:model} by introducing the precise mathematical set-up. In Section \ref{S:R}, we present the main results of our paper. Section \ref{S:main_result} is divided into three parts: in the first part \ref{S:stationary_result}, we present the deterministic dynamics of \eqref{eq:NFE_specific} and the manifold of stationary solutions $\mathcal{U}$ defined in \eqref{eq:def_U_circle}. In the second part we introduce two ways of defining some phase reduction along $\mathcal{U}$, the variational phase (Proposition \ref{prop:phase_proj}) and isochronal phase (Proposition \ref{prop:ex_isochron}). In the last part, Theorem \ref{thm:UN_close_U} ensures that if the system is close to $\mathcal{U}$, it stays so for a long time, and  with Theorem \ref{thm:theta_fluc}, we analyze the dynamics of the isochronal phase of $U_N$ along $\mathcal{U}$. Such dynamics are represented in the simulations of Figure \ref{fig:wandering1}. In Section \ref{S:litterature}, we explain how our paper is linked to the present litterature on the subject. In Section \ref{S:strategy}, we sketch the strategy of proof we follow. Section \ref{S:proof_stat} collects the proofs of the results of Sections \ref{S:stationary_result} and \ref{S:projection}, Section \ref{S:long_time} concerns the proof of the proximity between $U_N$ and $\mathcal{U}$ seen in Theorem \ref{thm:UN_close_U} and Section \ref{S:fluct} is devoted to prove the diffusive behavior of $U_N$ along $\mathcal{U}$ seen in Theorem \ref{thm:theta_fluc}. Some technical estimates and computations are gathered in the appendix.

\paragraph{Acknowledgments.}
This is a part of my PhD thesis. I would like to warmly thank my PhD
supervisors Eric \textsc{Luçon} and Ellen \textsc{Saada} for introducing this subject, for their useful advices and for their encouragement and guidance. This research has been conducted within the FP2M federation (CNRS FR 2036), and is supported by ANR-19-CE40-0024 (CHAllenges in MAthematical NEuroscience) and ANR-19-CE40-0023 (Project PERISTOCH). I would also like to thank Christophe \textsc{Poquet} for pointing out a mistake in a previous version of the paper.

\subsection{Notations and definition}
\subsubsection{Notations}

We denote by $C_{\text{parameters}}$ a constant $C>0$ which only depends on the parameters inside the lower index. These constants can change from line to line or inside a same equation, and when it is not relevant, we just write $C$. For any $d\geq 1$, we denote by $\vert x\vert$ and $x \cdot y$ the Euclidean norm and scalar product of $x,y\in \mathbb{R}^d$. For $(E,\mathcal{A},\mu)$ a measured space, for a function $g$ in $L^p(E,\mu)$ with $p\geq 1$, we write $\Vert g \Vert_{E,\mu,p}:=\left( \int_E \vert g \vert^p d\mu \right)^\frac{1}{p}$. When $p=2$, we denote by $\langle \cdot,\cdot \rangle$ the Hermitian  scalar product in $L^2(E,\mu)$. Without ambiguity, we may omit the subscript $(E,\mu)$ or $\mu$. For a  real-valued bounded function $g$ on a space $E$,  we write $\Vert g \Vert _\infty := \Vert g \Vert _{E,\infty}=\sup_{x\in E} \vert g(x) \vert$. 

For $(E,d)$ a metric space, we denote by $ \Vert g \Vert_{\text{lip}} = \sup_{x\neq y} \vert g(x) - g(y) \vert / d(x,y)$ the Lipschitz seminorm of a real-valued function $g$ on $E$. We denote by $\mathcal{C}(E,\mathbb{R})$ the space of continuous functions from $E$ to $\mathbb{R}$, and $\mathcal{C}_b(E,\mathbb{R})$ the space of continuous bounded ones. For any $T>0$, we denote by $\mathbb{D}\left([0,T],E\right)$ the space of c\`adl\`ag (right continuous with left limits) functions defined on $[0,T]$ and taking values in $E$. For any integer $N\geq 1$, we denote by $\llbracket 1, N \rrbracket$ the set $\left\{1,\cdots,N\right\}$. 

For any $h,k,l \in E$, we denote by $Dg(h)[k]\in S$ the derivative of $g:E\to F$ at $h$ in the direction $k$, and similarly for second derivatives $D^2g(h)[k,l]$.

\subsubsection{Definition of the model}\label{S:model}

We define now formally our process of interest. Definition \ref{def:H} follows a standard representation of point processes as thinning of independent Poisson measures, see \cite{Ogata1988,delattre2016}.

\begin{deff}\label{def:H} Let $\left(\pi_i(ds,dz)\right)_{1\leq i \leq N}$ be a sequence of i.i.d. Poisson random measures on $\mathbb{R}_+\times \mathbb{R}_+$ with intensity measure $dsdz$. The multivariate counting process  $\left(Z_{N,1}\left(t\right),...,Z_{N,N}\left(t\right)\right)_{t\geq 0}$ defined by, for all $t\geq 0$ and $i \in \llbracket 1, N \rrbracket$:
\begin{equation}\label{eq:def_ZiN}
Z_{N,i}(t) = \int_0^t \int_0^\infty \mathbf{1}_{\{z\leq \lambda_{N,i}(s)\}} \pi_i(ds,dz),
\end{equation}
where $ \lambda_{N,i}$ is defined in \eqref{eq:def_lambdaN_specific} is called \emph{a multivariate Hawkes process} with set of parameters
$\left(N,\kappa,\varrho,\rho\right)$.
\end{deff}

It has been showed in several works (see e.g. \cite{agathenerine2021multivariate,delattre2016} amongst others) that the process defined in \eqref{eq:def_ZiN} is well posed in the following sense.
\begin{prop}
\label{prop:exis_H_N} For a fixed realisation of the family $\left(\pi_i\right)_{1\leq i \leq N}$,  there exists a pathwise unique multivariate Hawkes process (in the sense of Definition \ref{def:H}) such that for any $T<\infty$, $$\sup_{t\in [0,T]} \sup_{1\leq i \leq N} \mathbf{E}[Z_{N,i}(t)] <\infty.$$
\end{prop}
Proposition \ref{prop:exis_H_N} can be found in \cite[Propositions 2.5]{agathenerine2021multivariate}. In our framework, the macroscopic intensity \eqref{eq:def_lambda_lim_generic} population limits is
\begin{equation}\label{eq:def_lambda_lim_specific}
\lambda_t(x)=f_{\kappa,\varrho}\left( \rho(x)e^{- t}+\int_S \cos(x-y) \int_0^{t} e^{-(t-s)} \lambda_s(y)dsdy\right),
\end{equation}
and the neural field equation \eqref{eq:NFE_gen} becomes
\begin{equation}\label{eq:NFE_specific}
\dfrac{\partial u_t(x)}{\partial t}=- u_t(x)+\int_S \cos(x-y)f_{\kappa,\varrho}(u_t(y))dy.
\end{equation}

\begin{prop}
\label{prop:exis_lambda_barre}
Let $T>0$. There exists a unique solution  $(u_t)_{t\in [0,T]}$ in $\mathcal{C}_b(S, \mathbb{R})$ to \eqref{eq:NFE_specific} with initial condition $u_0=\rho$.
\end{prop}
Proposition \ref{prop:exis_lambda_barre} can be found in \cite[Propositions 2.7]{agathenerine2021multivariate}, and follows from a standard Gr\"{o}nwall estimate. We can then define the flow of \eqref{eq:NFE_specific} by $(t,g)\mapsto \psi_t(g)$, that is the solution at time $t$ of \eqref{eq:NFE_specific} starting from $g$ at $t=0$: 
\begin{equation}
\label{eq:def_flow}
\psi_t(g)(x)=e^{-t}g(x)+\int_0^t e^{-(t-s)} \int_S \cos(x-y) f_{\kappa,\varrho}(\psi_s(g)(x))ds.
\end{equation}

\section{Stability of wandering bumps for interacting Hawkes processes}\label{S:R}
\subsection{Main results}\label{S:main_result}
\subsubsection{Stationary solutions to  \eqref{eq:NFE_specific}}\label{S:stationary_result}

We are concerned here with the stationary solutions to \eqref{eq:NFE_specific}, that is
\begin{equation}
\label{eq:NFE_stat_gen}
u(x)= \int_{-\pi}^{\pi} \cos(x-y) f(u(y))dy.
\end{equation} 
We follow a similar approach to \cite{kilpatrick2013}, see Appendix \ref{S:stat_H}.

\begin{remark}\label{rem:invariance}
For a general choice of $f$, if $u$ is solution to \eqref{eq:NFE_stat_gen}, then for any $\phi$, $x\mapsto u(x+\phi)$ is also solution to \eqref{eq:NFE_stat_gen} by invariance of $S$. Expanding the cosine, \eqref{eq:NFE_stat_gen} becomes
$$u(x)= \cos(x) \int_{-\pi}^{\pi} \cos(y) f(u(y))dy +  \sin(x) \int_{-\pi}^{\pi} \sin(y)f(u(y))dy.$$ 
By translation symmetry, with no loss of generality we can ask  $\int_{-\pi}^{\pi} \sin(y)f(u(y))dy=0$ and solving \eqref{eq:NFE_stat_gen} means finding $A\geq 0$ such that
\begin{equation}\label{eq:A_self}
A=  \int_S \cos (y) f\left( A \cos(y) \right)dy.
\end{equation}
As \eqref{eq:NFE_stat_gen} is invariant by translation, any $A$ solution to \eqref{eq:A_self} gives rise to the set $\mathcal{U}_A:=\left\{x\mapsto A\cos(x+\phi),~ \phi \in [-\pi, \pi]\right\}$ of stationary solutions to \eqref{eq:NFE_stat_gen}.
\end{remark}

Recall \eqref{eq:def_heaviside}, when $f=H_\varrho$ the Heaviside function with threshold $\varrho$, \cite{kilpatrick2013} and \cite{veltzFaugeras2010} showed that for $\varrho\in [-1,1]$, the unique solutions to \eqref{eq:A_self}  are
\begin{equation}\label{eq:def_A0}
A=0, \quad A_-(0)=\sqrt{1+\varrho} - \sqrt{1-\varrho} \quad \text{ and }
A_+(0):=\sqrt{1+\varrho} + \sqrt{1-\varrho}.
\end{equation}
This result is recalled in Appendix \ref{S:stat_H}. One can show that the set $\mathcal{U}_{A_-(0)}$ is unstable whereas $\mathcal{U}_{A(0)}$ and $\mathcal{U}_{0}$ are locally stable. In the following we focus on the largest fixed point $A_+(0)$ which we rename for $A(0)$ by convenience. Recall that in the paper, we are under the assumption that $f=f_{\kappa,\varrho}$ defined in \eqref{eq:def_sigmoid} for a small fixed $\kappa$. As $f_{\kappa,\varrho}\xrightarrow[\kappa \to 0]{}H_\varrho$, our first result is that when $\kappa$ is close enough to 0, we can still find a stationary solution to \eqref{eq:NFE_specific} of the form $u=A(\kappa)\cos$ where $A(\kappa)$ is also close to $A(0)$. 

\begin{prop}\label{prop:stat_sol_sig}
Assume $\varrho\in (-1,1)$. Then there exists $\kappa_0>0$ and a function $A:(0,\kappa_0)\to (\vert \varrho \vert,+\infty)$ of class $C^1$ such that  for any $\kappa \in (0,\kappa_0)$, $u=A(\kappa) \cos$ is a stationary solution to \eqref{eq:NFE_specific} when $f=f_{\kappa,\varrho}$ and $A(\kappa)\xrightarrow[\kappa\to 0]{} A(0)$ given in \eqref{eq:def_A0}.
Moreover, there exists  $\kappa_1\in (0,\kappa_0)$ such that for any $\kappa \in (0,\kappa_1)$,  $1<I(1,\kappa)<2$ for $I(1,\kappa):=\int_S f_{\kappa,\varrho}'(A(\kappa)\cos(x))dx$.
\end{prop}

Proposition \ref{prop:stat_sol_sig} is based on a simple implicit function argument and is proved in the Appendix \ref{S:stat_sol_sig}. An illustration of this Proposition is done in Figure \ref{fig:fixed}: we see that for each $A$ solving \eqref{eq:A_self} for the Heaviside function, there is indeed another close $A$ solving \eqref{eq:A_self} for the sigmoid function with small $\kappa$. For the rest of the paper we fix $\varrho\in (-1,1)$, $\kappa<\kappa_1$ and  $A=A(\kappa)$ and may omit the indexes $(\kappa,\varrho)$. We have then established that
\begin{equation}\label{eq:def_U_circle}
\mathcal{U}:=\left(A\cos(\cdot+\phi)\right)_{\phi \in S}=: \left(u_{\phi}\right)_{\phi \in S}
\end{equation}
is a set of stationary solutions to \eqref{eq:NFE_specific}, which is a manifold parameterized by the circle $S$. To study the stability of these stationary solutions, we introduce linear operators that are also parameterized by the circle $S$.
 
 \begin{figure}
\centering
\includegraphics[width=0.7\textwidth]{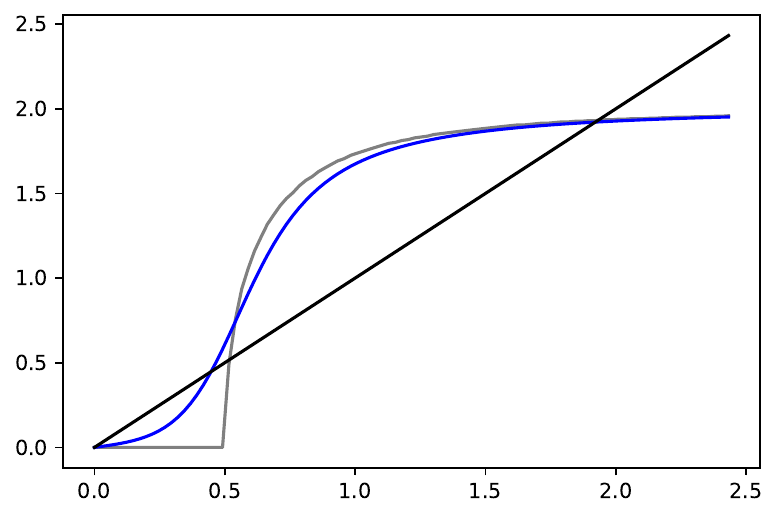}
\caption{Graph of $G:A\mapsto \int_S A\cos(x)f\left(A\cos(x)\right)dx$}
\label{fig:fixed}
\fnote{We represent the fixed-point function $G$ appearing in \eqref{eq:A_self} for the choice $f=H_{\varrho}$ defined in \eqref{eq:def_heaviside} in gray and its smooth version with $f=f_{\kappa,\varrho}$, defined in \eqref{eq:def_sigmoid} in blue. We chose $(\kappa,\varrho)=\left(\frac{1}{10},\frac{1}{2}\right)$.  The black line is the graph of $y=x$ and its intersections with the two other lines give the fixed points of $G$. Note that we are interested here on the fixed point on the far right, that is $A(0)$ for the gray line and $A(\kappa)$ for the blue line.}
\end{figure}

\begin{deff}\label{def:ops_phi}
Let $\phi\in S$, and define for any function $\psi\in L^2(S)$
\begin{align}
T_{\phi}\psi(x)&:=\int_S \cos(x-y)f'({u_{\phi}}(y))\psi(y)dy\label{eq:def_T_phi}\\
\mathcal{L}_{\phi}\psi&:=-\psi + T_{\phi}\psi \label{eq:def_L_phi}.
\end{align}
Define also $L^2_{\phi}:=L^2_{f'(u_{\phi})}$, that is the $L^2$ weighted space defined by the scalar product
$$\langle g_1,g_2\rangle_{2,\phi} = \int_S g_1(x)g_2(x)f '(u_{\phi}(x))dx.$$
We denote by $\Vert \cdot \Vert_{2,\phi}$ the associated norm. Recall \eqref{eq:def_U_circle} and define
\begin{equation}
\label{eq:def_vphi}
v_{\phi}:=\partial_x u_{\phi}=-A \sin(\cdot+\phi).
\end{equation}
We consider also the orthogonal projection $P_{\phi}^\circ$ on $\text{Span}(v_{\phi})$  and its complementary projection $P_{\phi}^\perp$, both defined for any $g\in L^2_\phi$ by
\begin{align}
P_{\phi}^\circ g&:= \dfrac{\langle g,v_{\phi}\rangle_{2,\phi}}{\Vert v_{\phi} \Vert_{2,\phi}} v_{\phi} =: \alpha_\phi^\circ(g) v_\phi\label{eq:def_proj_u'phi}\\
P_{\phi}^\perp g&:= g-P_{\phi}^\circ g.\label{eq:def_proj_stable}
\end{align}
We will also need the projection on $\text{Span}(u_{\phi})$ hence we define
\begin{equation}
\label{eq:def_proj_alpha_gamma}
\alpha_\phi^\gamma(g) =\dfrac{\langle g,u_{\phi}\rangle_{2,\phi}}{\Vert u_{\phi} \Vert_{2,\phi}}.
\end{equation}
\end{deff}

\begin{remark}\label{rem:op_phi} Without ambiguity and for a general $\phi$, we may write $\Vert\cdot\Vert_{\phi}$ instead of $\Vert \cdot\Vert_{2,\phi}$ to gain in clarity.
Note that by compactness of $S$, since $0<\inf_{[-A,A]} f' <\sup_{[-A,A]}f'<\infty$, the norms $\Vert \cdot \Vert_2$ and $\Vert \cdot\Vert_{2,\phi}$ are equivalent: there exists $C_0, \widetilde{C}_0>0$ (independent of $\phi$) such that for any $g\in L^2(S)$,
\begin{equation}\label{eq:cst_C0}
 \widetilde{C}_0 \Vert g \Vert_{2} \leq \sup_{\phi\in S} \Vert g \Vert_{2,\phi} \leq C_0 \Vert g \Vert_{2}.
\end{equation}
\end{remark}

\begin{prop}\label{prop:spect_L_phi} Let $\phi\in S$. The operator $\mathcal{L}_{\phi}$ defined in \eqref{eq:def_L_phi} is self-adjoint in $L^2_{\phi}$ and has three distinct eigenvalues, $-1$, $0$ and $ \gamma\in (-1, 0)$. If for $ \iota\in \left\lbrace -1, \gamma, 0\right\rbrace$, we denote by $ \mathcal{ E}_{ \iota}$ the eigenspace associated to the eigenvalue $ \iota$, one has that $ \mathcal{ E}_{ 0}= {\rm Ker} \mathcal{ L}_{ \phi}= {\rm Span} (v_{ \phi})$, $ \mathcal{ E}_{ \gamma}= {\rm Span}(u_{ \phi})$ and $ \mathcal{ E}_{ -1}= ({\rm Span}(u_{ \phi}, v_{ \phi}))^{ \perp}$. Moreover, $ \mathcal{ E}_{ 0}\perp \mathcal{ E}_{ \gamma}$. Furthermore, there exists $C_\mathcal{L}$, $C_P$ such that for any $\phi\in S$, $\mathcal{L}_{\phi}$ generates an analytic semigroup of contraction $\left(e^{t\mathcal{L}_\phi}\right)$ and for any $g\in L^2_{\phi}$, $t\geq 0$,
\begin{align}
\Vert e^{t\mathcal{L}_{\phi}} P_{\phi}^\perp  g \Vert_{2,\phi} &\leq e^{t\gamma}\Vert P_{\phi}^\perp g\Vert_{\phi},\label{eq:contraction_Ps_phi} \\
\Vert e^{t\mathcal{L}_\phi}g\Vert_{2}&\leq C_\mathcal{L} \Vert g \Vert_{2}, \label{eq:op_Lphi_borne1}\\
\Vert e^{t\mathcal{L}_\phi}P_{\phi}^\perp g \Vert_{2,\phi} &\leq C_P \Vert g \Vert_{2,\phi}. \label{eq:op_Lphi_borne2}
\end{align}
\end{prop}

Proposition \ref{prop:spect_L_phi} is proved in Section \ref{S:stat_sol_stab}.
A straightforward corollary of Proposition \ref{prop:spect_L_phi} is the following
\begin{cor}\label{cor:U_stable}
The manifold $\mathcal{U}$ is locally stable under the flow \eqref{eq:NFE_specific}: there exists $\varepsilon_0>0$ such that, for any $g\in L^2(S)$ satisfying $\text{dist}_{L^2}(g,\mathcal{U})\leq \varepsilon_0$, we have  $\lim_{t\to\infty}\text{dist}_{L^2}(\psi_t(g),\mathcal{U})=0$ where $\psi$ is defined in \eqref{eq:def_flow}. We denote by $B(\mathcal{U},\varepsilon_0):=\left\{g\in L^2(I), \text{dist}_{L^2}(g,\mathcal{U})\leq \varepsilon_0\right\}$.
\end{cor}

\subsubsection{Representation on the manifold}\label{S:projection}

Recall that we are interested in the behaviour of the process \eqref{eq:def_UN}, when the initial condition $U_N(0)$ to \eqref{eq:def_UN} is close to the manifold $\mathcal{U}$ introduced in \eqref{eq:def_U_circle}. We need a way to define a proper phase reduction of $U_N$ along $\mathcal{U}$. We have two ways to do so that we use in our results that are well explained in the recent work \cite{AdamsMacLaurin2022arxiv}, which takes the NFE as a good class of examples and motivation. The first one is via the \textit{variational phase}, defined in the following Proposition \ref{prop:phase_proj}:
 
\begin{prop}[Variational phase]\label{prop:phase_proj}
There exists $\varpi>0$ such that, for any $g\in L^2(S)$ satisfying $\textnormal{dist}_{L^2(S)}(g,\mathcal{U})\leq \varpi$, there exists a unique phase $\phi:=\textnormal{proj}_\mathcal{U}(g)\in S$ such that $P_{\phi}^\circ (g-u_{\phi})=0$ and the mapping $g\mapsto \textnormal{proj}_\mathcal{U}(g)$ is smooth.
\end{prop}

The second one is via the \textit{isochronal phase}, defined in the following Proposition \ref{prop:ex_isochron}. In a few words, as the manifold $\mathcal{U}$ is stable and attractive, a solution to the NFE from a neighborhood of $\mathcal{U}$ is attracted to $\mathcal{U}$ and converges to it. As $t\to\infty$, it identifies with one stationary solution of the manifold, we called it its isochron. 

\begin{prop}[Isochronal phase]\label{prop:ex_isochron} For any $g\in B(\mathcal{U},\varepsilon_0)$ (see Corollary \ref{cor:U_stable}), there exists a unique $\theta(g) \in S$ such that
\begin{equation}\label{eq:def_isochronal_phase}
\Vert \psi_t(g)-u_{\theta(g)}\Vert_{2} \xrightarrow[t\to\infty]{}0,
\end{equation}
where $\psi$ is defined in \eqref{eq:def_flow}. Such a map $\theta:B(\mathcal{U},\varepsilon_0)\to S$ is called the isochronal map of $\mathcal{U}$, and $\theta(g)$ is the isochronal phase of $g$. Moreover, it is three times continuously Fréchet differentiable (in fact $C^\infty$), and in particular for $u_\phi\in \mathcal{U}$, $h,l\in L^2(S)$, we have
\begin{equation}\label{eq:diff_theta}
D\theta(u_\phi)[h]= \dfrac{\langle v_\phi,h\rangle_\phi}{\Vert v_\phi\Vert_\phi}, \text{ and}
\end{equation}
\begin{multline}\label{eq:diff2_theta}
D^2\theta(u_\phi)[h,l] = \dfrac{1}{2A^2} \left(  \alpha_\phi^\circ(h)\beta_\phi(v_\phi,l) +\alpha_\phi^\circ(l)\beta_\phi(v_\phi,h)+\beta_{ \phi}(h,l)\right) \\
+ \frac{ 1+ \gamma}{ 2A^{ 2}(1- \gamma)}\left(  \alpha_\phi^\gamma(h)\beta_\phi(u_\phi,l) +\alpha_\phi^\gamma(l)\beta_\phi(u_\phi,h)\right) - \frac{ (2- \gamma)(1+ \gamma)}{ 2(1-\gamma)} \left( \alpha_\phi^\circ(h)\alpha_\phi^\circ(l)+\alpha_\phi^\gamma(h)\alpha_\phi^\gamma(l)\right),
\end{multline}
where $\alpha_\phi^\circ$ and $\alpha_\phi^\gamma$ are respectively defined in \eqref{eq:def_proj_u'phi} and \eqref{eq:def_proj_alpha_gamma}, and
\begin{equation}
\beta_\phi(h,l) :=\int_S f''(u_\phi(y))v_\phi(y) h(y)l(y)dy. \label{eq:def_beta_phi2}
\end{equation}
\end{prop}

Note that in particular, as $u_{\theta(g)}\in \mathcal{U}$ and $\mathcal{U}$ consists in stationary points, $\psi_t(u_{\theta(g)})=u_{\theta(g)}$. Propositions \ref{prop:phase_proj} and \ref{prop:ex_isochron} are proved in Section \ref{S:proof_proj}.

\subsubsection{Long time behavior}

The first result uses the variational phase to ensure that $(U_N(t))$ defined in \eqref{eq:def_UN} reaches a neighborhood of $\mathcal{U}$ in time of order $\log(N)$ and stays inside it for arbitrary polynomial times in $N$.

\begin{thm}
\label{thm:UN_close_U}
Suppose that $\rho\in B(\mathcal{U},\varepsilon_0)$ and
\begin{equation}
\label{eq:hyp_init}
\Vert U_N(0)-\rho\Vert_{2} \xrightarrow[N\to\infty]{}0.
\end{equation} 
Let $\alpha,\tau_f>0$. There exists some $C>0$ such that, defining for any $N\geq 1$, $T_0(N):=C\log(N)$, for any $\varepsilon>0$,
\begin{equation}\label{eq:thm_UNU}
\mathbf{P}\left(\sup_{t\in [T_0(N),N^{\alpha} \tau_f]}\textnormal{dist}_{L^2}\left( U_N(t),\mathcal{U}\right) \leq \varepsilon \right) \xrightarrow[N\to\infty]{}1.
\end{equation}
\end{thm}

\begin{remark}
\label{rem:speed_thm_closeU}
In fact, we show a more precise result than \eqref{eq:thm_UNU} that will be useful for the proof of Theorem \ref{thm:theta_fluc}: we prove that for any fixed $\eta\in (0,\frac{1}{4})$, we have with some constant $C>0$
$$\mathbf{P}\left(\sup_{t\in [T_0(N),N^{\alpha} \tau_f]}\textnormal{dist}_{L^2}\left( U_N(t),\mathcal{U}\right) \leq  CN^{\eta-1/2} \right) \xrightarrow[N\to\infty]{}1.
$$
\end{remark}

Theorem \ref{thm:UN_close_U} is proved in Section \ref{S:long_time}. The second main result of the paper is the analysis of the behavior of $U_N$ along $\mathcal{U}$ when $\alpha=1$.

\begin{thm}
\label{thm:theta_fluc}
Let $\rho\in B(\mathcal{U},\varepsilon_0)$. Suppose \ref{eq:hyp_init}. Let $\tau_f>0$. There exist a deterministic $\theta_0\in S$ and for every $N$ some $\tau_0(N)\propto \dfrac{\log(N)}{N}$ and a càdlàg process $\left(W_N(t)\right)_{t\in (\tau_0(N),\tau_f)}$ that converges weakly in $\mathbb{D}\left([0,\tau_f],S\right)$ towards a standard Brownian such that for every $\varepsilon>0$,
\begin{equation}
\label{eq:theta_fluc}
\lim_{N\to\infty}\mathbf{P}\left(\sup_{\tau\in (\tau_0(N),\tau_f)} \left\Vert U_N(N\tau) - u_{\theta_0 + \sigma W_N(\tau)}\right\Vert_{2}\leq \varepsilon\right)=1,
\end{equation}
where
\begin{equation}\label{eq:def_sigma}
\sigma:=\left( 2\pi \int_S \sin^2(x)f(A\cos(x))dx\right)^{\frac{1}{2}},
\end{equation}
with $A=A(\kappa)$ defined with Proposition \ref{prop:stat_sol_sig}.
\end{thm}
Theorem \ref{thm:theta_fluc} is proved in Section \ref{S:fluct}. We have run several simulations to illustrate our results, seen in Figure \ref{fig:wandering1}. We represent the evolution of the current $U_N(t,x)$ for $t\in [0,T_{\text{max}}]$ where the time is on the x-axis and spatial position on the y-axis. The different values taken are scaled with a color bar. We can see the \emph{wandering bumps} evolving in Figure \ref{subfig:wandering1}, whereas in Figure \ref{subfig:wandering3} the initialization is too far from the manifold and the system is no longer attracted to $\mathcal{U}$. 

\begin{figure}
\centering
\subfloat[The initialization in the vicinity of $\mathcal{U}$ leads to wandering bumps]{\includegraphics[width=0.85\textwidth]{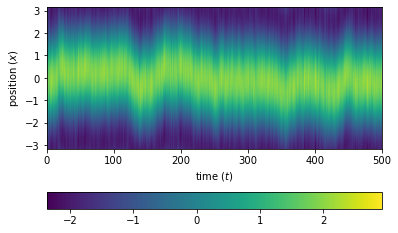}\label{subfig:wandering1}}
\\
\subfloat[The initialization far from the vicinity of $\mathcal{U}$ does not trigger the wandering bumps]{\includegraphics[width=0.85\textwidth]{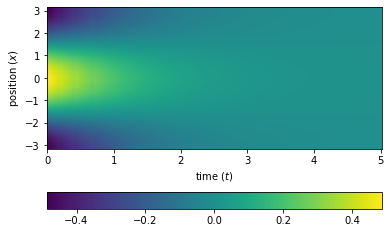}
\label{subfig:wandering3}}
\caption{Evolution of the voltage $U_N(t)(x)$}
\label{fig:wandering1}
\fnote{We chose $(\kappa,\varrho)=\left(\frac{1}{20},\frac{1}{2}\right)$ and run simulations of $N=500$ neurons following \eqref{eq:def_lambdaN_specific}. We represent the evolution of the current $U_N(t,x_i)$ obtained for two different simulations where we changed the initial profile $\rho$. In \ref{subfig:wandering1}, we start in a vicinity of $\mathcal{U}$ as we take for initialization $\rho(x)=A(\kappa)\cos(x) +\cos(2x)$, where $A(\kappa)$ solving \eqref{eq:A_self} for $f=f_{\kappa,\varrho}$ is found by a numerical root finding method, with a final time $T_{\text{max}}=500$ (of the same order that the size of the population). In \ref{subfig:wandering3}, we initialize the system with $\rho(x)=\frac{1}{4} A(\kappa)\cos(x)$. It is too far from the manifold $\mathcal{U}$ and we can see that the dynamics is attracted to $\mathcal{U}_A$ where $A$  is the smallest solution of \eqref{eq:A_self} (in Figure \ref{fig:fixed} it corresponds to the far left intersection of the black and blue lines) which is approximately 0, hence we only run the simulation with a final time $T_{\text{max}}=5$.}
\end{figure}

\subsection{Link with the literature}\label{S:litterature}


Hawkes processes have been introduced in \cite{HAWKES1971} to model earthquakes and have been thoroughly studied since, see e.g. \cite{bremaud1996stability}. The seminal work of \cite{delattre2016} has renewed the interest for large population of interacting Hawkes processes, which have proven to be particularly useful in a neuroscience context to model the mutually exciting properties of a population of neurons, see for instance \cite{Ditlevsen2017,CHEVALLIER20191}. 

In this respect, a common setting for the modelling of interacting neurons is the mean-field framework. For instance, in \cite{Baladron2012}, the authors describe the propagation of chaos in networks of Hodgkin-Huxley and FitzHugh-Nagumo neurons. Another popular model is the integrate-and-fire dynamics, first introduced in the seminal work of Lapique \cite{lapique1907}, and still studied mathematically, as e.g. in \cite{DeMasi2014, Delarue2015} and also \cite{Cormier2020}.


Several works have extended the mean-field framework to take into account the presence of a macroscopic spatial structure in the interaction, originally for diffusion models (see \cite{Touboul2014,Luon2016}), as well as for Hawkes processes (see \cite{Ditlevsen2017,CHEVALLIER20191}). The main difficulty with this extension is that we lose the exchangeability specific to homogeneous mean-field models as in \cite{Sznitman1989,delattre2016}. Concerning our present model, \cite{CHEVALLIER20191}  was the first to provide with a rigorous mesoscopic interpretation of the neural field equation \eqref{eq:NFE_gen} in terms of the limit of spatially extended Hawkes processes interacting through a mesoscopic spatial kernel. The recent work \cite{agathenerine2021multivariate} extend this result for Hawkes processes interacting on inhomogeneous random graphs.  Another possiblity to circumvent the exchangeability issue would have been to use replica mean-field models as \cite{davydov2022propagation} and describe the propagation of chaos for an infinite number of replicas. Note however that this description keeps the size $N$ of the population fixed, whereas we want to have $N\to\infty$.

Note also that the present model include interaction that may be negative: this reflects some inhibitive effect among neurons with opposite orientations. Modelling the inhibition present in the brain has been historically difficult. For Hawkes processes, a common approach is to allow the synaptic kernel $h$ in \eqref{eq:def_lambda_generic} to take negative values. This is however impossible for linear Hawkes processes as the intensity cannot be negative. To circumvent this, one has to choose a non-negative and nonlinear function $f$ to preserve the non-negativity of the intensity. A classic choice is to take $f(x)=\max\left(0,\mu+x\right)$ (see for instance \cite{bonnet_21} for estimation model or \cite{Costa2020,CCC2022} with $h$ in \eqref{eq:def_lambda_generic} signed and with compact support). One can also introduce inhibition through a signed multiplying factor (that may depend or not on the neuron), see for instance \cite{Duarte2016,Ditlevsen2017,Pfaffelhuber2022}. Some works have also parted the whole population into two subclasses of neurons, the excitatory ones and the inhibitory ones  \cite{raad2020stability,duval_lucon_pouzat2022}. In the latter, the inhibition is made thanks to a (small) multiplicative factor onto the intensity of the excitatory population. The present work is another contribution concerning models with inhibition, as it is present thanks to the cosine interaction kernel that takes negative values. This choice is essential to our dynamics as the balance between excitation and inhibition within the population of neurons allows to have a stable manifold of stationary solutions to \eqref{eq:NFE_specific}. \\ 


The analysis of mean-field interacting processes on long time scales has a significant history in the case of interacting diffusions, in particular in the case of phase oscillators as the Kuramoto model \cite{kuramoto75} (see \cite{giacomin_poquet2015} and references therein for a comprehensive review on the subject). The techniques used in the present work have some formal similarities to the ones used for diffusions, the main difference being that with Hawkes processes, the noise is Poissonnian (rather Brownian) and multiplicative (rather than additive). The so-called uniform propagation of chaos concerns situations where estimates such as \eqref{eq:chaos_generic} are uniform in time. Such estimates are commonly met in reversible situations (e.g. granular type media diffusions \cite{Bolley:2013}). See also the recent paper of \cite{Colombani2022}, where the authors studies a uniform propagation of chaos on the FitzHugh-Nagumo diffusive model. 

Let us comment on the analysis of the Kuramoto model as it presents some informal proximity with our model. One is here interested in the longtime behavior of the empirical measure $ \mu_{ N, t}:= \frac{ 1}{ N} \sum_{ i=1}^{ N} \delta_{ \theta_{ i, t}}$ of the system of interacting diffusions $(\theta_{ 1}, \ldots, \theta_{ N})$ solving the system of coupled SDEs 
$$ d\theta_{ i,t}= - \frac{ K}{ N} \sum_{ j=1}^{ N} \sin( \theta_{ i,t}- \theta_{ j,t})d t + dB_{ i, t},$$
with $(B_i)$ i.i.d. Brownian motions. Standard propagation of chaos techniques show that $ \mu_{N}$ converges weakly on a bounded time interval $[0, T]$ to the solution $ \mu_{ t}$ to the nonlinear Fokker-Planck (NFP) equation 
\begin{equation}
\label{eq:NFP}
\partial_t \mu_t\, =\, \frac{1}{2} \partial_{ \theta}^{ 2} \mu_t+K\partial_\theta \Big( \mu_t(\sin * \mu_t)\Big),
\end{equation}
(to compare with our microscopic current $U_{N,i}$ in \eqref{eq:def_UiN} converging towards $u_t$ solution to the NFE \eqref{eq:NFE_specific}). One can easily prove the existence of a phase transition for \eqref{eq:NFP}: when $K\leq 1$, $ \mu\equiv \frac{ 1}{ 2\pi}$ is the only (stable) stationary point of \eqref{eq:NFP} (subcritical case), whereas it coexists with a stable circle of synchronised profiles when $K>1$ (supercritical case). A series of papers have analysed the longtime behavior of the empirical measure $\mu_N$ of the Kuramoto model (and extensions) in both the subcritical and supercritical cases, the first one being \cite{bertini14}, followed by \cite{giacomin2012,lucon_poquet2017,Coppini2022,Delarue2021}. The main arguments of the mentioned papers lie in a careful analysis of two contradictory phenomena that arise on a long-time scale: the stability of the deterministic dynamics around stationary points (that forces $ \mu_{ N}$ to remain in a small neighborhood of these points) and the presence of noise in the microscopic system (which makes $ \mu_{ N}$ diffuse around these points). 

We are here in a similar situation to the supercritical case: the deterministic dynamics of the spatial profile $U_{ N}$ (given by \eqref{eq:def_UN}) has a stationary manifold $\mathcal{U}$ (defined in \eqref{eq:def_U_circle}) which possesses sufficient stability properties, see Corollary \ref{cor:U_stable}. The point of the analysis relies then on a time discretization and some careful control on the diffusive influence of noise that competes with the deterministic dynamics. In a previous work \cite{agathenerine_longtime_arxiv}, we have analysed in depth the case where \eqref{eq:def_ell_previous} has a unique solution, that would be comparable to the subcritical case of the Kuramoto model.


The first main result of the paper is to show that once $U_N(0)$ is close to the stationary manifold $\mathcal{U}$, it stays so for a long time, see Theorem \ref{thm:UN_close_U}. The next step is to find a way to describe the projection of the dynamics onto $\mathcal{U}$. A convenient tool for this is the use of isochronicity, we refer to \cite{guckenheimer1974} for a precise approach on the subject, and to  \cite{Giacomin2018} for their use of isochronicity to study the proximity between the noisy trajectory of interacting particles and the limit cycle in a finite dimensional setting. See also \cite{lucon-poquet2021} where the microscopic system is a diffusion and the large population limit admits a stable periodic solution: they show that the empirical measure stays close to the periodic solution with a random dephasing. The isochron map in this case helps to describe the dephasing as a Brownian motion with a constant drift. \\


Going back to Hawkes processes, several other works have already complemented the propagation of chaos result mentioned in \eqref{eq:chaos_generic} and studied finite approximations of the NFE, mostly at the level of fluctuations. Central Limit Theorems (CLT) have been obtained in \cite{delattre2016,Ditlevsen2017} for homogeneous mean-field Hawkes processes (when both time and $N$ go to infinity) or with age-dependence in \cite{Chevallier2017}. One should also mention the functional fluctuation result recently obtained in \cite{Heesen2021}, also in a pure mean-field setting. A result closer to our case with spatial extension is \cite{ChevallierOst2020}, where a functional CLT is obtained for the spatial profile $U_{ N}$ around its limit. Note here that all of these works provide approximation results of quantities such that $ \lambda_{ N}$ or $U_{ N}$ that are either valid on a bounded time interval $[0, T]$ or under strict growth condition on $T$ (see in particular the condition $ \frac{ T}{ N} \to 0$ for the CLT in \cite{Ditlevsen2017}), whereas we are here concerned with time-scales that grow polynomially with $N$. 

Another alternative to study large time behavior is to use a Brownian approximation of the dynamics of $U_{ N}$, see the initial work of \cite{Ditlevsen2017}. However this approximation is based on the comparison of the corresponding semigroups and is not uniform in time. Nevertheless, let us comment on this diffusive approximation in large population regime on bounded time intervals that can be found in both \cite{Ditlevsen2017,ChevallierOst2020}. A second order approximation of the NFE was proposed in \cite{ChevallierOst2020} with (adapted to the notations of the present article)
\begin{equation}
\label{eq:diff_approx_litt}
dU_N(t)=-U_N(t)dt + w\ast f(U_N(t))dt + C\int_S w(x,y) \frac{\sqrt{f(U_N(t)(y))}}{\sqrt{N}}W(dt,dy),
\end{equation}
where $W$ is a Gaussian white noise. This approximating diffusion process \eqref{eq:diff_approx_litt} is a noisy NFE, it can be seen as an intermediate modeling between the microscopic scale given by the Hawkes process and the macroscopic scale given by the NFE. In our framework with a cosine kernel, the infinitesimal increment of the noise in \eqref{eq:diff_approx_litt} can be expanded as
$$ C \cos(x) \int_S \cos(y) \frac{\sqrt{f(U_N(t)(y))}}{\sqrt{N}}W(dt,dy) + C \sin(x) \int_S \sin(y) \frac{\sqrt{f(U_N(t)(y))}}{\sqrt{N}}W(dt,dy).$$
To compare with our result, let us informally project the last quantity on $\text{Ker}(\mathcal{L}_0)$ introduced in Proposition \ref{prop:spect_L_phi}. The scalar product $\langle \cdot, v_0\rangle_{2,0}$ with $v_0=-A\sin(\cdot)$ gives that the cosine term becomes zero and the noise left is a random variable of the form
$$ -CA\int_S \sin^2 f'(A\cos) \int_S \sin (y) \frac{\sqrt{f(U_N(t)(y))}}{\sqrt{N}}W(dt,dy)=-C\int_S \sin (y) \frac{\sqrt{f(U_N(t)(y))}}{\sqrt{N}}W(dt,dy)$$
using \eqref{eq:A_kappa&Isin}. The infinitesimal noise that effectively drives the dynamics of \eqref{eq:diff_approx_litt} along $\mathcal{U}$ is then Gaussian with variance proportional to
$$\int_S \sin^2 (y) \frac{f(U_N(t)(y))}{N}dydt$$
which is exactly the variance found in \eqref{eq:def_sigma}, rescaled by $\frac{1}{N}$ and where $U_N(t)$ has been replaced by the limit $u_t$. This analogy remains informal, but shows that our results are compatible to the computations of \cite{Ditlevsen2017} and \cite{ChevallierOst2020}: one could see the present result as a rigorous justification that the approximation introduced by \cite{Ditlevsen2017,ChevallierOst2020} can be extended for polynomial times in $N$. 

Approximation between Hawkes and Brownian dynamics has also been studied in \cite{chevallierMT2021,erny2023annealed}, based on Koml\'os, Major and Tusn\'ady (KMT) coupling techniques (see \cite{ethier_kurtz1986}). Recently, Prodhomme \cite{Prodhomme_2023} used similar KMT coupling techniques applied to finite dimensional Markov chains and found Gaussian approximation to remain precise for very large periods of time. However these results are valid for $\mathbb{Z}^d$-valued continous-time Markov chains, it is unclear how they can be applied in our situation (with infinite dimension and space extension).  The proof we propose is direct and does not rely on such Brownian coupling. 

The question of Stochastic Neural Field Equations has also been considered directly from a macroscopic perspective at multiple times. It consists in considering the NFE \eqref{eq:NFE_gen} with an additive or multiplicative spatio-temporal noise, see for instance \cite{bressloff_webber_12_front,kruger_stannat_14_front}. Existence and uniqueness results have been obtained for various expressions of the noise, see \cite{faugeras_inglis_SNFE2015,inglis_mclaurin_16}. Let us mention in particular \cite{kilpatrick2013,MacLaurinBressloff2020,kilpatrick2022arxiv} who propose a heuristical derivation of the diffusion coefficient of the wandering bumps in a setting similar to ours (the ring model with $f$ the Heaviside function). See also \cite{maclaurin_2023} where the author studies the effect of the added noise on patterns such that traveling waves and oscillations thanks to the use of some projection of the dynamics, to obtain long time stability. Whereas all of the previous results are concerned with a macroscopic approach concerning stochastic perturbation of the NFE, we provide here a rigorous and microscopic interpretation of this phenomenon.

\subsection{Strategy of proof of the long time behavior} \label{S:strategy}
\subsubsection{About Theorem \ref{thm:UN_close_U}}

Section \ref{S:long_time} is devoted to prove the proximity result of Theorem \ref{thm:UN_close_U}. This in particular requires some spectral estimates on the operators $\mathcal{L}_\phi$ introduced in Definition \ref{def:ops_phi} and the stability of stationary solutions to \eqref{eq:NFE_specific}, results that are gathered in Section \ref{S:stationary_result} and proved in Section \ref{S:proof_stat}. The main lines of proof for Theorem \ref{thm:UN_close_U} are given in Section \ref{S:long_time}. The strategy of proof is sketched here, and follows the one used in a previous work \cite{agathenerine_longtime_arxiv}.\\

First we show in Proposition \ref{prop:control_initial} that one can find some initial time $T_0(N)\propto \log(N)$ for which $\text{dist}_{L^2}\left( U_N(T_0(N)),\mathcal{U}\right) \leq \dfrac{N^{2\eta}}{\sqrt{N}}$, with $0<\eta<\frac{1}{4}$. This essentially boils down to following the predominant deterministic dynamics of the NFE. Let $T_f(N)=N^\alpha$, we discretize the interval of interest $[T_0(N),T_f(N)]$ into $n_f$ intervals of same length $T$ denoted by $[T_i, T_{i+1}]$, $T$ chosen sufficiently large below. On each subinterval, we can decompose the dynamics of $U_N(t)$ in terms of, at first order, the linearized dynamics of \eqref{eq:NFE_specific} around any stationary solution, modulo some drift terms coming from the mean-field approximation, some noise term coming from the underlying Poisson measure, and some quadratic remaining error coming from the nonlinearity of $f$. It gives a semimartingale decomposition of $U_N(t)- u_{\text{proj}(U_N(T_i))}$ for $t\in [T_i,T_{i+1}]$, detailed in Section \ref{S:mild}. 

Provided one has some sufficent control on each of these terms in the semimartingale expansion on a bounded time interval, we do an iterative procedure  that works as follows: the point is to see that provided $U_N$ is initially close to $u_{\text{proj}(U_N(T_i))} \in \mathcal{U}$, it will remain close to it for a time interval of length $T$ for some sufficiently large deterministic $T>0$ so that the deterministic dynamics prevails upon the other contributions. 
The time horizon at which one can pursue this recursion is controlled by moment estimates on the noise in Proposition \ref{prop:control_noise}.

\subsubsection{About Theorem \ref{thm:theta_fluc}}

Section \ref{S:fluct} is devoted to prove the analysis of the behavior of $U_N$ along $\mathcal{U}$ seen in Theorem \ref{thm:theta_fluc}. We sketch here the strategy of proof. First we use the semimartingale decomposition of $U_N$ $$dU_N(t)=B_N(t)dt + dM_N(t)$$ (with $B_N$ some drift and $M_N$ a martingale defined in \eqref{eq:def_M_N}) and Itô formula to write the semimartingale decomposition of $\theta(U_N(t))$ on the interval $[T_0(N), N\tau_f]$. As in Theorem \ref{thm:UN_close_U}, one can show a careful control on each of the terms appearing in the semimartingale decomposition, as done in Section \ref{S:lemmas_pr}.
The difficulty here is to show rigorously that there is no macroscopic drift appearing on this time scale (this point is essentially due to the invariance by rotation of the whole problem). After rescaling the time by $N$, we identify the noise with a Brownian motion thanks to Aldous' tightness criterion and Lévy's characterization so that the result of Theorem \ref{thm:theta_fluc} follows.

\subsubsection{Extensions}

\paragraph{On the interaction kernel}\

Note that Theorem \ref{thm:theta_fluc} is of local nature: stability holds provided the initial condition $\rho$ is sufficiently close to $\mathcal{U}$. Following \cite{kilpatrick2013}, it would be possible to consider the more general interaction kernel
$$w(x,y)=\sum_{k=0}^n A_k\cos ( k (x-y)),$$
with more that one Fourier mode. The fixed point equation \eqref{eq:A_self} becomes a more complicated system of equations
\begin{equation}
\label{eq:system_Ak}
A_k=\int_S \cos(kx) f\left( \sum_{k=0}^n A_k \cos(kx) \right) dx.
\end{equation}
The exact number of solutions to \eqref{eq:system_Ak} remain unclear but if one can solve \eqref{eq:system_Ak} and show local stability of the solutions $u_\phi(x)=\sum_{k=0}^N A_k \cos(k(x+\phi))$, the same strategy should apply: we would obtain local stability provided one starts sufficiently close to these structures.

\paragraph{Oscillatory behavior}\

Note that $\mathcal{U}$ consists of stationary points. We claim that a similar strategy should apply also to situations where \eqref{eq:def_lambdaN_specific} admits generic oscillations, see \cite{Poquet_GB_L14} in a context of diffusion. We have in particular in mind the  framework proposed in \cite{Ditlevsen2017}: the authors study interacting Hawkes processes with Erlang memory kernel. The population is divided into classes, and the classes interact with a cycling feedback system, so that the large population limit is attracted to non-constant periodic orbits. It is reasonable to think that our techniques can be transposed to this situation, to show that the microscopic system is closed to the limit cycle under their hypotheses in large times and without using the approximating diffusion process.

\section{Stationary solutions (proofs)}\label{S:proof_stat}

Let us first define for any function $r\in L^2(S)$
\begin{equation}\label{eq:def_Iphi}
\mathcal{I}(r):=\int_S r(y)f '(u_{0}(y))dy,
\end{equation} 
where $u_0$ is defined in \eqref{eq:def_U_circle}. We start by giving a computation Lemma that will be useful in the whole paper.

\begin{lem}\label{lem:computation_A}
We have
$$\mathcal{I}(\sin^2)=1, \quad \mathcal{I}(\cos^2)=\mathcal{I}(1)-1 \quad \text{ and } \quad \mathcal{I}(\sin\cos)=0.$$
\end{lem}

\begin{proof} Recall that $u_0=A\cos$, as $A$ solves \eqref{eq:A_self} by integrating by parts we obtain
\begin{equation}\label{eq:A_kappa&Isin}
A= \int_S\cos (y) f \left( A  \cos(y) \right)dy=A  \int_S \sin^2(y) f '\left( u_{0}(y) \right)dy=A  \mathcal{I}(\sin^2),
\end{equation}
and as $A >0$ it implies $\mathcal{I}(\sin^2)=1$. By integrating by parts we also have
$$-A \mathcal{I}(\cos\sin)=\int_{-\pi}^{\pi} \sin(y)f(A\cos(y))dy.$$
Since $y\to \sin(y)f(A\cos(y))$ is odd, we obtain that $\mathcal{I}(\cos\sin)=0$. As $\cos^2= 1-\sin^2$ and $\mathcal{I}$ is linear, we have
$ \mathcal{I}(\cos^2)= \mathcal{I}(1)-\mathcal{I}(\sin^2)=\mathcal{I}(1)-1$.
\end{proof}

\subsection{Stability}
\label{S:stat_sol_stab}

Here we prove Proposition \ref{prop:spect_L_phi}. 

\begin{proof} Let $\phi\in S$. Let us first show that the operator $\mathcal{L}_\phi$ is indeed self-adjoint in $L^2_\phi$. Let $g_1,g_2\in L^2_{\phi}$, we have by Fubini's theorem and recalling Definition \ref{def:ops_phi}
\begin{align*}
\langle \mathcal{L}_{\phi}g_1,g_2\rangle_{\phi}&= -\int_S g_1 g_2 f '(u_{\phi}) + \int_S \left( \int_S \cos(x-y)f '(u_{\phi}(y))g_1(y)dy\right)g_2(x)f '(u_{\phi}(x))dx\\
&= -\int_S g_1 g_2 f '(u_{\phi}) + \int_S f '(u_{\phi}(y))g_1(y) \left( \int_S \cos(x-y)g_2(x)f '(u_{\phi}(x))dx\right)dy\\
&=\langle g_1,\mathcal{L}_{\phi} g_2\rangle_{\phi},
\end{align*}
hence $\mathcal{L}_{\phi}$ is self-adjoint in $L^2_{\phi}$. 

We focus now on its spectrum, we want to prove that it has three distinct eigenvalues, -1, 0 and $\gamma\in(-1,0)$. The following arguments follow the same procedure of the one that can be found in \cite{kilpatrick2013}. First note that $T_\phi$ is compact in $L^2_\phi$ (in fact, with finite range). Hence it has a discrete spectrum consisting of eigenvalues. Let $\lambda$ be an eigenvalue of $\mathcal{L}_{\phi}$ and $\psi$ an associated eigenvector, that is $\mathcal{L}_{\phi} \psi = \lambda \psi$ hence $(\lambda+1)\psi = T_{\phi} \psi$ with Definition \ref{def:ops_phi}.  As seen in Remark \ref{rem:invariance}, $\lambda$ does not depend on $\phi$ and if $\psi$ is an eigenvector for $\phi=0$, then $\psi(\cdot-\phi)$ is an eigenvector for $\phi$. Hence, in the following, we focus on the case $\phi=0$. We have
\begin{equation}
\label{eq:decomp_TO}
T_{0}\psi(x)=A_{0}(\psi)\cos(x) + B_{0}(\psi)\sin(x),
\end{equation}
with
\begin{equation}\label{eq:def_ABphi}
A_{0}(\psi):=  \int_S \cos(y)f '\left( u_{0}(y)\right)\psi(y)dy, \quad
B_{0}(\psi) := \int_S \sin(y)f '\left( u_{0}(y)\right)\psi(y)dy.
\end{equation}
The eigenvalue -1 is spanned by functions $\psi\in L^2$ such that $A_0(\psi)=B_0(\psi)=0$. Recall \eqref{eq:def_Iphi}, we have that, since $(\lambda+1)\psi=T_0\psi$,
\begin{align}\label{eq:lambaAphi}
(\lambda+1) A_{0}(\psi)&= \int_S \cos(y) (\lambda+1)\psi(y) f '(u_{0}(y))dy \notag\\
&=  \int_S \cos(y) \left(A_{0}(\psi)\cos(y)+B_{0}(\psi)\sin(y)\right) f '(u_{0}(y))dy \notag\\
&= A_{0}(\psi)\mathcal{I}(\cos^2) + B_{0}(\psi) \mathcal{I}(\sin\cos),
\end{align}
and similarly,
\begin{equation}\label{eq:lambaBphi}
(\lambda+1)B_0(\psi)= A_{0}(\psi) \mathcal{I}(\sin\cos) + B_{0}(\psi) \mathcal{I}(\sin^2).
\end{equation}
See Lemma \ref{lem:computation_A} for the computations of 
$\mathcal{I}(\cos^2)$, $\mathcal{I}(\sin^2)$ and $\mathcal{I}(\sin\cos)$. Putting these computations into \eqref{eq:lambaAphi} and \eqref{eq:lambaBphi} implies that $(\lambda,\psi)$ solves $\mathcal{L}_0\psi=\lambda\psi$ if and only if
$$
\left\{
\begin{array}{ll}
(\lambda +1) A_0(\psi) &= \left(\mathcal{I}(1)-1\right) A_0(\psi) \\
(\lambda +1) B_0(\psi) &=  B_0(\psi).
   \end{array}
   \right.
 $$
Recall that with no loss of generality, one can suppose that $ \psi$ is such that $(A_{ 0}(\psi), B_{ 0}(\psi))\neq (0,0)$. Then $( \lambda, \psi)$ solves the previous system if and only if, either $ \lambda=0$ with $A_{ 0}(\psi)= 0$ and $B_{ 0}(\psi)\neq 0$ (and hence we see from \eqref{eq:decomp_TO} that the eigenvalue $0$ is spanned by $\sin \propto v_{ 0}$) or $ \lambda= \gamma$ given by
\begin{equation}
\label{eq:def_lambda_stable}
\gamma:=\mathcal{I}(1)-2=\int_S f'(A\cos(x))dx-2,
\end{equation}
with $A_{ 0}(\psi)\neq 0$ and $ B_{ 0}(\psi)=0$, so that the eigenspace related to $\gamma$ is one-dimensional, spanned by $ \cos \propto u_{ 0}$. The fact that $ \left\langle u_{ \phi}\, ,\, v_{ \phi}\right\rangle_{ \phi}=0$ follows immediately from the fact that $ u_{ \phi}$ is even and $ v_{ \phi}$ is odd. The last eigenvalue $\lambda=-1$ is spanned by $\psi$ such that $A(\psi)=B(\psi)=0$.\\

To conclude the proof of Proposition \ref{prop:spect_L_phi}, it remains to prove the inequalities \eqref{eq:contraction_Ps_phi}, \eqref{eq:op_Lphi_borne1} and \eqref{eq:op_Lphi_borne2}. We come back to a general $\phi \in S$. By definition of the projection $P_{\phi}^\circ$ in \eqref{eq:def_proj_u'phi},  we have that $\mathcal{L}_\phi P_{\phi}^\circ= 0$. Moreover,  by definition of $P_{\phi}^\perp $ in \eqref{eq:def_proj_stable}, we have that for any $g\in L^2_{\phi}$, $P_{\phi}^\perp g$ belongs in the orthogonal of $\text{Ker}(\mathcal{L}_{\phi})$ in  $L^2_{\phi}$. Then $\mathcal{L}_\phi P_{\phi}^\perp =\mathcal{L}_\phi (Id-P_{\phi}^\circ)$ generates a contraction semigroup on $L^2(S)$ and \eqref{eq:contraction_Ps_phi} follows then from functional analysis (see e.g. Theorem 3.1 of \cite{Pazy1974}). For the two last inequalities, we use Remark \ref{rem:op_phi}. From the definition of the projection $P_{\phi}^\circ $ in \eqref{eq:def_proj_u'phi}, we have that
$$e^{t\mathcal{L}_\phi}P_{\phi}^\circ g = \dfrac{\langle g,v_{\phi}\rangle_{\phi}}{\Vert v_{\phi} \Vert_{\phi}}e^{t\mathcal{L}_\phi} v_{\phi}= \dfrac{\langle g,v_{\phi}\rangle_{\phi}}{\Vert v_{\phi} \Vert_{\phi}} v_{\phi}$$
as $v_\phi\in \text{Ker}(\mathcal{L}_\phi)$. We obtain then $\Vert e^{t\mathcal{L}_\phi}P_{\phi}^\circ g \Vert_\phi \leq \Vert g \Vert_\phi \Vert v_\phi \Vert_\phi$. From \eqref{eq:contraction_Ps_phi} we have $\Vert e^{t\mathcal{L}_\phi}P_{\phi}^\perp g \Vert_\phi\leq e^{\gamma t}\Vert P_{\phi}^\perp g\Vert_\phi \leq C_P \Vert g \Vert_\phi$ for some $C_P>0$, that is exactly \eqref{eq:op_Lphi_borne2}. As $\Vert e^{t\mathcal{L}_\phi}g\Vert_{2} \leq \Vert e^{t\mathcal{L}_\phi}P_{\phi}^\circ g\Vert_{2} + \Vert e^{t\mathcal{L}_\phi}P_{\phi}^\perp g\Vert_{2}$, \eqref{eq:op_Lphi_borne1} follows for the choice $C_\mathcal{L}=C_1 C_2 \max \left(\sup_{\phi \in S} \Vert v_{\phi} \Vert_{\phi}, C_P\right)$. 
\end{proof}

\subsection{Projections on the manifold}\label{S:proof_proj}

We prove that both the variational phase seen in Proposition \ref{prop:phase_proj} and isochronal phase seen in Proposition \ref{prop:ex_isochron} are well defined.

\begin{proof}[Proof of Proposition \ref{prop:phase_proj}] (similar to \cite{lucon_poquet2017}[Lemma 2.8])
Define for any $(g,\phi)\in L^2(S)\times S$:
$$F(g,\phi):=\int_S \left( g(x) - u_{\phi}(x)\right) v_{\phi}(x) f '(u_{\phi}(x))dx=\langle g-u_{\phi},v_{\phi} \rangle_{\phi}.$$
We have for any fixed $\phi_0$, $F(u_{\phi_0},\phi_0)=0$. Note that $F$ is smooth in both variables as it can be written
$F(g,\phi)=-A \int_S \left( g(x) - A \cos(x+\phi)(x)\right) \sin(x+\phi) f '(u_{\phi}(x))dx$.
Moreover, $\partial_\phi F(u_{\phi_0},\phi_0) = - \langle v_{\phi_0},v_{\phi_0}\rangle_{\phi_0}=-A ^2\mathcal{I}_{\phi_0}(\sin^2)$ with $\mathcal{I}_{\phi}(r):=\int_S r(y+\phi)f '(u_{\phi}(y))dy$. By invariance on the circle $\mathcal{I}_{\phi_0}(\sin^2)=\mathcal{I}(\sin^2)$ defined in \eqref{eq:def_Iphi} and Lemma \ref{lem:computation_A} implies then that $\partial_\phi F(u_{\phi_0},\phi_0) = -A ^2=-A(\kappa)^2\neq 0$ with Proposition \ref{prop:stat_sol_sig}. By the implicit function theorem, for any $\phi_0$ there exists a neighborhood $\mathcal{V}(u_{\phi_0})$ of $u_{\phi_0}$ such that the projection is well defined (i.e. for any $g\in \mathcal{V}(u_{\phi_0})$, there exists a unique $\phi$ such that $F(g,\phi)=0$ and $g\mapsto	\text{proj}_{\mathcal{U}}(g)$ is smooth). By compactness of $\mathcal{U}$, the existence of $\varpi$ and the result of Proposition \ref{prop:phase_proj} follow. The situation can be summarized by the following Figure \ref{fig:projection}.

\begin{figure}[h]
\begin{center}
\begin{tikzpicture}[scale=0.5]
\draw (0,0) circle (4) ;
\draw [dotted] (0,0) circle (3) ;
\draw [dotted] (0,0) circle (5) ;
\draw [dashed] (-5,0)--(-4,0) node[midway,above] {$\varpi$};
\draw (-30:4) node[below right]{$u_{\phi}=A\cos(\cdot+\phi)$} ;
\draw (-30:4) node {$\bullet$} ;
\draw (45:4) node[below]{$u_{\phi_0}$} ;
\draw (45:4) node {$\bullet$} ;
\draw (45:4.8) node[right]{$g$} ;
\draw (45:4.8) node {$\bullet$} ;
\draw [->] (45:4.8) -- (45:4.1);
\draw [->] (45:4) -- ++(-1,1) node[above]{$v_{\phi_0}$};
\draw (-90:4) node[below]{$\mathcal{U}$};
\draw (8,4) node[rectangle,fill=gray!20]{$\phi_0=\text{proj}(g)$};
\draw (8,3) node[rectangle,fill=gray!20]{$g-u_{\phi_0}\perp v_{\phi_0}$};
\end{tikzpicture}
\end{center}
\caption{Projection of $g\in L^2(S)$ on $\mathcal{U}$}\label{fig:projection}
\end{figure}
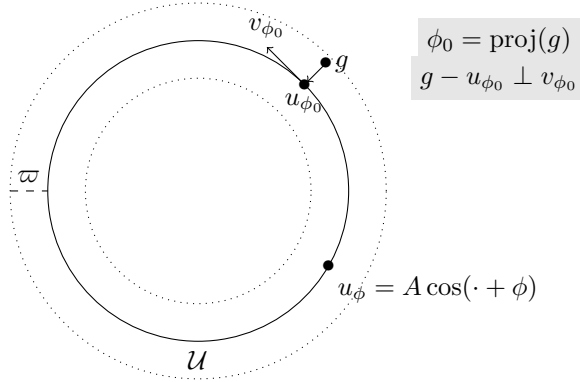
\end{proof}

\begin{proof}[Proof of Proposition \ref{prop:ex_isochron} ]
We reproduce the argument of \cite[Theorem 3.1]{AdamsMacLaurin2022arxiv} that establishes the existence and regularity of the isochron map in a more general context than here.

Let $g\in B(\mathcal{U},\varepsilon_0)$ and $\left(\epsilon_n\right)_n$ a sequence decreasing to 0. The first step is to prove that $\theta(g)$ satisfying \eqref{eq:def_isochronal_phase} exists. To do so, using the stability of $\mathcal{U}$ proved in Corollary \ref{cor:U_stable}, one can find an increasing sequence of times $(t_n)$ and a sequence of closed non-empty sets $\Phi_n\subset \mathcal{U}$ such that for all $n\in \mathbb{N}$ and $\theta\in \Phi_n$, $\Vert \psi_{t_n} (g) - u_\theta\Vert_2\leq C\epsilon_n$ for some constant $C>0$. It gives in particular that the diameter of $\Phi_n$ tends to zero as $n\to\infty$, hence the existence of an unique $\theta(g)$ such that $\cap_{n\in\mathbb{N}} \Phi_n=\{ u_{\theta(g)} \}$ by Cantor's Intersection Theorem. The second step is to prove the regularity of $\theta:B(\mathcal{U},\varepsilon_0)\to S$. As $\mathcal{U}$ is parameterized by $S$, we can define $\pi(u)$ for $u\in \mathcal{U}$ as the unique $\phi\in S$ such that $u=u_\phi$. As the flow $\psi$ is $\mathcal{C}^\infty$, the map $g\mapsto =\lim_{t\to\infty}\psi_t(g)$ is well defined and $\mathcal{C}^\infty$, and we have also $\lim_{t\to\infty}\psi_t(g)=u_{\theta(g)}$. Then $\theta(g)$ can be written as $\pi \left( \lim_{t\to\infty} \psi_t (g)\right)$, hence $g\mapsto \theta(g)$ is indeed $\mathcal{C}^\infty$.\\

We focus now on the derivatives of $g\mapsto \theta(g)$. Define $\Gamma: g\in B(\mathcal{U},\varepsilon_0)\mapsto \Gamma(g)=\lim_{t\to\infty}\Psi_t g =u_{\theta(g)}\in \mathcal{U}$. From Proposition \ref{prop:ex_isochron}, $\Gamma$ is smooth and is differentiable, and for $g,h\in L^2(S)$, $D\Gamma(g)[h]=u_{\theta(g)}'D\theta(g)[h]=v_{\theta(g)}D\theta(g)[h]\in L^2$. Applied for $g=u_\phi$ and taking the scalar product with $v_\phi$, one obtains
\begin{equation}\label{eq:diff_theta_aux}
\langle D\Gamma(u_\phi)[h],v_\phi\rangle=D\theta(u_\phi)[h] \Vert v_\phi \Vert^2.
\end{equation}

Let us focus on $ D\Psi_t g[h]$. Let $g_t$ be the solution of \eqref{eq:NFE_specific} with $g_0=g$, that is $g_t=\Psi_t(g)$, and $h_t$ the solution of \eqref{eq:NFE_specific} with $h_0=g+h$, that is $h_t=\Psi_t(g+h)$. Then
\begin{align*}
\partial_t(h_t-g_t)&=-(h_t-u_t)+\cos\ast\left( f(h_t)-f(g_t)\right)\\
&=-(h_t-g_t)+\cos\ast \left( f'(g_t)(h_t-g_t)\right)+ r_t
\end{align*}
with Taylor's formula and where $r_t:=\cos\ast\left( \left(h_t-g_t\right)^2 \int_0^1 (1-s)f''\left(g_t+s\left(h_t-g_t\right)\right)ds\right)=o(\Vert h \Vert)$.  We have then that $D\Psi_t(g)[h]=:w_t$ with 
\begin{equation}
\label{eq:Dtheta_aux}
\partial_t w_t=-w_t+\cos\ast \left( f'(\Psi_t g)(w_t)\right), \quad w_0=h.
\end{equation}
In particular for the choice $g=u_\phi$,  $D\Psi_t(u_\phi)[h]=e^{t\mathcal{L}_\phi}h$ where $\mathcal{L}_\phi$ is defined in \eqref{eq:def_L_phi}. Moreover we can write with the operators defined in Definition \ref{def:ops_phi}
$$e^{t\mathcal{L}_\phi}h=e^{t\mathcal{L}_\phi} \left( P_{\phi}^\circ h+ P_{\phi}^\perp h\right)= \dfrac{\langle h, v_\phi\rangle_\phi}{\Vert v_\phi\Vert_\phi} v_\phi + e^{t\mathcal{L}_\phi} P_{\phi}^\perp h.$$
From \eqref{eq:contraction_Ps_phi}, $\Vert e^{t\mathcal{L}_\phi} P_{\phi}^\perp h \Vert_\phi \leq e^{t\gamma}\Vert P_{\phi}^\perp h\Vert_\phi$ hence $\lim_{t\to\infty}e^{t\mathcal{L}_\phi}h= \dfrac{\langle h, v_\phi\rangle_\phi}{\Vert v_\phi\Vert_\phi} v_\phi$. As $\Gamma(u_\phi)=\lim_{t\to\infty}\Psi_tu_\phi=u_\phi$ and $\lim_{t\to\infty}D\Psi_t(u_\phi)[h]=\dfrac{\langle h, v_\phi\rangle_\phi}{\Vert v_\phi\Vert_\phi} v_\phi$, we obtain that
$$D\Gamma(u_\phi)[h]= D\left(\lim_{t\to\infty} \Psi_t u_\phi\right)[h]=\lim_{t\to\infty} D\Psi_t (u_\phi)[h]=\lim_{t\to\infty} e^{t\mathcal{L}_\phi}h=\dfrac{\langle h, v_\phi\rangle_\phi}{\Vert v_\phi\Vert_\phi} v_\phi,$$
which gives with \eqref{eq:diff_theta_aux} the result \eqref{eq:diff_theta}.

We focus now on $D^2\theta$. Recall $\Gamma$, for $g,h,l\in B(\mathcal{U},\varepsilon_0)$, $D^2\Gamma(g)[h,l]=-D\theta(g)[h]D\theta(g)[l]u_{\theta	(g)} + D^2\theta(g)[h,l]v_{\theta(g)}$. Applied for $g=u_\phi$, it gives with  \eqref{eq:diff_theta}
$$D^2\Gamma(u_\phi)[h,l]=-\dfrac{\langle v_\phi,h\rangle_\phi\langle v_\phi,l\rangle_\phi}{\Vert v_\phi\Vert_\phi^2}u_{\phi} + D^2\theta(u_\phi)[h,l]v_\phi.$$Taking the scalar product with $v_\phi$, as $\langle u_\phi,v_\phi\rangle_\phi=0$ we obtain
\begin{equation}
\label{eq:aux_diff2_theta}
D^2\theta(u_\phi)[h,l] = \dfrac{\langle D^2\Gamma(u_\phi)[h,l],v_\phi\rangle_\phi}{\Vert v_\phi\Vert_\phi^2}.
\end{equation}
Let us focus on $ D^2\Psi_t g[h,l]$. We have that $D\Psi_t(g)[h]=w_t$, recall that it solves \eqref{eq:Dtheta_aux}. Let $D\Psi_t(g+l)[h]:=\tilde{w}_t$, it solves
$$\partial_t \tilde{w}_t=-\tilde{w}_t+\cos\ast \left( f'(\Psi_t (g+l))\tilde{w}_t\right), \quad \tilde{w}_0=h.$$
As done before, we obtain that $\zeta_t:=\tilde{w}_t-w_t$ solves with $ \zeta_0=0$
\begin{align*}
\partial_t\zeta_t
&= - \zeta_t+ \cos \ast \left[ f'\left( \Psi_t(g+l)\right) (\zeta_t+w_t)- f'\left(\Psi_t g\right)w_t\right]\\
&=  - \zeta_t+ \cos \ast \left[ f'\left( \Psi_t(g+l)\right) \zeta_t\right] +\cos\ast \left[ \left(  f'\left( \Psi_t(g+l)\right) - f'\left(\Psi_t g\right)\right)w_t\right].
\end{align*}
From Taylor expansion in $l$,
\begin{align*}
f'\left( \Psi_t(g+l)\right)&=f'\left(\Psi_t(g)+D\Psi_t(g)[l] + \int_0^1 (1-s)D^2\Psi_t(g)[l]^2ds \right)\\
&= f'\left(\Psi_t(g)\right) +f''\left(\Psi_t(g)\right) D\Psi_t(g)[l] + o(\Vert l \Vert)
\end{align*}
hence
$$\cos \ast \left[ f'\left( \Psi_t(g+l)\right) \zeta_t\right]  = \cos \ast \left( f'\left(\Psi_t(g)\right)  \zeta_t\right)+ O(\Vert l \Vert),$$
and
\begin{align*}
\cos\ast \left[ \left(  f'\left( \Psi_t(g+l)\right) - f'\left(\Psi_t g\right)\right)w_t\right]&=\cos\ast\left(  f''(\Psi_tg)D\Psi_t g[l] w_t\right) + o(\Vert l \Vert)\\
&= \cos\ast\left(  f''(\Psi_tg)D\Psi_t g[l]D\Psi_t g[h]\right) +o(\Vert l \Vert).
\end{align*}
We obtain then after linearizing that $ D^2\Psi_t g[h,l]=\xi_t$ is solution of 
$$\partial_t \xi_t = -\xi_t + \cos\ast\left(f'(\Psi_t g) \xi_t \right) +\cos\ast\left(  f''(\Psi_tg)D\Psi_t g[l]D\Psi_t g[h]\right), \quad \xi_0=0.$$
In particular, for the choice $g=u_\phi$,
$$\partial_t \xi_t = \mathcal{L}_\phi\xi_t + \cos\ast\left[ f''(u_\phi) \left( e^{t\mathcal{L}_\phi}h\right) \left( e^{t\mathcal{L}_\phi}l\right) \right], \quad \xi_0=0,$$
hence it solves the mild equation
$$\xi_t= \int_0^t e^{(t-s)\mathcal{L}_\phi} \left( \cos \ast \left( f''(u_\phi) \left( e^{s\mathcal{L}_\phi}h\right) \left( e^{s\mathcal{L}_\phi}l\right) \right) \right) ds.$$ 
Recall \eqref{eq:aux_diff2_theta}, hence we focus now on $\langle \xi_t,v_\phi\rangle_\phi$.  From Proposition \ref{prop:spect_L_phi}, $\mathcal{L}_\phi$ is self-adjoint hence 
\begin{align*}
\langle \xi_t,v_\phi\rangle_\phi 
&=  \int_0^t \langle  \cos \ast \left( f''(u_\phi) \left( e^{s\mathcal{L}_\phi}h\right) \left( e^{s\mathcal{L}_\phi}l\right) \right),   e^{(t-s)\mathcal{L}_\phi} v_\phi\rangle_\phi ds\\
&=  \int_0^t \langle  \cos \ast \left( f''(u_\phi) \left( e^{s\mathcal{L}_\phi}h\right) \left( e^{s\mathcal{L}_\phi}l\right) \right),  v_\phi\rangle_\phi ds
\end{align*}
as $v_\phi \in \text{Ker} \mathcal{L}_\phi$. Recall \eqref{eq:def_proj_u'phi} and \eqref{eq:def_proj_alpha_gamma}. By the spectral decomposition of $ \mathcal{ L}_{ \phi}$ along its eigenvalues $0$, $ \gamma$ and $-1$, one has with Proposition \ref{prop:spect_L_phi}, for $s\geq0$,
\begin{align*}
e^{s \mathcal{ L}_{ \phi}}h&= \alpha_{\phi}^\circ(h) v_{ \phi} + e^{s \gamma} \alpha_{\phi}^\gamma(h) u_{ \phi} + e^{-s} \left(h- \alpha_{\phi}^\circ(h) v_{ \phi} -\alpha_{\phi}^\gamma(h) u_{ \phi}\right)\\
&= e^{-s }h + \alpha_{\phi}^\circ(h) (1- e^{-s}) v_{ \phi} + \alpha_{\phi}^\gamma(h) \left(e^{ s \gamma} - e^{ -s}\right) u_{ \phi},
\end{align*}
so that one obtains
\begin{align}
\left(e^{s \mathcal{ L}_{ \phi}}h\right)\left(e^{s \mathcal{ L}_{ \phi}}l\right)&= \alpha_{\phi}^\circ(h) \alpha_{\phi}^\circ(l) \left(1- e^{ -s}\right)^{ 2} v_{ \phi}^{ 2} + \alpha_{\phi}^\gamma(h) \alpha_{\phi}^\gamma(l) \left(e^{ s \gamma} - e^{ -s}\right)^{ 2} u_{ \phi}^{ 2} \label{aux:xi1}\\
&+e^{ -s}(1-e^{ -s}) \left\lbrace\alpha_{\phi}^\circ(h) l+ \alpha_{\phi}^\circ(l) h\right\rbrace v_{ \phi} \label{aux:xi2}\\
&+ e^{ -s} \left(e^{ s \gamma} - e^{ -s}\right) \left\lbrace \alpha_{\phi}^\gamma(h) l + \alpha_{\phi}^\gamma(l) h\right\rbrace u_{ \phi} \label{aux:xi3}\\
&+ \left(1- e^{ -s}\right) \left(e^{ s \gamma}- e^{ - s}\right) \left\lbrace \alpha_{\phi}^\circ(h) \alpha_{\phi}^\gamma(l)+ \alpha_{\phi}^\circ(l) \alpha_{\phi}^\gamma(h)\right\rbrace u_{ \phi} v_{ \phi} \label{aux:xi4}\\
&+ e^{ -2 s} hl. \label{aux:xi5}
\end{align}
We compute now $ \left\langle \xi_{ t}\, ,\, v_{ \phi}\right\rangle_{ \phi}$ based on the previous decomposition. Fix some generic test functions $h$ and $l$. Then
\begin{align*}
\left\langle \cos \ast \left(f^{ \prime\prime}(u_{ \phi}) hl\right)\, ,\, v_{ \phi}\right\rangle_{ \phi}&= \int_{ S} v_{ \phi}(x) f^{ \prime}(u_{ \phi}(x)) \int_{ S} \cos(x-y) f^{ \prime\prime}(u_{ \phi})(y) h(y)l(y) dy~  dx.
\end{align*}
Expanding the cosine within the convolution and noticing that $ \int_{ S} v_{ \phi}(x) f^{ \prime}(u_{ \phi}(x)) \cos(x+ \phi) dx=0$, we have with Lemma \ref{lem:computation_A}
\begin{align*}
\left\langle \cos \ast \left(f^{ \prime\prime}(u_{ \phi}) hl\right)\, ,\, v_{ \phi}\right\rangle_{ \phi}&= \left(\int_{ S} v_{ \phi}(x) f^{ \prime} \left(u_{ \phi}(x)\right) \sin(x+ \phi) dx \right) \int_{ S} \sin(y+ \phi) f^{ \prime\prime}( u_{ \phi}(y))h(y)l(y) dy,\\&=
-A \mathcal{I}(\sin^2) \int_{ S} \sin(y+ \phi) f^{ \prime\prime}( u_{ \phi}(y))h(y)l(y) dy= \int_{ S} f^{ \prime\prime}( u_{ \phi}(y)) v_{ \phi}(y)h(y)l(y) dy.
\end{align*}
If now we take $h=l= v_{ \phi}$ or $h=l= u_{ \phi}$, we see that the two terms of \eqref{aux:xi1} give a zero contribution to $ \left\langle \xi_{ t}\, ,\, v_{ \phi}\right\rangle_{ \varphi}$ as the function within the last integral is odd. Taking now $h=v_{ \phi}$ (resp. $h=u_{ \phi}$) for given $l$, we see that the generic term within \eqref{aux:xi2} (resp. \eqref{aux:xi3}) gives rise to
\begin{align*}
\left\langle \cos \ast \left(f^{ \prime\prime}(u_{ \phi}) l v_{ \phi}\right)\, ,\, v_{ \phi}\right\rangle_{ \phi}&= \int_{ S} f^{ \prime\prime}(u_{ \phi}) v_{ \phi}(y)^{ 2}l(y) dy,\\
\left\langle \cos \ast \left(f^{ \prime\prime}(u_{ \phi}) l u_{ \phi}\right)\, ,\, v_{ \phi}\right\rangle_{ \phi}&= \int_{ S} f^{ \prime\prime}(u_{ \phi}) v_{ \phi}(y) u_{ \phi}(y)l(y) dy.
\end{align*}
Applying finally the last expression for $l=v_{ \phi}$ gives for \eqref{aux:xi4}, by integration by parts
\begin{align*}
\left\langle \cos \ast \left(f^{ \prime\prime}(u_{ \phi}) u_{ \phi}v_{ \phi}\right)\, ,\, v_{ \phi}\right\rangle_{ \phi}&= \int_{ S} f^{ \prime\prime}(u_{ \phi}) v_{ \phi}(y)^{ 2} u_{ \phi}(y) dy= - \int_{ S} \frac{ {\rm d}}{ {\rm d}y} \left\lbrace u_{ \phi}(y) v_{ \phi}(y)\right\rbrace f^{ \prime}(u_{ \phi}(y)) dy,\\&=- \int_{ S} v_{ \phi}(y)^{ 2} f^{ \prime}(u_{ \phi}(y)) dy+ \int_{ S} u_{ \phi}(y)^{ 2} f^{ \prime}(u_{ \phi}(y)) dy= A^{ 2} \gamma,
\end{align*}
where we used \eqref{eq:def_lambda_stable}. Recall the definition of $ \beta_{\phi}$ in \eqref{eq:def_beta_phi2}, putting all these estimates together we obtain 
\begin{multline*}
\langle \xi_t,v_\phi\rangle_\phi =\int_0^t \left[  e^{-s}\left(1-e^{-s}\right)\left(  \alpha_\phi^\circ(h)\beta_\phi(v_\phi,l) +\alpha_\phi^\circ(l)\beta_\phi(v_\phi,h)\right) \right.\\ \left.+ e^{-s}\left(e^{s\gamma}-e^{-s}\right) \left(  \alpha_\phi^\gamma(h)\beta_\phi(u_\phi,l) +\alpha_\phi^\gamma(l)\beta_\phi(u_\phi,h)\right) \right.\\ \left. + \left(1-e^{-s}\right)\left(e^{s\gamma}-e^{-s}\right) A^2\gamma \left( \alpha_\phi^\circ(h)\alpha_\phi^\circ(l)+\alpha_\phi^\gamma(h)\alpha_\phi^\gamma(l)\right)+ e^{-2 s}\beta_\phi(h,l) \right] ds,
\end{multline*}
so that
\begin{multline*}
 \lim_{ t\to\infty}\left\langle \xi_{t}\, ,\, v_{ \phi}\right\rangle_{ \phi}= \frac{ 1}{ 2} \left(  \alpha_\phi^\circ(h)\beta_\phi(v_\phi,l) +\alpha_\phi^\circ(l)\beta_\phi(v_\phi,h)\right) + \frac{ 1+ \gamma}{ 2(1- \gamma)} \left(  \alpha_\phi^\gamma(h)\beta_\phi(u_\phi,l) +\alpha_\phi^\gamma(l)\beta_\phi(u_\phi,h)\right) \\
-  A^{ 2}\frac{ (2- \gamma)(1+ \gamma)}{ 2(1-\gamma)} \left( \alpha_\phi^\circ(h)\alpha_\phi^\circ(l)+\alpha_\phi^\gamma(h)\alpha_\phi^\gamma(l)\right) + \frac{ 1}{ 2} \beta_{ \phi}(h,l).
\end{multline*}
As $D^2\theta(u_\phi)[h,l]=\dfrac{1}{A^2} \lim_{t\to\infty} \langle \xi_t,v_\phi\rangle_\phi$, we obtain \eqref{eq:diff2_theta}.
\end{proof}

\section{Long time behavior (proofs)}
\label{S:long_time}

The aim of this section is to prove Theorem \ref{thm:UN_close_U}. 

\subsection{Main structure of the proof of Theorem  \ref{thm:UN_close_U}} \label{S:main_structure_close}

First, fix some constant $\eta$ such that 
\begin{equation}\label{eq:hyp_eta}
0<\eta<\frac{1}{4}.
\end{equation}
We also look for some $T>0$ that verifies
\begin{equation}\label{eq:choice_T}
C_PC_{\mathcal{L}} e^{T\gamma} \leq 1/4,
\end{equation}
where  $C_P$, $C_{\mathcal{L}} $ and $\gamma$ are introduced in Proposition  \ref{prop:spect_L_phi}. We first define the initial time $T_0(N)$ thanks to the following Proposition, whose proof is postponed to Section \ref{S:initialisation}.

\begin{prop}[Initialisation]\label{prop:control_initial}  In the framework of Theorem \ref{thm:UN_close_U}, there exists a deterministic phase $\theta_0\in S$, an event $B_N$ such that $\mathbf{P}(B_N)\xrightarrow[N\to\infty]{}1$ and a constant $C>0$ such that for all $\varepsilon >0$, for $N$ sufficiently large, on the event $B_N$, the projection $\psi=\psi_0^N=\text{proj}\left(U_N\left(C\log N\right)\right)$ is well defined and 
\begin{align}
&\Vert U_N(C\log N) - u_{\psi_0^N}\Vert_{2}\leq \dfrac{N^{2\eta}}{\sqrt{N}}, \label{eq:UN_init}\\
&\vert \psi_0^N - \theta_0 \vert \leq \varepsilon. \label{eq:init_prox_phase}
\end{align}
\end{prop}
We define $T_0(N)$ thanks to Proposition \ref{prop:control_initial} by $T_0(N)=C\log(N)$. Define the time discretisation of the interval $[T_0(N),N^\alpha\tau_f]$ into subintervalls of length $T$, $[T_n, T_{n+1}]$: define  $n_f=\inf\{n\in \mathbb{N}, ~ N^\alpha \tau_f \leq T_0(N)+n T\}$ and for $n=0,\cdots, n_f-1$, $T_n=T_0(N)+nT$. Let $T_f(N):=T_{n_f}$, by construction, $T_f(N)\geq N^\alpha\tau_f$. We prove in fact a more precise result that Theorem \ref{thm:UN_close_U} as stated in Remark \ref{rem:speed_thm_closeU}: we show that there exists some $C>0$ such that we have
\begin{equation}\label{eq:UN_U_aux}
\mathbf{P}\left(\sup_{t\in [T_0(N),T_f(N)]}\textnormal{dist}_{L^2}\left( U_N(t),\mathcal{U}\right) \leq  CN^{\eta-1/2} \right) \xrightarrow[N\to\infty]{}1.
\end{equation}
 
We focus on a process $(V_n(t))_{n\in \llbracket 1,n_f\rrbracket, t\in [0,T]}$ that iteratively compares $U_N$ and its projection on $\mathcal{U}$ at each step. We ensure it is correctly defined in the next part, then we give the main proof before the proof of some technical results we also need.

\paragraph{Discretization} In order to define the projection of $U_N(T_n)$ into $\mathcal{U}$, following Proposition \ref{prop:phase_proj}, we need to ensure that $\text{dist}_{L^2}\left(U_N(T_n),\mathcal{U}\right)\leq \varpi$. In order to do so, we introduce the stopping couple
\begin{equation}\label{eq:def_stop_couple}
(n_\tau,\tau):=\inf \left\{ (n,t)\in \llbracket 1,n_f\rrbracket \times [0,T]:  \text{dist}_{L^2}\left(U_N(T_{n-1}+t),\mathcal{U}\right) > \varpi\right\},
\end{equation}
where the infimum corresponds to the lexicographic order. We introduce then
\begin{equation}\label{eq:def_tau_n}
\tau_n:= \left\{
    \begin{array}{ll}
        T& \text{if } n<n_\tau \\
        \tau & \text{if } n\geq n_\tau.
    \end{array}
\right.
\end{equation}
The process we consider is then $\left(U_N(T_{n\wedge n_\tau-1}+t\wedge \tau_n)\right)_{n\in \llbracket 1,n_f\rrbracket, t\in [0,T]}$. The projection of this stopped process is well defined on the whole interval $[T_0(N),T_f(N)]$ by construction, so that we can now define rigorously the random phases $\phi_{n-1}$ for $n=1,\cdots,n_f$ by
\begin{equation}\label{eq:def_phi_n}
\phi_{n-1}:=\text{proj}(U_N(T_{n\wedge n_\tau-1}).
\end{equation}
 The object of interest is then the process $V_n(t)$ of $L^2(S)$ defined for $n=1,\cdots,n_f$ and $t\in [0,T]$ by
\begin{equation}\label{eq:def_Vnt}
V_n(t):=U_N(T_{n\wedge n_\tau-1}+t\wedge \tau_n)-u_{\phi_{n-1}},
\end{equation}
as \eqref{eq:UN_U_aux} translates then into

\begin{prop}\label{prop:control_vn}
There exists an event $\Omega_{N}$ with $\mathbf{P}(\Omega_N)\xrightarrow[N\to\infty]{}1$ such that on $\Omega_N$,
\begin{equation}\label{eq:control_vn}
\sup_{1\leq n \leq n_f} \sup_{t\in[0,T]} \Vert V_n(t)\Vert_{2}  =O\left(\dfrac{N^{2\eta}}{\sqrt{N}}\right),
\end{equation}
where the error is uniform on $\Omega_N$.
\end{prop}

Here are the steps of the proof of Proposition \ref{prop:control_vn}.

\paragraph{Step 1 -}
We show that the process $(V_n(t))_{n\in \llbracket 1,n_f\rrbracket, t\in [0,T]}$ satisfies the mild equation
\begin{equation}\label{eq:mild_form_vn}
V_n(t)=e^{(t\wedge \tau_n)\mathcal{L}_{\phi_{n-1}}}V_{n}(0) + \int_{0}^{t\wedge \tau_n} e^{(t\wedge \tau_n-s)\mathcal{L}_{\phi_{n-1}}}R_n(s)ds+ \zeta_{n}(t\wedge \tau_n)
\end{equation}
where
\begin{equation}\label{eq:def_zeta_n}
\zeta_{n}(t):= \int_{0}^{t} e^{(t-s)\mathcal{L}_{\phi_{n-1}}}dM_N(s),
\end{equation}
and 
\begin{multline}\label{eq:def_R_n}
R_{n}(t)= \cos\ast\left(y\mapsto V_n(t)(y)^2 \int_0^1 f ''\left( u_{\phi_{n-1}}(y)+rV_n(t)(y)\right)(1-r)dr \right)
\\+ \left(\sum_{i,j=1}^N \dfrac{2\pi\cos(x_i-x_j)}{N} f (U_{N,j}(t-))\mathbf{1}_{B_{N,i}}- \cos\ast f (U_{N}(t))\right),
\end{multline}
where the notation $\ast$ stands for the convolution  $f*g(x)=\int_{-\pi}^\pi f(x-y)g(y)dy$. The rigorous meaning of \eqref{eq:mild_form_vn} is given in Proposition \ref{prop:mild_tildeU}, postponed to Section \ref{S:mild}.

\paragraph{Step 2 -} We show a control of several terms of \eqref{eq:mild_form_vn} with the following Proposition, whose proof is postponed to Section \ref{S:noise}.

\begin{prop}[Noise perturbation]\label{prop:control_noise} Define the event 
\begin{equation}
A_N:=\left\{\sup_{1\leq n \leq n_f} \sup_{t\in[0,T]} \left\Vert \zeta_{n}(t) \right\Vert_{2}  \leq \dfrac{N^\eta}{\sqrt{N}}\right\}.\label{eq:def_event_AN}
\end{equation}
In the framework of Theorem \ref{thm:UN_close_U}, $\mathbf{P}(A_N)\xrightarrow[N\to\infty]{}1$.
\end{prop}

Now let $\Omega_N:=A_N\cup B_N$ (recall $B_N$ from Proposition \ref{prop:control_initial}) , we have $\mathbf{P}(\Omega_N)\xrightarrow[N\to\infty]{}1$ with Propositions \ref{prop:control_noise} and \ref{prop:control_initial}. For the rest of the proof, we place ourselves now on this event $\Omega_N$.

\paragraph{Step 3 -} Based on Steps 1 and 2 above, it remains to prove \eqref{eq:control_vn}. We proceed by induction. We know (as $\Omega_N\subset B_N$) that $\Vert V_1(0)\Vert \leq N^{2\eta-1/2}$. Suppose that $\Vert V_n(0)\Vert_{2} \leq N^{2\eta-1/2}$ for some $n\geq 1$. From the mild formulation satisfied by $(V_n(t))$ seen in  \eqref{eq:mild_form_vn} we get
$$\Vert V_n(t) \Vert_{2}= \left\Vert e^{(t\wedge \tau_n)\mathcal{L}_{\phi_{n-1}}}V_{n}(0)\right\Vert_{2} + \left\Vert \int_{0}^{t\wedge \tau_n} e^{(t\wedge \tau_n-s)\mathcal{L}_{\phi_{n-1}}}R_n(s)ds\right\Vert_{2}+ \left\Vert \zeta_{n}(t\wedge \tau_n)\right\Vert_{2}.$$

Recall \eqref{eq:def_phi_n} and Proposition \ref{prop:phase_proj}, by definition of the phase projection, $P_{\phi_{n-1},0}\left(U_N(T_{n\wedge n_\tau-1})-u_{\phi_{n-1}}\right)=0$ hence $V_n(0)=U_N(T_{n\wedge n_\tau-1})-u_{\phi_{n-1}}= P_{\phi_{n-1},s} V_n(0)$. Proposition \ref{prop:spect_L_phi} and more especially \eqref{eq:contraction_Ps_phi} give then, with the induction hypothesis
$$\Vert e^{(t\wedge \tau_n)\mathcal{L}_{\phi_{n-1}}} V_n(0)\Vert_{\phi_{n-1}} \leq e^{(t\wedge \tau_n)\gamma}\Vert V_n(0)\Vert_{\phi_{n-1}}\leq C_{0} e^{(t\wedge \tau_n)\gamma}N^{2\eta-\frac{1}{2}}$$
where $C_0$ is introduced in \eqref{eq:cst_C0}. From Proposition \ref{prop:spect_L_phi}, we have
$$\left\Vert \int_{0}^{t\wedge \tau_n} e^{(t\wedge \tau_n-s)\mathcal{L}_{\phi_{n-1}}}R_n(s)ds\right\Vert_{2}\leq TC_\mathcal{L}  \sup_{0\leq s\leq T} \Vert R_n(s)\Vert_{2}.$$ By definition of $A_N$, $\sup_{1\leq n \leq n_f} \sup_{t\in[0,T]} \left\Vert \zeta_{n}(t) \right\Vert_2 \leq N^{\eta-1/2}$ as we are on $\Omega_N$. We obtain then, for any $t\in [0,T]$
\begin{equation}\label{eq:aux_mild_gen}
\Vert V_n(t) \Vert_{2}\leq C_0e^{(t\wedge \tau_n)\gamma}N^{2\eta-\frac{1}{2}} + TC_\mathcal{L} \sup_{0\leq s \leq T}\left\Vert R_n(s)\right\Vert_{2}+ N^{\eta-1/2}.
\end{equation}

For any $t\in [0,T]$, recalling \eqref{eq:def_R_n}, 
\begin{align}
\sup_{0\leq s \leq t} \Vert R_{n}(s)\Vert_{2} &\leq \sup_{0\leq s \leq t} \left\Vert \cos\ast\left(y\mapsto V_n(s)(y)^2 \int_0^1 f ''\left( u_{\phi_{n-1}}(y)+rV_n(s)(y)\right)(1-r)dr \right)\right\Vert_{2}\notag\\
&+\sup_{0\leq s \leq t}  \left\Vert\sum_{i,j=1}^N \dfrac{2\pi\cos(x_i-x_j)}{N} f (U_{N,j}(s-))\mathbf{1}_{B_{N,i}}- \cos\ast f (U_{N}(s))\right\Vert_{2} = (A) + (B). \label{eq:induction_defAB}
\end{align}
Using Young's inequality $\Vert u \ast v \Vert_2 \leq \Vert u \Vert_1 \Vert v \Vert_2$ and the boundedness of $f''$, we have
$$(A) \leq \sup_{0\leq s \leq t} \left( \Vert \cos \Vert_2 \int_S \left\vert V_n(s)(y)^2 \int_0^1 f ''\left( u_{\phi_{n-1}}(y)+rV_n(s)(y)\right)(1-r)dr \right\vert dy\right) \leq C\sup_{0\leq s \leq t} \Vert V_n(s) \Vert_{2}^2$$
for some positive $C$. For the second term $(B)$ of \eqref{eq:induction_defAB}, we introduce
\begin{align}
\Upsilon_{1,i,s} &= \dfrac{2\pi}{N}\sum_{j=1}^N \cos(x_i-x_j) \left( f(U_{N,j}(s-))-f(U_{N,j}(s)) \right)\notag\\
\Upsilon_{2,i,s} &= \dfrac{2\pi}{N}\sum_{j=1}^N \cos(x_i-x_j)f(U_{N,j}(s))-\int_S \cos(x_i-y)f(U_N(s)(y))dy\notag\\
\Upsilon_{3,i,s}(x) &=\int_S\left(  \cos(x_i-y)-\cos(x-y)\right)f(U_N(s)(y))dy, \quad x\in S. \label{eq:def_Upsilon_i}
\end{align}
From the Lipschitz continuity of $f$ and the fact that $Z_{N,1},\cdots,Z_{N,N}$ do not jump simultaneously, $\vert \Upsilon_{1,i,s} \vert\leq \dfrac{C}{N}$ hence $\left\Vert \sum_{i=1}^N\Upsilon_{1,i,s}\mathbf{1}_{B_{N,i}} \right\Vert_{2}^2=O\left(\dfrac{1}{N^2}\right)$. As $\mathbf{1}_{B_{N,i}}\mathbf{1}_{B_{N,j}}\equiv 0$ for $i\neq j$, for any $0\leq s \leq t$ we have
$$\left\Vert \sum_{i=1}^N \Upsilon_{2,i,s}\mathbf{1}_{B_{N,i}} \right\Vert_{2}^2= \dfrac{2\pi}{N} \sum_{i=1}^N \left(\sum_{j=1}^N \int_{B_{N,j}} \left( \cos(x_i-x_j)-\cos(x_i-y)\right) f(U_N(s)(y))dy \right)^2.
$$
As $f$ is bounded (by 1) and $\cos$ is 1-Lipschitz continuous, we obtain 
\begin{equation}\label{eq:norme_Upsilon_2}
\left\Vert \sum_{i=1}^N \Upsilon_{2,i,s}\mathbf{1}_{B_{N,i}} \right\Vert_{2}^2\leq \dfrac{2\pi}{N} \sum_{i=1}^N \left(\sum_{j=1}^N \int_{B_{N,j}} \vert x_j-y\vert dy \right)^2\leq \dfrac{8\pi^5}{N^2}.
\end{equation}
Similarly,
\begin{align}\label{eq:norme_Upsilon_3}
\left\Vert \sum_{i=1}^N \Upsilon_{3,i,s}\mathbf{1}_{B_{N,i}} \right\Vert_{2}^2&= \int_S \sum_{i=1}^N \Upsilon_{3,i,s}(x)^2\mathbf{1}_{B_{N,i}}(x)dx\notag\\
&= \sum_{i=1}^N \int_{B_{N,i}} \left(\int_S\left(  \cos(x_i-y)-\cos(x-y)\right)f(U_N(s)(y))dy \right)^2dx \notag\\
&\leq  \sum_{i=1}^N \int_{B_{N,i}} \left(\int_S \vert x_i-x\vert dy \right)^2dx \leq  \dfrac{8\pi^5}{N^2}.
\end{align}
Hence we have for some positive $C_{R,1}$
\begin{equation}\label{eq:maj_RN}
\sup_{0\leq s \leq t} \Vert R_{n}(s)\Vert_{2} \leq C_{R,1} \left( \sup_{0\leq s \leq t} \Vert V_n(s) \Vert_{2}^2 + \dfrac{1}{N}\right).
\end{equation}
Define then $t^*$ as
\begin{equation}\label{def:time_t*}
t^*:= \inf\left\{ t\in [0,T]: \Vert V_n(t) \Vert_{2} \geq  2C_0\dfrac{N^{2\eta}}{\sqrt{N}}\right\}.
\end{equation}
Note that with no loss of generality, one can assume that $C_0>1$. Since by assumption $\Vert V_n(0)\Vert_2\leq \dfrac{N^{2\eta}}{\sqrt{N}}<C_0 \dfrac{N^{2\eta}}{\sqrt{N}}$, we have $\Vert V_n(t)\Vert_2\leq 2C_0 \dfrac{N^{2\eta}}{\sqrt{N}}$ at least for $t<t_1$ where $t_1$ is the first jump among $\left(Z_{N,1},\cdots,Z_{N,N}\right)$. Hence $t^*>0$. If $t\leq t^*$, $\sup_{0\leq s \leq t} \Vert R_{n}(s)\Vert_{2} \leq C_{R,2} N^{4\eta-1}$ (as $\eta>0$, $N^{-1}\ll N^{4\eta-1}$). Coming back to \eqref{eq:aux_mild_gen}, we obtain that (for some positive constant $C_R$)
\begin{equation}\label{eq:aux_mild_gen2}
\Vert V_n(t) \Vert_{2}\leq C_0e^{(t\wedge \tau_n)\gamma}N^{2\eta-\frac{1}{2}} + TC_R  N^{4\eta-1}+ N^{\eta-1/2}.
\end{equation}
Since $0<\eta<\dfrac{1}{4}$, $ N^{4\eta-1}\ll N^{2\eta-1/2}$ hence for $N$ large enough $TC_R  N^{4\eta-1}+ N^{\eta-1/2}\leq C_0 N^{2\eta-1/2}$ thus as $\gamma<0$, $t^*=T$. By construction of the stopping time $\tau_n$ in \eqref{eq:def_tau_n}, we have then that $\tau_n=T$, hence 
\begin{equation}\label{eq:ite_vn}
\sup_{0\leq t\leq T} \Vert V_n(t)\Vert_{2} \leq 2C_0N^{2\eta-1/2}
.\end{equation}
To conclude the induction, we need to show that $\Vert V_{n+1}(0)\Vert_{2}\leq N^{2\eta-1/2}$. By definition \eqref{eq:def_Vnt} and as $\tau_n=T$, $V_{n+1}(0)= U_N(T_n)-u_{\phi_n}$ and $V_n(T)=U_N(T_n)-u_{\phi_{n-1}}$ hence $V_{n+1}(0)=V_n(T)+u_{\phi_{n-1}}-u_{\phi_n}$. Moreover, as $V_{n+1}(0)=P_{\phi_n}^\perp V_{n+1}(0)$ since by definition $V_{n+1}(0)\in \text{Ker}\left( \mathcal{L}_{\phi_n}\right)^\perp$ (recall Proposition \ref{prop:spect_L_phi}), we obtain 
\begin{align}
V_{n+1}(0) &= P_{\phi_n}^\perp \left(V_n(T)+u_{\phi_{n-1}}-u_{\phi_n}\right)\notag\\
&=  \left(P_{\phi_n}^\perp -P_{\phi_{n-1}}^\perp \right)V_n(T)+P_{\phi_{n-1}}^\perp V_n(T) +P_{\phi_n}^\perp \left( u_{\phi_{n-1}}-u_{\phi_n}\right).\label{eq:ite_vn_+1}
\end{align}
We are going to control each term of \eqref{eq:ite_vn_+1}. First, using the smoothness of the phase projection from Proposition \ref{prop:phase_proj}, 
\begin{align}
\vert \phi_{n-1}-\phi_n \vert &= \vert \text{proj}\left( U_N\left(T_{(n-1)\wedge n_\tau -1}\right)\right) - \text{proj}\left( U_N\left(T_{n-\wedge n_\tau -1}\right)\right)\vert\notag\\
&\leq C_\text{proj} \left\Vert  U_N\left(T_{(n-1)\wedge n_\tau -1}\right) - U_N\left(T_{n-\wedge n_\tau -1} \right)\right\Vert_{2}\notag\\
&\leq C_\text{proj} \left\Vert  V_{n-1}(0) - V_{n-1}(T)\right\Vert_{2}\leq C N^{2\eta-1/2},\label{eq:control_phi}
\end{align}
using \eqref{eq:ite_vn}. Recall \eqref{eq:def_U_circle} and \eqref{eq:def_vphi}, we have for any $x\in S$
\begin{align*}
u_{\phi_{n-1}}(x)-u_{\phi_n}(x) &= A \cos\left(x+\phi_{n-1}\right) -A \cos\left(x+\phi_{n}\right)\\
&= -2A \sin\left( \phi_{n-1}-\phi_n\right) \sin\left( x+\phi_n + \dfrac{ \phi_{n-1}-\phi_n}{2}\right)\\
&= 2  \sin\left( \phi_{n-1}-\phi_n\right) \left(\cos \left(  \dfrac{ \phi_{n-1}-\phi_n}{2}\right) v_{\phi_n}(x)-\sin \left(  \dfrac{ \phi_{n-1}-\phi_n}{2}\right) u_{\phi_n}(x) \right)
\end{align*}
thus, as $P_{\phi_n}^\perp v_{\phi_n}=0$,
$$ P_{\phi_n}^\perp \left(u_{\phi_{n-1}}-u_{\phi_n}\right) = -2 \sin\left( \phi_{n-1}-\phi_n\right))\sin \left(  \dfrac{ \phi_{n-1}-\phi_n}{2}\right) P_{\phi_n}^\perp u_{\phi_n}.$$
As $u_{\phi_n}$ is bounded and $\sin$ is Lipschitz continuous, we obtain with \eqref{eq:control_phi} a control of the third term of \eqref{eq:ite_vn_+1} 
\begin{equation}\label{eq:control_P_u_phi}
\Vert P_{\phi_n}^\perp \left( u_{\phi_{n-1}}-u_{\phi_n}\right)\Vert_{2}\leq C \left(\phi_{n-1}-\phi_n\right)^2 =  O(N^{4\eta-1}).
\end{equation}
Similarly, recall \eqref{eq:def_proj_stable}, $\phi\mapsto P_{\phi}^\perp $ is smooth, hence for some $C>0$
\begin{equation}\label{eq:control_P_v}
\left\Vert  \left(P_{\phi_n}^\perp -P_{\phi_{n-1}}^\perp \right)V_n(T)\right\Vert_{2}\leq C \vert \phi_{n-1}-\phi_n \vert \Vert V_n(T)\Vert= O(N^{4\eta-1}).
\end{equation}
Combining \eqref{eq:control_P_u_phi} and \eqref{eq:control_P_v} in \eqref{eq:ite_vn_+1}, using \eqref{eq:ite_vn} at time $t=T$ and recalling Proposition \ref{prop:spect_L_phi}, we obtain for $N$ large enough
$$\Vert V_{n+1}(0)\Vert_{2}\leq \Vert P_{\phi_{n-1}}^\perp V_n(T) \Vert_{2} + O(N^{4\eta-1})\leq 2C_PC_0 e^{T\gamma} N^{2\eta-1/2}+ O(N^{4\eta-1}).
$$
From the choice of $T$ satisfying \eqref{eq:choice_T}, the fact that $\Vert V_{n+1}(0)\Vert_{2}\leq N^{2\eta-1/2}$ follows and the recursion is concluded, so that Theorem \ref{thm:UN_close_U} follows.

\subsection{About the mild formulation} \label{S:mild}

\textbf{Step 1} of Section \ref{S:main_structure_close} is a direct consequence of the following proposition.

\begin{prop}\label{prop:mild_tildeU} Fix $\phi\in S$ and $0<t_a<t_b$. Recall the definition of $U_N$ in \eqref{eq:def_UN}, and define, for any $t\in [t_a,t_b]$, 
\begin{equation}\label{eq:def_tildeU}
\widetilde{U}_{N,\phi}(t)=U_N(t)-u_{\phi}.
\end{equation}
The process $\left(\widetilde{U}_{N,\phi}(t)\right)_{t\in[t_a,t_b]}$ satisfies  the following semimartingale decomposition in $D([t_a,t_b],L^2(S))$, written in a mild form:  for any $t_a\leq t\leq t_b$
\begin{equation}\label{eq:mild_form_tildeU}
\widetilde{U}_{N,\phi}(t)=e^{(t-t_a)\mathcal{L}_{\phi}}\widetilde{U}_{N,\phi}(t_a) + \int_{t_a}^{t} e^{(t-s)\mathcal{L}_{\phi}}r_{N,\phi}(s)ds+\int_{t_a}^{t} e^{(t-s)\mathcal{L}_{\phi}} dM_N(s),
\end{equation}
with
\begin{equation}\label{eq:def_M_N}
M_N(t)= \sum_{i=1}^N \sum_{j=1}^N \dfrac{2\pi\cos(x_i-x_i)}{N} \left( Z_{N,j}(t) - \int_0^{t}\lambda_{N,j}(s)ds\right) \mathbf{1}_{B_{N,i}}
\end{equation}
and
\begin{multline}\label{eq:def_rN}
r_{N,\phi}(t)= \cos\ast\left(y\mapsto \widetilde{U}_{N,\phi}(t)(y)^2 \int_0^1 f ''\left( u_{\phi}(y)+r\widetilde{U}_{N,\phi}(t)(y)\right)(1-r)dr \right)
\\+ \left(\sum_{i,j=1}^N \dfrac{2\pi\cos(x_i-x_j)}{N} f (U_{N,j}(t-))\mathbf{1}_{B_{N,i}}- \cos\ast f (U_{N}(t))\right).
\end{multline}
\end{prop}

\begin{proof}[Proof of Proposition \ref{prop:mild_tildeU}]
From \eqref{eq:def_UiN}, we obtain that $U_N$ verifies
\begin{equation}
\label{eq:dUN}
dU_N(t)=-U_N(t)dt+\sum_{i,j=1}^N \dfrac{2\pi\cos(x_i-x_j)}{N}dZ_{N,j}(t)\mathbf{1}_{B_{N,i}}.
\end{equation}
The centered noise  $M_N$ defined in \eqref{eq:def_M_N} verifies $$dM_{N}(t):= \sum_{i=1}^N \sum_{j=1}^N \dfrac{2\pi\cos(x_i-x_j)}{N} \left( dZ_{N,j}(t) - f (U_{N,j}(t-))dt\right) \mathbf{1}_{B_{N,i}},$$ and is a martingale in $L^2(S)$. Thus recalling that $u_{\phi}$ solves \eqref{eq:NFE_stat_gen} and by inserting the terms $\displaystyle \sum_{i=1}^N \sum_{j=1}^N \dfrac{2\pi\cos(x_i-x_j)}{N}f (U_{N,j}(t-))dt\mathbf{1}_{ B_{N,i}}$ and $u_{\phi}$, we obtain 
$$d\widetilde{U}_{N,\phi}(t)=-\widetilde{U}_{N,\phi}(t)dt + dM_N(t) + \left(\sum_{i,j=1}^N \dfrac{2\pi\cos(x_i-x_j)}{N} f (U_{N,j}(t-))\mathbf{1}_{B_{N,i}}-\int_{-\pi}^\pi \cos(\cdot-y)f (u_{\phi}(y))dy\right)dt.$$
A Taylor's expansion gives that for any $y\in S$,
$$f (U_N(t)(y))-f (u_{\phi}(y))=f '(u_{\phi}(y))\widetilde{U}_{N,\phi}(t)(y)+\int_0^1 f ''\left( u_{\phi}(y)+r\widetilde{U}_{N,\phi}(t)(y)\right)(1-r)dr \widetilde{U}_{N,\phi}(t)(y)^2,$$
hence identifying the operator $\mathcal{L}_{\phi}$ defined in  \eqref{eq:def_L_phi} we have
\begin{multline*}
d\widetilde{U}_{N,\phi}(t)= \mathcal{L}_{\phi}\widetilde{U}_{N,\phi}(t)dt + dM_N(t) + \int_{-\pi}^\pi \cos(\cdot - y)\int_0^1 f ''\left( u_{\phi}(y)+r\widetilde{U}_{N,\phi}(t)(y)\right)(1-r)dr \widetilde{U}_N(t)(y)^2dydt\\
 + \left(\sum_{i,j=1}^N \dfrac{2\pi\cos(x_i-x_j)}{N} f (U_{N,j}(t-))\mathbf{1}_{B_{N,i}}-\int_{-\pi}^\pi \cos(\cdot-y) f (U_{N}(t)(y))\right)dt,
\end{multline*}
and recognizing $r_{N,\phi}$ defined in \eqref{eq:def_rN} we have
\begin{equation}\label{eq:def_dUtildeN_mild}
d\widetilde{U}_{N,\phi}(t) = \mathcal{L}_{\phi}\widetilde{U}_{N,\phi}(t)dt + r_{N,\phi}(t)dt + dM_N(t).
\end{equation} 
Then the mild formulation  \eqref{eq:mild_form_tildeU} is a direct consequence of Lemma 3.2 of \cite{Zhu2017}: the unique strong solution to \eqref{eq:def_dUtildeN_mild} is indeed given by \eqref{eq:mild_form_tildeU}.
\end{proof}

\subsection{About the initialisation}\label{S:initialisation}

We prove here Proposition \ref{prop:control_initial}, that we use to define the initial time $T_0(N)$ and in the second part of \textbf{Step 2} of Section \ref{S:main_structure_close}. 

\begin{proof}[Proof of Proposition \ref{prop:control_initial}]
To prove Proposition \ref{prop:control_initial}, we proceed in several steps, as done in  \cite{lucon_poquet2017}[Proposition 2.9].
\begin{itemize}
\item[\textbf{Step a.}] We rely on the convergence in finite time of $U_N$ to its large population limit, that is $u_t$ solving \eqref{eq:NFE_specific} with initial condition $\rho$. From the deterministic behavior of $u_t$ and the stability of $\mathcal{U} $, $U_N$ approaches $\mathcal{U}$ in a $2\varepsilon_0$-neighborhood; and this takes a time interval of order $\vert\log\varepsilon_0\vert$. 
\item[\textbf{Step b.}]  We rely on the stability of $\mathcal{U}$ and the control ofn the noise to show that, from a $2\varepsilon_0$-neighborhood, $U_N$ approaches $\mathcal{U}$ in a $N^{2\eta-1/2}$-neighborhood; and this takes a time interval of order $\log N$. 
\item[\textbf{Step c.}]  We ensure that $U_N$ stays at distance $N^{2\eta-1/2}$ from $\mathcal{U}$ at time $T_0(N)$.
\end{itemize}

\paragraph{Step a.} We focus first on $\psi_t(\rho)$, solution to \eqref{eq:NFE_specific}  with initial condition $\rho\in B(\mathcal{U},\varepsilon_0)$. Thanks to Corollary \ref{cor:U_stable}, we have that it converges as $t\to\infty$ towards some $u_{\theta_0}\in \mathcal{U}$. Thus, there exists a time $s_1\geq 0$ such that $\Vert u_{s_1} - u_{\theta_0}\Vert_{2}\leq \varepsilon_0$, and this time is of order $\dfrac{1}{\gamma}\log \varepsilon_0$.  We focus then on the random profile $U_N$. We use a mild formulation similar to the one used in Proposition \ref{prop:mild_tildeU}: one can obtain, with $u_t$ solving \eqref{eq:NFE_specific}
\begin{multline*}
d\left( U_N(t)-u_t\right)=-\left( U_N(t)-u_t\right)dt + dM_N(t) \\+ \left( \sum_{i,j=1}^N \dfrac{2\pi \cos(x_i-x_j)}{N}f\left( U_{N,j}(t-)\right)\mathbf{1}_{B_{N,i}}-\int_{-\pi}^\pi \cos(\cdot-y)f(u_t(y))dy\right)dt,
\end{multline*}
where $M_N$ is defined in \eqref{eq:def_M_N}. We have then for any $t\geq 0$
$$
U_N(t)-u_t=e^{-t}\left(U_N(0)-\rho \right) + \int_0^t e^{-(t-s)}dM_N(s) + \int_0^t e^{-(t-s)}r_N(s)ds
$$
with
\begin{multline}\label{eq:def_r_N}
r_N(s):=\sum_i \mathbf{1}_{B_{N,i}} \sum_j\dfrac{2\pi \cos(x_i-x_j)}{N}\left( f\left( U_{N,j}(s-)\right) - f\left( U_{N,j}(s)\right) \right) \\+ 
  \sum_i \mathbf{1}_{B_{N,i}}\left( \sum_j\dfrac{2\pi \cos(x_i-x_j)}{N}f\left( U_{N,j}(s)\right) - \int_{-\pi}^{\pi} \cos(x_i-y)f(U_N(s)(y))dy\right) \\ +  \sum_i \mathbf{1}_{B_{N,i}}\int_{-\pi}^\pi \left(\cos(x_i-y)-\cos(\cdot-y)\right)f(U_N(s)(y))dy + \int_{-\pi}^\pi \cos(\cdot-y)\left( f(U_N(s)(y)) - f(u_s(y))\right)dy\\
= \sum_{i=1}^N \mathbf{1}_{B_{N,i}} \left( \Upsilon_{1,i,s} + \Upsilon_{2,i,s}+ \Upsilon_{3,i,s}\right) + \Upsilon_{4,s}.
\end{multline}
As done for $\Upsilon_{1,i,s}$, $\Upsilon_{2,i,s}$ and $\Upsilon_{3,i,s}$ \eqref{eq:def_Upsilon_i} in Proposition \ref{prop:control_vn}, we have for some $C>0$
$$\left\Vert \sum_{i=1}^N \mathbf{1}_{B_{N,i}} \left( \Upsilon_{1,i,s} + \Upsilon_{2,i,s}+ \Upsilon_{3,i,s}\right)\right\Vert_{2}^2 \leq \dfrac{C}{N^2}.$$
Moreover an immediate computation gives, as $f$ is Lipschitz continuous
$$ \left\Vert \Upsilon_{4,s}\right\Vert_{2} \leq C \Vert U_N(s)-u_s\Vert_{2}.$$
Then we have for any $t\in [0,s_1]$ with $\zeta_N(s):=\int_0^s e^{-(s-u)}dM_N(u)$,
\begin{equation}\label{eq:mild_s1_f}
\Vert U_N(t)-u_t \Vert_{2}\leq \Vert U_N(0)-\rho \Vert_{2} + \Vert \zeta_N(t)\Vert_{2} + \dfrac{C}{N}+\int_0^t e^{-(t-s)} \Vert U_N(s)-u_s \Vert_{2} ds.
\end{equation}
Take $N$ sufficiently large so that $\Vert U_N(0)-\rho\Vert_{2}\leq \dfrac{\varepsilon_0}{2}$. We place ourselves on the event
\begin{equation}\label{eq:def_event_C_N}
C_N:=\left\{ \sup_{t\in [0,s_1]} \Vert \zeta_N(t)\Vert_{2} \leq N^{\eta-1/2}\right\}.
\end{equation}
As done in Proposition \ref{prop:control_noise}, $\mathbf{P}(C_N)\xrightarrow[N\to\infty]{}1$. Going back to \eqref{eq:mild_s1_f}, we have on $C_N$
$$\Vert U_N(t)-u_t \Vert_{2}\leq \dfrac{\varepsilon_0}{2}+ N^{\eta-1/2}+ \dfrac{C}{N}+\int_0^t e^{-(t-s)} \Vert U_N(s)-u_s \Vert_{2} ds.$$
We deduce with Gr\"{o}nwall lemma that for $N$ large enough, $\Vert U_N(s_1)-u_{s_1}\Vert_{2}\leq \varepsilon_0$ on $C_N$, which means that $\Vert U_N(s_1)-u_{\theta_0}\Vert_{2}\leq 2\varepsilon_0$ hence $\text{dist}\left(U_N(s_1),\mathcal{U}\right)\leq 2\varepsilon_0$. Choosing $\varepsilon_0$ small enough so that $2\varepsilon_0<\varpi$ (recall Proposition \ref{prop:phase_proj}), we can define $\psi_0^1=\text{proj}\left(U_N(s_1)\right)$ and $\vert \psi_0^1 - \theta_0 \vert\leq C\varepsilon_0$.

\paragraph{Step b.} Since we know that $\text{dist}_{L^2}\left( U_N(s_1),\mathcal{U}\right)\leq 2\varepsilon_0$ with increasing probability as $N\to\infty$, we show that $U_N$ approaches $\mathcal{U}$ up to a distance $N^{2\eta-1/2}$ doing a similar iteration as in Proposition \ref{prop:control_vn}.
Define the sequence $(h_n)$ such that $h_1=2\varepsilon_0$ and $h_{n+1}=h_n/2$, and let $\widetilde{n}_f:=\inf\left\{n\geq 1, h_n\leq N^{2\eta-1/2}\right\}$. Note that such $\widetilde{n}_f$ is of order $O(\log N)$. Fix $\widetilde{T}$ satisfying
\begin{equation}\label{eq:choice_T_tilde}
 C_PC_0 e^{\widetilde{T}\gamma} \leq 1/4,
\end{equation}
and define then for any $n\in \llbracket 1,\widetilde{n}_f\rrbracket$ the times $\widetilde{T}_n=s_1+(n-1) \widetilde{T}$. As in \eqref{eq:def_stop_couple} and \eqref{eq:def_tau_n}, define 
\begin{equation}\label{eq:def_stop_couple_tilde}
(\widetilde{n}_\tau,\widetilde{\tau}):=\inf \left\{ (n,t)\in \llbracket 1,\widetilde{n}_f\rrbracket \times [0,\widetilde{T}]: \text{dist}_{L^2}\left(U_N(\widetilde{T}_{n-1}+t),\mathcal{U}\right) > \varpi\right\},
\end{equation}
and 
\begin{equation}\label{eq:def_tau_n_tilde}
\widetilde{\tau}_n:= \left\{
    \begin{array}{ll}
        \widetilde{T}& \text{if } n<\widetilde{n}_\tau \\
        \widetilde{\tau} & \text{if } n\geq \widetilde{n}_\tau.
    \end{array}
\right.
\end{equation}
The process we consider is then $\left(U_N\left(\widetilde{T}_{n\wedge n_{\widetilde{\tau}}-1}+t\wedge \widetilde{\tau}_n\right)\right)_{n\in \llbracket 1,\widetilde{n}_f\rrbracket, t\in [0,\widetilde{T}]}$, which is exactly $(U_N(t))_{t\in [s_1,\widetilde{T}_{\widetilde{n}_f}]}$ 
unless the process has been stopped. The projection of this stopped process is well defined on the whole interval, so that we can now define rigorously the random phases $\widetilde{\phi}_{n-1}$ for $n=1,\cdots,\widetilde{n}_f$ by
\begin{equation}\label{eq:def_phi_n_tilde}
\widetilde{\phi}_{n-1}:=\text{proj}(U_N(\widetilde{T}_{n\wedge \widetilde{n}_{\tau}-1}).
\end{equation}
The object of interest is then the process $\widetilde{V}_n(t)$ of $L^2(S)$ defined for $n=1,\cdots,\widetilde{n}_f$ and $t\in [0,\widetilde{T}]$ by
\begin{equation}\label{eq:def_Vnt_tilde}
\widetilde{V}_n(t):=U_N(\widetilde{T}_{n\wedge \widetilde{n}_{\tau}-1}+t\wedge \widetilde{\tau}_n)-u_{\widetilde{\phi}_{n-1}}.
\end{equation}
It satisfies the mild equation 
\begin{equation}\label{eq:mild_form_vn_tilde}
\widetilde{V}_n(t)=e^{(t\wedge {\widetilde{\tau}}_n)\mathcal{L}_{{\widetilde{\phi}_{n-1}}}}\widetilde{V}_{n}(0) + \int_{0}^{t\wedge \widetilde{\tau}_n} e^{(t\wedge \widetilde{\tau}_n-s)\mathcal{L}_{\widetilde{\phi}_{n-1}}}\widetilde{R}_n(s)ds+ \widetilde{\zeta}_{n}(t\wedge \widetilde{\tau}_n)
\end{equation}
where
\begin{equation}\label{eq:def_zeta_n_tilde}
\widetilde{\zeta}_{n}(t):= \int_{0}^{t} e^{(t-s)\mathcal{L}_{\widetilde{\phi}_{n-1}}}dM_N(s),
\end{equation}
and 
\begin{multline}\label{eq:def_R_n_tilde}
\widetilde{R}_{n}(t)= \cos\ast\left(y\mapsto \widetilde{V}_n(t)(y)^2 \int_0^1 f ''\left( u_{\widetilde{\phi}_{n-1}}(y)+r\widetilde{V}_n(t)(y)\right)(1-r)dr \right)
\\+ \left(\sum_{i,j=1}^N \dfrac{2\pi\cos(x_i-x_j)}{N} f (U_{N,j}(t-))\mathbf{1}_{B_{N,i}}- \cos\ast f (U_{N}(t))\right).
\end{multline}
Define the event 
\begin{equation}\label{eq:def_event_D_N}
B_N:=C_N\bigcap\left\{ \sup_{n\in \llbracket 1,\widetilde{n}_f\rrbracket} \sup_{t\in [0,\widetilde{T}]} \Vert \widetilde{\zeta}_n(t)\Vert_{2} \leq N^{\eta-1/2}\right\}.
\end{equation}
As done in Proposition \ref{prop:control_noise}, $\mathbf{P}(B_N)\to 1$ and from now on we work under $B_N$. We want to show by induction that on $B_N$, for all $n\in \llbracket 1 , \tilde{n}_f\rrbracket$, $\widetilde{V}_n(0)\leq h_n$. The first step of the proof ensures that on $C_N$, $\widetilde{V}_1(0)\leq h_1$. Assume for some $n<\widetilde{n}_f$, $\widetilde{V}_n(0)\leq h_n$. From the mild formulation \eqref{eq:mild_form_vn_tilde} we obtain (as done  in \eqref{eq:aux_mild_gen2})
\begin{equation}\label{eq:aux_mild_gen_tilde}
\Vert \widetilde{V}_n(t) \Vert_{2}\leq C_0e^{(t\wedge \widetilde{\tau}_n)\gamma}h_n +  \widetilde{T}C_\mathcal{L}  \sup_{0\leq s\leq \widetilde{T}} \Vert \widetilde{R}_n(s)\Vert_{2}+ N^{\eta-1/2}.
\end{equation}
Define then $t^*$ as
\begin{equation}\label{def:time_t*_tilde}
\widetilde{t*}:= \inf\left\{ t\in [0,\widetilde{T}]: \Vert \widetilde{V}_n(t) \Vert_{2} \geq  2C_0h_n\right\}.
\end{equation}
We have $\widetilde{t^*}>0$, and if $t\leq \widetilde{t^*}$, $\sup_{0\leq s \leq t} \Vert \widetilde{R}_{n}(s)\Vert_{2} \leq C_{R_2}(h_n^2+N^{-1})$, as done in \eqref{eq:maj_RN}. Coming back to \eqref{eq:aux_mild_gen_tilde}, we obtain that (for some positive constant $C_{\widetilde{R}}$)
\begin{equation}\label{eq:aux_mild_gen2_tilde}
\Vert \widetilde{V}_n(t) \Vert_{2}\leq C_0e^{(t\wedge \tau_n)\gamma}h_n + TC_{\widetilde{R}}(h_n^2 + N^{-1})+ N^{\eta-1/2}.
\end{equation}
Since $n<\widetilde{n}_f$, $2\varepsilon_0\geq h_n>N^{2\eta-1/2}$ hence for $N$ large enough, $N^{\eta-1/2}$, $N^{-1}$ are negligible with respect to $h_n$, same for $h_n^2$ thus $\widetilde{t^*}\geq \widetilde{T}$. To conclude the induction, we need to show that $\Vert \widetilde{V}_{n+1}(0)\Vert\leq h_{n+1}=\frac{h_n}{2}$. As shown in \eqref{eq:ite_vn_+1}, $\widetilde{V}_{n+1}(0) =  \left(P_{\widetilde{\phi}_n}^\perp -P_{\widetilde{\phi}_{n-1}}^\perp \right)\widetilde{V}_n(\widetilde{T})+P_{\widetilde{\phi}_{n-1}}^\perp \widetilde{V}_n(\widetilde{T}) +P_{\widetilde{\phi}_n}^\perp \left( u_{\widetilde{\phi}_{n-1}}-u_{\widetilde{\phi}_n}\right)$.
 From the similar controls \eqref{eq:control_P_u_phi} and \eqref{eq:control_P_v} and using \eqref{eq:aux_mild_gen2_tilde} for $t=\widetilde{T}$, we have for $N$ large enough,
$$\Vert \widetilde{V}_{n+1}(0)\Vert_{2}\leq \Vert P_{\widetilde{\phi}_{n-1}}^\perp \widetilde{V}_n(\widetilde{T}) \Vert_{2} + O(h_n^2)\leq 2 C_PC_0 e^{\widetilde{T}\gamma}h_n+ O(h_n^2).
$$
Recall \eqref{eq:choice_T_tilde} and $\gamma<0$, the fact that $\Vert \widetilde{V}_{n+1}(0)\Vert_{2}\leq h_{n+1}$ follows then and the iteration is concluded. Thus, we have constructed a time $s_2=s_1+(\widetilde{n}_f-1)\widetilde{T}$ such that, on $B_N$ for $N$ large enough, setting $\psi_0^2:=\text{proj}\left(U_N(s_2)\right)$, we have $\Vert U_N(s_2)-u_{\psi_0^2}\Vert_2 \leq N^{2\eta-1/2}$ and $\vert\psi_0^2-\psi_0^1\vert\leq C\varepsilon_0$, which gives $\vert \psi_0^2-\theta_0\vert \leq C'\varepsilon_0$ sor some $C'>0$. 

\paragraph{Step c.} So far, we have constructed a time $s_2=C\left(\vert \log \varepsilon_0\vert + \log N\right)$ for which we have $\text{dist}_{L^2}\left(U_N(s_2),\mathcal{U}\right)\leq N^{2\eta-1/2}$. We want some $s_3=\tilde{C}\log N \geq s_2$, $\tilde{C}=C+1$, independent of $\varepsilon_0$ such that with $\psi_0^3:=\text{proj}(U_N(s_3))$, $\left\Vert U_N(s_3) - u_{\psi_0^3}\right\Vert \leq N^{2\eta-1/2}$. For this, it suffices to decompose the dynamics on $[s_2,s_3]$ in a same way as before in both Steps 1 and 2. This induces a drift $\vert \psi_0^3-\psi_0^2\vert \leq C N^{2\eta-1/2}\log(N)\leq \varepsilon_0$ for $N$ large enough. This last step concludes the proof with $T_0(N)=s_3$.
\end{proof}

\section{Fluctuations on the manifold (proofs)}
\label{S:fluct}

The aim of this section is to prove Theorem \ref{thm:theta_fluc}. We start by giving an auxiliary lemma.

\begin{lem}\label{lem:isochron_close}
There exists some $C>0$ such that for any $g \in B(\mathcal{U},\varepsilon_0)$,
$$ \textnormal{dist}_{L^2}\left(g,\mathcal{U}\right) \leq \Vert g -u_{\theta(g)}\Vert_{2} \leq C \textnormal{dist}_{L^2}\left(g,\mathcal{U}\right).$$
\end{lem}
\begin{proof} Let $g\in  B(\mathcal{U},\varepsilon_0)$. The first inequality directly comes from the definition of $\text{dist}_{L^2}\left(g,\mathcal{U}\right)$. By compactness of $\mathcal{U}$, there exists some $y\in \mathcal{U}$ such that $\textnormal{dist}_{L^2}\left(g,\mathcal{U}\right) = \Vert g -y\Vert_{2}$ (and $y=u_{\theta(y)}$). Then 
$$\Vert g - u_{\theta(g)}\Vert_{2}\leq \Vert g-y \Vert_{2} + \Vert u_{\theta(y)}-u_{\theta(g)}\Vert_{2},$$
and as $\phi\mapsto u_\phi$ and $\theta$ are Lipschitz continuous (recall $u_\phi=A\cos(\cdot+\phi)$ and $\theta$ is $\mathcal{C}^2$ from Proposition \ref{prop:ex_isochron}), $\Vert u_{\theta(y)}-u_{\theta(g)}\Vert_{2} \leq \hat{C} \Vert g-y\Vert_{2}$ for some $\hat{C}>0$ (independent of the choice of $g$).
\end{proof}

\subsection{Main structure of the proof of Theorem \ref{thm:theta_fluc}}

First, \ref{thm:UN_close_U} and Lemma \ref{lem:isochron_close} give that one can find an event $\Omega_N$ such that $\mathbf{P}(\Omega_N)\xrightarrow[N\to\infty]{}1$ and on this event
\begin{equation}
\label{eq:prox_UN_utheta}
\sup_{t\in [T_0(N),T_f(N)]} \left\Vert  U_N(t) - u_{\theta(U_N(t))} \right\Vert_2  = O\left( N^{\eta-1/2}\right),
\end{equation} 
 with $T_0(N)=C\log(N)$ and $T_f(N)=N\tau_f$. It remains to study the behavior of the isochron map of the process, that is $\theta(U_N(t))$. We do a change of variables and introduce $\tau_0(N):= \dfrac{T_0(N)}{N}$, we define for any $\tau\in [\tau_0(N),\tau_f]$ the rescaled process
\begin{equation}
\label{eq:def_theta_hat}
\widehat{\theta}_N(\tau)=\theta\left( U_N\left(N\tau\right)\right).
\end{equation}
In the proof, we keep the notation $t$ for the microscopic time variable, that is when $t\in [T_0(N),T_f(N)]$ and $\tau$ for the macroscopic time variable, when $\tau\in [\tau_0(N),\tau_f]$. Theorem \ref{thm:theta_fluc} relies on the following decomposition of $\widehat{\theta}_N$, obtained by Itô's lemma.

\begin{prop}\label{prop:theta_hat_decomp}
For any initial condition $\tau_0\geq \tau_0(N)$, for any $\tau\geq \tau_0$, $\widehat{\theta}_N(\tau)$ can be written as 
\begin{equation}\label{eq:theta_hat_decomp}
\widehat{\theta}_N(\tau)=\widehat{\theta}_N(\tau_0)+\vartheta_N(\tau_0,\tau)+\Theta_N(\tau_0,\tau),
\end{equation}
where 
$$ \sup_{\tau_0(N)\leq \tau_0\leq \tau \leq \tau_f} \mathbf{E} \left( \vert \vartheta_N(\tau_0,\tau)\vert\right)\xrightarrow[N\to\infty]{}0$$
and $\Theta_N(\tau_0,\tau)$ is a real martingale with quadratic variation
\begin{equation}
\label{eq:crochet_widetheta}
\left[ \Theta_N\right]_\tau = \dfrac{1}{N} \sum_{j=1}^N \int_{\tau_0}^{\tau} \Phi(x_j,\widehat{\theta}_N(s))f(u_{\widehat{\theta}_N(s)}(x_j))ds
\end{equation}
with 
\begin{equation}
\label{eq:def_Phi}
\Phi(x,\theta):=4\pi^2 \sin^2(x+\theta).
\end{equation}
\end{prop}

The proof of Proposition \ref{prop:theta_hat_decomp} is postponed to Section \ref{S:decomposition_hattheta}. The remaining of the proof of Theorem \ref{thm:theta_fluc} is to prove the tightness of $\left( \widehat{\theta}_N(t)\right)$ and to identify its limit. We apply Aldous criterion: note first that for any $\tau \in [\varepsilon,\tau_f]$, $\widehat{\theta}_N(\tau)\in S$ a compact set. Let $(\tau_N)_N$ be a bounded sequence of $\widehat{\theta}_N$-optional times, let $(h_N)$ be a sequence of positive constants such that $h_N\to 0$. From Proposition \ref{prop:theta_hat_decomp}, we have
$$\widehat{\theta}_N(\tau_N+h_N)-\widehat{\theta}_N(\tau_N) = \vartheta_N(\tau_N,\tau_N+h_N) +\Theta_N(\tau_N,\tau_N+h_N),$$
where $\vartheta_N(\tau_N,\tau_N+h_N)\xrightarrow[N\to\infty]{L^1}0$ and  $\Theta_N$ has the quadratic variation
$$\left[ \Theta_N \right]_{\tau_N+h_N}= \dfrac{1}{N} \sum_{j=1}^N \int_{\tau_N}^{\tau_N+h_N} \Phi(x_j,\widehat{\theta}_N(s))f(u_{\widehat{\theta}_N(s)}(x_j))ds.$$
Using Burkholder-Davis-Gundy inequality, as $\Phi$ and $f$ are bounded, we have that
$$\mathbf{E}\left[\Theta_N(\tau_N,\tau_N+h_N)^2\right]\leq C \mathbf{E}\left[ \left[ \Theta_N\right]_{\tau_N+h_N} \right] \leq Ch_N$$
for some positive constants $C$. We obtain then that $\widehat{\theta}_N(\tau_N+h_N)-\widehat{\theta}_N(\tau_N) \xrightarrow[N\to\infty]{L^1}0$ hence the convergence in probability: for all $\varepsilon>0$,
$$\mathbf{P}\left( \left\vert \widehat{\theta}_N(\tau_N+h_N)-\widehat{\theta}_N(\tau_N) \right\vert> \varepsilon \right)\xrightarrow[N\to\infty]{}0.$$
We can then use Aldous criterion (see Theorem 16.8 of \cite{billingsley}): $\left(\tau\in [\varepsilon_N,\tau_f] \mapsto \widehat{\theta}_N(\tau) \right)_N$ is tight. Let $\tau\mapsto\widehat{\theta}(\tau)$ be a limit in distribution of any subsequence of $(\tau\mapsto\widehat{\theta}_N(\tau))_N$ (by convenience renamed $\widehat{\theta}_N$) , that is $\widehat{\theta}_N \xrightarrow[N\to\infty]{law} \widehat{\theta}$. By Skorokhod's representation theorem, we can represent this convergence on a common probability space such that $\widehat{\theta}_N \xrightarrow[N\to\infty]{a.s.} \widehat{\theta}$. Using this in \eqref{eq:crochet_widetheta}, we obtain that for any $\tau\in [0,\tau_f]$, as $N$ goes to infinity, the quadratic variation of $\widehat{\theta}$ is
$$\left[\widehat{\theta}\right]_\tau= 2\pi \int_0^\tau \int_S\sin^2(x+\widehat{\theta}(s))f\left(A\cos(x+\widehat{\theta}(s))\right) dx~ ds=\sigma^2 \tau,$$
with $\sigma$ defined in \eqref{eq:def_sigma}. We conclude by Lévy's characterization theorem and obtain \eqref{eq:theta_fluc}.


\subsection{About the decomposition of Proposition \ref{prop:theta_hat_decomp}} \label{S:decomposition_hattheta}

\begin{proof}[Proof of Proposition \ref{prop:theta_hat_decomp}]  To show \eqref{eq:theta_hat_decomp}, we study $\left(\theta(U_N(t)\right)_{t\in [T_0(N),T_f(N)]}$. To simplify the notations, we introduce
\begin{equation}
\label{eq:def_theta_N}
\theta_N(t):=\theta\left(U_N(t)\right).
\end{equation}
Note that from the decomposition \eqref{eq:mild_form_tildeU} of $U_N(t)$ and the definition $M_N(t)$ in \eqref{eq:def_M_N}, one can write
$$dU_N(t)=B_N(t)dt + dM_N(t)$$
where $B_N(t):=-U_N(t)+ \cos\ast f\left( U_N(t)\right)+ \Upsilon_t$, with 
\begin{equation}\label{eq:def_Upsilon}
\Upsilon_t(x)=\sum_{i=1}^N \left(  \dfrac{2\pi}{N}\sum_{j=1}^N \cos(x_i-x_j)f(U_{N,j}(t-)) - \int_S\cos(x-y)f(U_N(t)(y))dy\right)\mathbf{1}_{B_{N,i}}(x).
\end{equation}
The starting point is to write the semimartingale decomposition of $\theta(U_N(t))$ from Itô formula:
\begin{multline}\label{eq:ito_theta}
\theta(U_N(t))= \theta\left(U_N(t_0)\right)+ \int_{t_0}^t  D\theta \left(U_N(s-)\right)[-U_N(s)+ \cos\ast f\left( U_N(s-)\right)]ds\\ +\int_{t_0}^t  D\theta \left(U_N(s-)\right)\Upsilon_sds+ \int_{t_0}^t  D\theta \left(U_N(s-)\right)[dM_N(s)] \\+ \sum_{j=1}^N\int_{t_0}^t\int_0^\infty \left[ \theta\left( U_N(s-)+\chi_j(s,z)\right) - \theta\left(U_N(s-)\right) - D\theta\left(U_N(s-)\right)\left[\chi_j(s,z)\right]\right]\pi_j(ds,dz)\\
=: \theta(U_N(t_0))+I_1^N(t_0,t)+ I_2^N(t_0,t) + I_3^N(t_0,t) + I_4^N(t_0,t).
\end{multline}
We are going to focus on each of the terms of \eqref{eq:ito_theta}, that is $I_k^N(t_0,t)$ for $k\in \{1,2,3,4\}$. We have the following lemmas.

\begin{lem}\label{lem:I1}
We have
\begin{equation}\label{eq:ctrl_I1}
\sup_{t_0\in [T_0(N),T_f(N)]} \sup_{t\in (t_0, T_f(N))} \left\vert I_1^N(t_0,t) \right\vert \xrightarrow[N\to\infty]{}0
\end{equation}
in probability.
\end{lem}

\begin{lem}\label{lem:I2}
We have
\begin{equation}\label{eq:ctrl_I2}
\sup_{t_0\in [T_0(N),T_f(N)]} \sup_{t\in (t_0, T_f(N))} \left\vert I_2^N(t_0,t) \right\vert \xrightarrow[N\to\infty]{}0
\end{equation}
in probability.
\end{lem}

\begin{lem}\label{lem:I3}
For any $t_0,t\in [T_0(N),T_f(N)]$, $t_0\leq t$, we have
\begin{equation}\label{eq:ctrl_I3}
I_3^N(t_0,t)=\widetilde{\Theta}_N(t_0,t) + J_3^N(t_0,t)
\end{equation}
where $\sup_{s\in (t_0,T_f(N))} \mathbf{E} \left( \left\vert J_3^N(t_0,s) \right\vert \right) \xrightarrow[N\to\infty]{}0$ and $\widetilde{\Theta}_N $ is a real martingale with quadratic variation $$[\widetilde{\Theta}_N]_t=\dfrac{1}{N^2}\sum_{j=1}^N\int_{t_0}^t \Phi(x_j,\theta(U_N(s-))) f(u_{\theta(U_N(s-))}(x_j))ds$$ with $\Phi$ defined in \eqref{eq:def_Phi}.
\end{lem}

\begin{lem}\label{lem:I4} We have
$$\sup_{t_0 \in [T_0(N),T_f(N)]}  \sup_{t\in (t_0, T_f(N))}  \mathbf{E} \left( \left\vert I_4^N(t_0,t)\right\vert \right) \xrightarrow[N\to\infty]{}0.$$
\end{lem}

The proofs of these fours lemmas are postponed to Section \ref{S:lemmas_pr}.
Combining them, we can define some random variable $J_N(t_0,t)$ such that  $\sup_{s\in (t_0,T_f(N))} \mathbf{E} \left( \left\vert J^N(t_0,s) \right\vert \right) \xrightarrow[N\to\infty]{}0$ and for any $t_0,t\in [T_0(N),T_f(N)]$, $t_0\leq t$,
$$\theta\left(U_N(t)\right) = \theta\left(U_N(t_0)\right) + J^N(t_0,t) +\widetilde{\Theta}_N(t_0,t). $$
Recall the change of variables used to define $\widehat{\theta}$ in \eqref{eq:def_theta_hat}. Define similarly $\vartheta_N(\tau_0,\tau):=J^N(N\tau_0,N\tau)$ and $\Theta_N(\tau_0,\tau)=\widetilde{\Theta}_N(N\tau_0,N\tau)$ for $\tau_0=t_0/N$ and $\tau=t/N$. Then we have exactly shown \eqref{eq:theta_hat_decomp}.
\end{proof}

\subsection{Control of the terms of the decomposition} \label{S:lemmas_pr}

For simplicity, we may write $I_k(t)$ instead of $I_k^N(t_0,t)$. In the following, we use the notations $g_s=O(\alpha_N)$ with $g:s\in I \mapsto g_s \in L^2(S)$ for some time interval $I$ and a sequence $(\alpha_N)$ independent of the time $s$ when there exists some $C$ (independent of $N$) such that for all $x\in S$, $\sup_{s\in I} \sup_{x\in S}\vert g_s(x) \vert \leq C \alpha_N$. Recall the definition of $\theta_N(t)$ in \eqref{eq:def_theta_N}. In the following proofs, this notation will be essentially used for $t=s-$, so that we write for simplicity $\theta_N=\theta_N(s-)$.

\subsubsection{Proof of Lemma \ref{lem:I1}}

Recall that $$I_1(t):=\int_{t_0}^t  D\theta \left(U_N(s-)\right)[-U_N(s)+ \cos\ast f\left( U_N(s)\right)]ds.$$
Define for $g\in L^2(S)$
$$\mathcal{V}(g):=-g+\cos\ast f(g).$$
Recall that for any $\phi\in S$, $\mathcal{L}(u_\phi)=0$ and $D\mathcal{V}(u_\phi)[h]=\mathcal{L}_\phi h$. Let $g\in B(\mathcal{U},\varepsilon_0)$, and $t\mapsto g_t:=\psi_t(g)$ defined in \eqref{eq:def_flow}, that is the flow of \eqref{eq:NFE_specific} under initial condition $g$. Note that by definition of the isochron map $\theta$ in Proposition \ref{prop:ex_isochron} and the fact that $\mathcal{U}$ consists of stationary solutions to \eqref{eq:NFE_specific}, one has that $\theta(\psi_t(g))=\theta(\psi_0(g))=\theta(g)$.
 Differentiating with respect to $t$ (recall Proposition \ref{prop:ex_isochron}) gives that $D\theta (g_t)[\partial_t g_t]=D\theta (g_t)[-g_t+\cos\ast f(g_t)]=0$. Since this is for all $t\geq 0$, taking $t=0$ gives $D\theta(g)[-g+\cos\ast f(g)]=0.$ Hence for any $s$, $ D\theta \left(U_N(s)\right)[\mathcal{V}(U_N(s))]=0$ and as $\mathcal{V}(u_{\theta(U_N(s))})=0$, we have
\begin{align*}
I_1(t)&= \int_{t_0}^t  D\theta \left(U_N(s-)\right)[\mathcal{V}(U_N(s))]ds\\&=\int_{t_0}^t  (D\theta \left(U_N(s-)\right)-D\theta \left(U_N(s)\right))[\mathcal{V}(U_N(s))]ds\\
&= \int_{t_0}^t  (D\theta \left(U_N(s-)\right)-D\theta \left(U_N(s)\right))[\mathcal{V}(U_N(s))-\mathcal{V}(u_{\theta(U_N(s))}]ds.
\end{align*}
As $\theta$ and $\mathcal{V}$ are Lipschitz continuous, as from \eqref{eq:def_UiN} a jump of the process gives a.s. at most an increment of $\frac{2\pi}{N}$ between $U_N(s-)$ and $U_N(s)$, using \eqref{eq:prox_UN_utheta} there exists some $C>0$ (independent of $N$ and of the time) such that 
$$I_1(t)\leq (t-t_0) \Vert \theta\Vert_{\text{lip}} \dfrac{2\pi}{N} \Vert \mathcal{V}\Vert_{\text{lip}} \Vert  U_N(s-) -u_{\theta(U_N(s))}\Vert_2 \leq \dfrac{C T_f(N)}{N} N^{\eta-1/2}$$
on the event $\Omega_N$ (given by Theorem \ref{thm:UN_close_U}). As $T_f(N) \propto N$ and from the choice on $\eta$, \eqref{eq:ctrl_I1} follows.

\subsubsection{Proof of Lemma \ref{lem:I2}}

We place ourselves again on the event $\Omega_N$ (given by Theorem \ref{thm:UN_close_U}) on which we have \eqref{eq:prox_UN_utheta}. Recall that $I_2(t):=\int_{t_0}^t  D\theta \left(U_N(s-)\right)\Upsilon_sds$, where the definition of $\Upsilon$ is given in \eqref{eq:def_Upsilon}.
We have
\begin{equation}\label{eq:termes_I2}
I_2(t)= \int_{t_0}^t\left(  D\theta (U_N(s-))-D\theta(u_{\theta_N})\right)\left[\Upsilon_s\right]ds+ \int_{t_0}^t  D\theta \left(u_{\theta_N}\right)\left[\Upsilon_s- \tilde{\Upsilon}_s\right]ds+\int_{t_0}^t  D\theta \left(u_{\theta_N}\right)\left[\tilde{\Upsilon}_s\right]ds,
\end{equation}
with 
\begin{equation}\label{eq:def_Upsilon_tilde}
\tilde{\Upsilon}_s(x)=\sum_{i=1}^N \left(  \dfrac{2\pi}{N}\sum_{j=1}^N \cos(x_i-x_j)f(u_{\theta_N}(x_j)) - \int_S\cos(x-y)f(u_{\theta_N}(y))dy\right)\mathbf{1}_{B_{N,i}}.
\end{equation}
From \eqref{eq:norme_Upsilon_2} and \eqref{eq:norme_Upsilon_3} we have that $\Vert \Upsilon_s\Vert_2 \leq \dfrac{C}{N}$ for some $C>0$ independent of $N$ and $s$, thus, for the first term of \eqref{eq:termes_I2}, as done before using \eqref{eq:prox_UN_utheta},
$$  \int_{t_0}^t\left(  D\theta (U_N(s-))-D\theta(u_{\theta(U_N(s-))})\right)\left[\Upsilon_s\right]ds \leq (t-t_0) C \Vert \theta \Vert_{\text{lip}} N^{\eta-1/2} \dfrac{C}{N}\leq \dfrac{CT_f(N)}{N} N^{\eta-1/2}.$$

For the third term of \eqref{eq:termes_I2}, using \eqref{eq:diff_theta}, we have $\displaystyle D\theta \left(u_{\theta(U_N(s-))}\right)\left[\tilde{\Upsilon}_s\right] = \dfrac{\langle v_{\theta(U_N(s-))},\tilde{\Upsilon}_s\rangle_{\theta(U_N(s-))}}{\Vert v_{\theta(U_N(s-))} \Vert_{\theta(U_N(s-))}}$. As shown in \eqref{eq:A_kappa&Isin}, $\Vert v_{\theta_N} \Vert_{\theta_N}=A$. From trigonometric formula one has

\begin{align}
\langle v_{\theta_N},\tilde{\Upsilon}_s\rangle_{\theta_N}&= \langle v_{\theta_N}, \dfrac{2\pi}{N}\sum_{i,j=1}^N \cos(x_i-x_j)f(u_{\theta_N}(x_j))\mathbf{1}_{B_{N,i}}- \int_S\cos(\cdot-y)f(u_{\theta_N}(y))dy \rangle_{\theta_N}\notag\\
&= \left(  \dfrac{2\pi}{N}\sum_{j=1}^N \cos(x_j+\theta_N)f(u_{\theta_N}(x_j))\right)\left( \sum_{i=1}^N \cos(x_i+\theta_N)  \langle v_{\theta_N}, \mathbf{1}_{B_{N,i}}\rangle_{\theta_N}\right) \notag\\
&\quad + \left(  \dfrac{2\pi}{N}\sum_{j=1}^N \sin(x_j+\theta_N)f(u_{\theta_N}(x_j))\right) \left( \sum_{i=1}^N \sin(x_i+\theta_N)  \langle v_{\theta_N}, \mathbf{1}_{B_{N,i}}\rangle_{\theta_N}\right) \notag\\
&\quad - \left( \int_S \cos(y+\theta_N)f(u_{\theta_N}(y))dy\right)  \langle v_{\theta_N},\cos(\cdot+\theta_N)\rangle_{\theta_N}\notag\\
&\quad - \left( \int_S \sin(y+\theta_N)f(u_{\theta_N}(y))dy\right)  \langle v_{\theta_N},\sin(\cdot+\theta_N)\rangle_{\theta_N}. \label{eq:tilde_Upsilon_aux}
\end{align}
By invariance of rotation and with Lemma \ref{lem:computation_A} we have $ \langle v_{\theta_N},\cos(\cdot+\theta_N)\rangle_{\theta_N} = \mathcal{I}(\sin\cos)=0$, and similarly $ \int_S \sin(y+\theta_N)f(u_{\theta_N}(y))dy=0$. We can then write \eqref{eq:tilde_Upsilon_aux} as
$\langle v_{\theta_N},\tilde{\Upsilon}_s\rangle_{\theta_N}=A_1 A_2 + A_3 A_4.$ 
From the computations \eqref{eq:A1_eq}, \eqref{eq:A2_eq}, \eqref{eq:A3_eq} and \eqref{eq:A4_eq} of Lemma \ref{lem:computations_Ai}, we obtain that 
\begin{equation}
\label{eq:Upsilon_v}
\langle v_{\theta_N},\tilde{\Upsilon}_s\rangle_{\theta_N}:= \dfrac{\pi A^2}{N}+  o\left(\dfrac{1}{N}\right).
\end{equation}

For the second term of \eqref{eq:termes_I2}, we have with Lemma \ref{lem:computation_A} that $ \langle v_{\theta_N},\sin(\cdot+\theta_N)\rangle_{\theta_N} =-A\mathcal{I}(\sin^2)=-A$ thus
\begin{align*}
\langle v_{\theta_N}, \Upsilon_s- \tilde{\Upsilon}_s \rangle_{\theta_N} &= A_2 \left( \dfrac{2\pi}{N} \sum_{j=1}^N \cos(x_j+\theta_N) \left( f\left(U_N(s-)(x_j)\right)-f\left(u_{\theta_N}(x_j)\right)\right) \right) \\
&\quad + A_4 \left( \dfrac{2\pi}{N} \sum_{j=1}^N \sin(x_j+\theta_N) \left( f\left(U_N(s-)(x_j)\right)-f\left(u_{\theta_N}(x_j)\right)\right) \right)  \\
&\quad + A \int_S \sin(y+\theta_N)f(U_N(s-)(y))dy.
\end{align*}
Let us show that 
\begin{equation}
\label{eq:aux_DN}
D_N:=\dfrac{2\pi}{N} \sum_{j=1}^N \cos(x_j+\theta_N) \left( f\left(U_N(s-)(x_j)\right)-f\left(u_{\theta_N}(x_j)\right)\right) = O\left( N^{\eta-1/2}\right).
\end{equation}
Setting $\widehat{u}_{\theta_N}(y):=\sum_{k=1}^N u_{\theta_N}(x_k)\mathbf{1}_{y\in B_{N,k}}$, we have
\begin{align*}
\vert D_N \vert &= \left\vert \sum_{j=1}^N \cos(x_j+\theta_N)  \int_S \left( f\left(U_N(s-)(x_j)\right)-f\left(u_{\theta_N}(x_j)\right)\right)\mathbf{1}_{y\in B_{N,j}}dy \right\vert\\
&\leq \Vert f \Vert_{\text{lip}} \sum_{j=1}^N \int_S \left\vert U_N(s-)(y)-\widehat{u}_{\theta_N}(y)\right\vert \mathbf{1}_{y\in B_{N,j}}dy.
\end{align*} 
With Cauchy–Schwarz inequality and Jensen's discrete inequality, we have
\begin{align*}
\vert D_N \vert &\leq C_f \sum_{j=1}^N \left( \int_S \left\vert U_N(s-)(y)-\widehat{u}_{\theta_N}(y)\right\vert^2 \mathbf{1}_{y\in B_{N,j}} dy\right)^{1/2} \left( \int_S \mathbf{1}_{y\in B_{N,j}}dy \right)^{1/2}\\
&= \dfrac{\sqrt{2\pi N} C_f }{N} \sum_{j=1}^N \left( \int_S \left\vert U_N(s-)(y)-\widehat{u}_{\theta_N}(y)\right\vert^2 \mathbf{1}_{y\in B_{N,j}} dy\right)^{1/2}\\
&\leq \sqrt{2\pi N} C_f \left( \dfrac{1}{N} \sum_{j=1}^N \int_S \left\vert U_N(s-)(y)-\widehat{u}_{\theta_N}(y)\right\vert^2 \mathbf{1}_{y\in B_{N,j}} dy\right)^{1/2}\\
&= C\sqrt{N} \left( \dfrac{1}{N} \int_S \left\vert U_N(s-)(y)-\widehat{u}_{\theta_N}(y)\right\vert^2  dy\right)^{1/2}\\
&= C \left\Vert U_N(s-)-\widehat{u}_{\theta_N}\right\Vert_2 \leq C \left\Vert U_N(s-)-u_{\theta_N}\right\Vert_2  + C \left\Vert u_{\theta_N}-\widehat{u}_{\theta_N}\right\Vert_2, 
\end{align*}
hence with \eqref{eq:prox_UN_utheta} and as $\left\Vert u_{\theta_N}-\widehat{u}_{\theta_N}\right\Vert_2=O(1/N)$, we have indeed shown that  $D_N = O\left( N^{\eta-1/2}\right)$. Similarly, one can show that $\dfrac{2\pi}{N} \sum_{j=1}^N \sin(x_j+\theta_N) \left( f\left(U_N(s-)(x_j)\right)-f\left(u_{\theta_N}(x_j)\right)\right)= O\left( N^{\eta-1/2}\right)$. Using Lemma \ref{lem:computations_Ai} and as $\int_S \sin(y+\theta_N)f(u_{\theta_N}(y))dy=0$, we have
\begin{multline*}
\langle v_{\theta_N}, \Upsilon_s- \tilde{\Upsilon}_s \rangle_{\theta_N} = O\left( N^{\eta-3/2}\right) + A \int_S \sin(y+\theta_N)\left( f(U_N(s-)(y)) - f(u_{\theta_N}(y))\right)dy \\ -A  \dfrac{2\pi}{N} \sum_{j=1}^N \sin(x_j+\theta_N) \left( f\left(U_N(s-)(x_j)\right)-f\left(u_{\theta_N}(x_j)\right)\right).
\end{multline*}
Using Taylor's expansion, we obtain
\begin{equation}\label{eq:aux_Upsilon_Upsilontilde}
\langle v_{\theta_N}, \Upsilon_s- \tilde{\Upsilon}_s \rangle_{\theta_N} = o\left( \dfrac{1}{N}\right) + A~  \Delta_N,
\end{equation}
where $$\Delta_N = \int_S \sin(y+\theta_N)f'(u_{\theta_N}(y))\left(U_N(s-)(y) -u_{\theta_N}(y))\right)dy - \dfrac{2\pi}{N} \sum_{j=1}^N \sin(x_j+\theta_N)f'(u_{\theta_N}(x_j))  \left( U_N(s-)(x_j)-u_{\theta_N}(x_j)\right).$$
Define $\widehat{u}_{\theta_N}(y):=\sum_{j=1}^N u_{\theta_N}(x_j)\mathbf{1}_{B_{N,j}}(y)$, we introduce it in $\Delta_N$ so that
\begin{align}\label{eq:aux_Delta}
\Delta_N&= \sum_{j=1}^N \int_{B_{N,j}} \left[  \sin(y+\theta_N)f'(u_{\theta_N}(y))\left(U_N(s-)(y) -u_{\theta_N}(y))\right) - \sin(x_j+\theta_N)f'(u_{\theta_N}(x_j))  \left( U_N(s-)(y)-\widehat{u}_{\theta_N}(y)\right)\right] dy\notag\\
&=  \sum_{j=1}^N \int_{B_{N,j}}  \left(  U_N(s-)(y)-\widehat{u}_{\theta_N}(y)\right) \left( \sin(y+\theta_N)f'(u_{\theta_N}(y))-  \sin(x_j+\theta_N)f'(u_{\theta_N}(x_j)) \right)dy \\
&\quad +  \sum_{j=1}^N \int_{B_{N,j}}  \sin(y+\theta_N)f'(u_{\theta_N}(y)) \left( \widehat{u}_{\theta_N}(y) - u_{\theta_N}(y)\right) dy. \notag
\end{align}
For the first term of $\Delta_N$, let $\alpha_N(y):= \sin(y+\theta_N)f'(u_{\theta_N}(y))-  \sin(x_j+\theta_N)f'(u_{\theta_N}(x_j))$, one has with Cauchy–Schwarz inequality that
$$\sum_{j=1}^N \int_S \left(  U_N(s-)(y)-\widehat{u}_{\theta_N}(y)\right)\alpha_N(y)dy\leq \sum_{j=1}^N \left( \int_S \left(  U_N(s-)(y)-\widehat{u}_{\theta_N}(y)\right)^2 \mathbf{1}_{B_{N,j}}dy\right)^{1/2}\left( \int_{B_{n,j}}\alpha_N(y)^2dy\right)^{1/2}.$$
As $\int_{B_{n,j}}\alpha_N(y)^2dy \leq \int_{B_{n,j}}(y-x_j)^2dy = O\left(\dfrac{1}{N^{3/2}}\right)$, for some $C>0$, using Jensen’s inequality
\begin{align*}
\sum_{j=1}^N \int_S \left(  U_N(s-)(y)-\widehat{u}_{\theta_N}(y)\right)\alpha_N(y)dy &\leq \dfrac{C}{\sqrt{N}} \dfrac{1}{N} \sum_{j=1}^N \left( \int_S \left(  U_N(s-)(y)-\widehat{u}_{\theta_N}(y)\right)^2 \mathbf{1}_{B_{N,j}}dy\right)^{1/2}\\
&\leq \dfrac{C}{\sqrt{N}} \sqrt{ \dfrac{1}{N} \sum_{j=1}^N  \int_S \left(  U_N(s-)(y)-\widehat{u}_{\theta_N}(y)\right)^2 \mathbf{1}_{B_{N,j}}dy}\\
&\leq  \dfrac{C}{\sqrt{N}} \sqrt{ \dfrac{1}{N}\Vert U_N(s-)-\widehat{u}_{\theta_N} \Vert_2^2}.
\end{align*}
As $\Vert u_\theta-\widehat{u}_{\theta_N} \Vert_2^2=O\left(\dfrac{1}{N^2}\right)$ and with \eqref{eq:prox_UN_utheta}, we obtain that the first term of \eqref{eq:aux_Delta} is in $O\left( \dfrac{N^{\eta-1/2}}{N}\right)$. For the second term of $\Delta_N$, we have
\begin{align*}
&\sum_{j=1}^N \int_{B_{N,j}}  \sin(y+\theta_N)f'(u_{\theta_N}(y)) \left( \widehat{u}_{\theta_N}(y) - u_{\theta_N}(y)\right) dy\\
&= \sum_{j=1}^N \int_{B_{N,j}}  \sin(y+\theta_N)f'(u_{\theta_N}(y)) \left( A\cos(x_j+\theta_N) -A\cos(y+\theta_N)\right) dy\\
&=A \sum_{j=1}^N \int_{B_{N,j}}  \sin(y+\theta_N)f'(u_{\theta_N}(y)) \sin(x_j+\theta_N)(y-x_j) dy+o\left(\dfrac{1}{N}\right)\\
&= A \sum_{j=1}^N \sin(x_j+\theta_N)^2 f'(u_{\theta_N}(x_j)) \int_{B_{N,j}} (y-x_j)dy +o\left(\dfrac{1}{N}\right)\\
&=A \sum_{j=1}^N \sin(x_j+\theta_N)^2 f'(u_{\theta_N}(x_j)) \left( - \dfrac{2\pi^2}{N^2}\right) +o\left(\dfrac{1}{N}\right)\\
&= - \dfrac{\pi}{N} \int_S A \sin(y+\theta_N)^2 f'(u_{\theta_N}(y))dy +o\left(\dfrac{1}{N}\right)= - \dfrac{A\pi}{N}  +o\left(\dfrac{1}{N}\right).
\end{align*}
Coming back to \eqref{eq:aux_Upsilon_Upsilontilde}, we have then that 
\begin{equation}
\label{eq:Upsilon_Upsilon_tilde}
\langle v_{\theta_N}, \Upsilon_s- \tilde{\Upsilon}_s \rangle_{\theta_N}=-\dfrac{\pi A^2}{N}+  o\left(\frac{1}{N}\right).
\end{equation}
This term \eqref{eq:Upsilon_Upsilon_tilde} cancels with the previous computation \eqref{eq:Upsilon_v} up to some rest of order $o\left(\dfrac{1}{N}\right)$. We obtain then  \eqref{eq:ctrl_I2} after integrating on $(t_0,t)$ and using $T_f(N) \propto N$.

\subsubsection{Proof of Lemma \ref{lem:I3}}

Recall that $I_3(t):=\int_{t_0}^t  D\theta \left(U_N(s-)\right)[dM_N(s)]$. Recall the definition of $\chi_j$ in \eqref{eq:def_chi} and the compensated measure $\tilde{\pi}_j$, we can re-write the term $I_3^N(t_0,t)$ and introduce $D\theta(u_{\theta_N})$:
\begin{equation}
\label{eq:I3}
I_3(t)=\sum_{j=1}^N \int_{t_0}^t \int_0^\infty \left(  D\theta \left(U_N(s-)\right)- D\theta \left(u_{\theta_N}\right)\right)[\chi_j(s,z)]\tilde{\pi}_j(ds,dz) + \sum_{j=1}^N \int_{t_0}^t \int_0^\infty  D\theta \left(u_{\theta_N}\right)[\chi_j(s,z)]\tilde{\pi}_j(ds,dz).
\end{equation}
Let us focus first on 
$$Q_0(t):=\sum_{j=1}^N \int_{t_0}^t \int_0^\infty \left(  D\theta \left(U_N(s-)\right)- D\theta \left(u_{\theta_N}\right)\right)[\chi_j(s,z)]\tilde{\pi}_j(ds,dz).$$
It is a real martingale. We denote by $[Q_0]_t=\sum_{s\leq t} \left| \Delta Q_0(t) \right|^2$ its quadratic variation.  It is computed as follows (as the $(\pi_j)_{1\leq j \leq N}$ are independent, there are almost surely no simultaneous jumps so that $[\tilde{\pi}_j,\tilde{\pi}_{j'}]=0$ if $j\neq j'$):
\begin{align*}
[Q_0]_t&= \sum_{j=1}^N \int_{t_0}^t \int_0^\infty \left( \left(  D\theta \left(U_N(s-)\right)- D\theta \left(u_{\theta_N}\right)\right)[\chi_j(s,z)]\right)^2\pi_j(ds,dz)\\
&\leq C \Vert \theta \Vert_{\text{lip}}^2 \left( \sup_{t\in [t_0,t]} \Vert U_N(s-)-u_{\theta_N}\Vert_2\right)^2 \sum_{j=1}^N \int_{t_0}^t \int_0^\infty \Vert\chi_j(s,z)\Vert_2^2 \pi_j(ds,dz)\\
&\leq C N^{2\eta-1} \dfrac{1}{N^2}\sum_{j=1}^N \int_{t_0}^t \int_0^\infty \mathbf{1}_{z\leq \lambda_{N,j}(s)}  \pi_j(ds,dz),
\end{align*}
using \eqref{eq:prox_UN_utheta} and the computation \eqref{eq:norme_chi} for some constants $C>0$. Then, by Burkholder-Davis-Gundy inequality and as $f$ is bounded
$$ \mathbf{E}\left[ Q_0(t)^2 \right] \leq C \mathbf{E}\left[ \left[Q_0\right]_t \right]\leq C N^{2\eta-3} \mathbf{E}\left[\sum_{j=1}^N \int_{t_0}^t \int_0^\infty \mathbf{1}_{z\leq \lambda_{N,j}(s)}  \pi_j(ds,dz)\right]\leq CN^{2\eta-1} \dfrac{T_f(N)}{N} \xrightarrow[N\to\infty]{}0,$$
hence $Q_0(t)$ converges in $L^1$ towards 0 as $N\to\infty$ uniformly in $t$.
 
The other term $Q(t):=\sum_{j=1}^N \int_{t_0}^t \int_0^\infty  D\theta \left(u_{\theta(U_N(s-)}\right)[\chi_j(s,z)]\tilde{\pi}_j(ds,dz)$ in \eqref{eq:I3} is also a real martingale, we denote by $[Q]_t=\sum_{s\leq t} \left| \Delta Q(t) \right|^2$ its quadratic variation and it is computed as follows:
$$[Q]_t =\sum_{j=1}^N\int_{t_0}^t \int_0^\infty  \left(  D\theta \left(u_{\theta_N}\right)[\chi_j(s,z)]\right)^2\pi_{j}(ds,dz)=\sum_{j=1}^N\int_{t_0}^t \int_0^\infty  \left( \dfrac{\langle v_{\theta_N}, \chi_j(s,z) \rangle_{\theta_N}}{\Vert  v_{\theta_N} \Vert_{\theta_N}}\right)^2\pi_{j}(ds,dz),$$
where we used \eqref{eq:diff_theta}. Recall the notation $w_{ij}^{(N)}=2\pi\cos(x_i-x_j)$, from the computation \eqref{eq:A_kappa&Isin}, $\Vert  v_{\theta_N} \Vert_{\theta_N}=A$ hence
\begin{align}\label{eq:Qt_aux1}
[Q]_t
&= \dfrac{1}{A^2}\sum_{j=1}^N\int_{t_0}^t \int_0^\infty  \left( \langle v_{\theta_N}, \sum_{i=1}^N \mathbf{1}_{B_{N,i}} \dfrac{w_{ij}^{(N)}}{N}  \mathbf{1}_{z\leq \lambda_{N,j}} \rangle_{\theta_N}\right)^2\pi_{j}(ds,dz)\notag\\
&= \dfrac{1}{A^2} \sum_{j=1}^N\int_{t_0}^t \int_0^\infty  \left( \sum_{i=1}^N \dfrac{w_{ij}^{(N)}}{N}  \langle v_{\theta_N}, \mathbf{1}_{B_{N,i}} \rangle_{\theta_N}\right)^2  \mathbf{1}_{z\leq \lambda_{N,j}} \pi_{j}(ds,dz).
\end{align}
Let us focus on the term $ E_N:= \sum_{i=1}^N  \dfrac{w_{ij}^{(N)}}{N}  \langle v_{\theta_N}, \mathbf{1}_{B_{N,i}} \rangle_{\theta_N}$. We have with trigonometric formula 
$$E_N=\dfrac{2\pi}{N} \left(\cos(x_j+\theta_N)\left( \sum_{i=1}^N \cos(x_i+\theta_N) \langle v_{\theta_N}, \mathbf{1}_{B_{N,i}} \rangle_{\theta_N} \right)+\sin(x_j+\theta_N)\left(\sum_{i=1}^N \sin(x_i+\theta_N) \langle v_{\theta_N}, \mathbf{1}_{B_{N,i}} \rangle_{\theta_N} \right) \right).$$
As $\displaystyle\sum_{i=1}^N \cos(x_i+\theta_N) \langle v_{\theta_N}, \mathbf{1}_{B_{N,i}} \rangle_{\theta_N}\xrightarrow[N\to\infty]{}\int_S A \cos \sin f'(A\cos) =0$ (by symmetry) and  $\displaystyle\sum_{i=1}^N \sin(x_i+\theta_N) \langle v_{\theta_N}, \mathbf{1}_{B_{N,i}} \rangle_{\theta_N}\xrightarrow[N\to\infty]{}-\int_S A \sin^2 f'(A\cos) =-A$ with \eqref{eq:A_kappa&Isin}, we have that 
$$\sum_{i=1}^N \dfrac{w_{ij}^{(N)}}{N}  \langle v_{\theta_N}, \mathbf{1}_{B_{N,i}} \rangle_{\theta_N} \sim_{N\to\infty} - \dfrac{2\pi }{N} A \sin(x_j+\theta_N).$$
Hence we have $ \left( \sum_{i=1}^N  \dfrac{w_{ij}^{(N)}}{N}  \langle v_{\theta_N}, \mathbf{1}_{B_{N,i}} \rangle_{\theta_N}\right)^2= \dfrac{A^2}{N^2}\Phi(x_j,\theta_N)$ with $\Phi(x_j,\theta_N)\sim_{N\to\infty}\left(2\pi \sin(x_j+\theta_N)\right)^2$ (bounded independently of $N$, $\theta_N$).
Coming back to \eqref{eq:Qt_aux1}, we have
$$[Q]_t = \dfrac{1}{N^2}\sum_{j=1}^N\int_{t_0}^t \int_0^\infty  \Phi(x_j,\theta_N)  \mathbf{1}_{z\leq \lambda_{N,j}} \pi_{j}(ds,dz)+o\left(\dfrac{1}{N}\right).$$

Let 
\begin{align*}
Q_1(t)&:=\dfrac{1}{N^2}\sum_{j=1}^N\int_{t_0}^t \int_0^\infty  \Phi(x_j,\theta_N)  \left(\mathbf{1}_{z\leq f(U_{N,j}(s-))} -\mathbf{1}_{z\leq f(u_{\theta_N}(x_j))}\right) \pi_{j}(ds,dz)\\
Q_2(t)&:=\dfrac{1}{N^2}\sum_{j=1}^N\int_{t_0}^t \int_0^\infty  \Phi(x_j,\theta_N)  \mathbf{1}_{z\leq f(u_{\theta_N}(x_j))} \tilde{\pi}_{j}(ds,dz)\\
Q_3(t)&:=\dfrac{1}{N^2}\sum_{j=1}^N\int_{t_0}^t \int_0^\infty  \Phi(x_j,\theta_N)  \mathbf{1}_{z\leq f(u_{\theta_N}(x_j))}~dsdz,
\end{align*}
so that $[Q]_t=Q_1(t)+Q_2(t)+Q_3(t)+o\left(\dfrac{1}{N}\right)$. We have (recall that $\Phi$ is bounded)
\begin{align*}
\mathbf{E}\left[ \vert Q_1(t)\vert\right] &\leq \dfrac{1}{N^2} \sum_{j=1}^N \mathbf{E} \left[ \int_{t_0}^t \int_0^\infty  \Phi(x_j,\theta_N)  \left\vert\mathbf{1}_{z\leq f(U_{N,j}(s-))} -\mathbf{1}_{z\leq f(u_{\theta_N}(x_j))}\right\vert \pi_{j}(ds,dz)\right]\\
&= \dfrac{\Vert \Phi \Vert_\infty}{N^2} \sum_{j=1}^N  \int_{t_0}^t \mathbf{E} \left[  \left\vert f(U_{N,j}(s-)) -f(u_{\theta_N}(x_j))\right\vert \right] ds\\
&\leq \dfrac{\Vert \Phi \Vert_\infty \Vert f \Vert_{\text{lip}}}{N} (t-t_0) C N^{\eta-1/2} \leq  C \dfrac{T_f(N)}{N} N^{\eta-1/2} \xrightarrow[N\to\infty]{}0,
\end{align*}
using \eqref{eq:prox_UN_utheta}. About $Q_2$, we use once again that $Q_2$ is a real martingale with quadratic variation
\begin{align*}
[Q_2]_t&= \sum_{j=1}^N\int_{t_0}^t \int_0^\infty  \left(\dfrac{1}{N^2}\Phi(x_j,\theta_N)  \mathbf{1}_{z\leq f(u_{\theta_N}(x_j))}\right)^2 \pi_{j}(ds,dz)\\
&\leq \dfrac{C}{N^4}  \sum_{j=1}^N \int_{t_0}^t \int_0 ^\infty \mathbf{1}_{z\leq f(u_{\theta_N}(x_j))}\pi_{j}(ds,dz),
\end{align*}
hence with Burkholder-Davis-Gundy inequality,
$$\mathbf{E}\left[Q_2(t)^2\right]\leq C \mathbf{E}\left[ \left[Q_2\right]_t \right]\leq  \dfrac{C}{N^4}  \mathbf{E}\left[\sum_{j=1}^N \int_{t_0}^t \int_0^\infty  \mathbf{1}_{z\leq f(u_{\theta_N}(x_j))} \pi_j(ds,dz)\right]\leq \dfrac{C}{N^2}  \dfrac{T_f(N)}{N} \xrightarrow[N\to\infty]{}0.$$ 

The last term $\displaystyle Q_3(t)=\dfrac{1}{N^2}\sum_{j=1}^N\int_{t_0}^t \Phi(x_j,\theta_N) f(u_{\theta_N}(x_j))ds$ gives the term $\widetilde{\Theta}_N(t_0,t)$ in \eqref{eq:ctrl_I3}.

\subsubsection{Proof of Lemma \ref{lem:I4}}

Recall  that $I_4(t)$ is defined in \eqref{eq:ito_theta}. A Taylor's expansion gives that
\begin{align*}
I_4(t) &= \sum_{j=1}^N\int_{t_0}^t\int_0^\infty  \int_0^1 (1-r)D^2\theta \left( U_N(s-) + r\chi_j(s,z) \right) \left[\chi_j(s,z)\right]^2 dr~  \pi_j(ds,dz)\\
&=  \sum_{j=1}^N\int_{t_0}^t\int_0^\infty  \int_0^1 (1-r)D^2\theta \left( U_N(s-) + r\chi_j(s,z) \right) \left[\chi_j(s,z)\right]^2 dr~  \widetilde{\pi}_j(ds,dz)
\\& \quad  +\sum_{j=1}^N\int_{t_0}^t\int_0^\infty  \int_0^1 (1-r)\left( D^2\theta \left( U_N(s-) + r\chi_j(s,z) \right) -D^2 \theta\left( U_N(s-)\right) \right)\left[\chi_j(s,z)\right]^2 drdsdz
\\& \quad  +\sum_{j=1}^N\int_{t_0}^t\int_0^\infty  \int_0^1 (1-r)\left( D^2\theta \left( U_N(s-) \right) -D^2 \theta\left(u_{\theta_N}\right) \right)\left[\chi_j(s,z)\right]^2 drdsdz
\\& \quad  +\sum_{j=1}^N\int_{t_0}^t\int_0^\infty  \int_0^1 (1-r)D^2 \theta\left(u_{\theta_N}\right)\left[\chi_j(s,z)\right]^2 dr  dsdz =: L_1(t) + L_2(t) + L_3(t)+L_4(t).
\end{align*}

$L_1$ is a real martingale and 
\begin{align*}
[L_1](t)&=  \sum_{j=1}^N\int_{t_0}^t\int_0^\infty \left( \int_0^1 (1-r)D^2\theta \left( U_N(s-) + r\chi_j(s,z) \right) \left[\chi_j(s,z)\right]^2 dr\right)^2  \pi_j(ds,dz)\\
&\leq \dfrac{\Vert D^2 \theta \Vert_\infty^2}{2} \sum_{j=1}^N\int_{t_0}^t\int_0^\infty \Vert \chi_j(s,z) \Vert_2^4 \pi_j(ds,dz)\leq \dfrac{C}{N^4 } \sum_{j=1}^N\int_{t_0}^t\int_0^\infty \mathbf{1}_{z\leq \lambda_{N,j}(s)}   \pi_j(ds,dz).
\end{align*}
As done for $Q_2$ in the proof of Lemma \ref{lem:I3}, we obtain that
\begin{equation}\label{eq:I4_aux}
\mathbf{E}\left[\vert L_1(t)^2 \vert\right]\leq \dfrac{C}{N^2} \dfrac{T_f(N)}{N}\xrightarrow[N\to\infty]{}0.
\end{equation}
We have, using \eqref{eq:norme_chi} and the fact that $f$ is bounded
\begin{align*}
L_2(t)&=\sum_{j=1}^N\int_{t_0}^t\int_0^\infty  \int_0^1 (1-r)\left( D^2\theta \left( U_N(s-) + r\chi_j(s,z) \right) -D^2 \theta\left( U_N(s-)\right) \right)\left[\chi_j(s,z)\right]^2 drdsdz\\
&\leq \sum_{j=1}^N\int_{t_0}^t\int_0^\infty \Vert D^2\theta \Vert_{lip} \Vert  \chi_j(s,z) \Vert_{2} \Vert \chi_j(s,z)\Vert_{2}^2 ~ dsdz\\
&\leq C \sum_{j=1}^N \int_{t_0}^t\int_0^\infty \left( \dfrac{1}{N}  \mathbf{1}_{z\leq \lambda_{N,j}(s)} \right)^3 dsdz \leq \dfrac{C}{N^3}  \sum_{j=1}^N \int_{t_0}^t \lambda_{N,j}(s) ds \leq \dfrac{C T_f(N)}{N^2}.
\end{align*}

Similarly, using \eqref{eq:prox_UN_utheta}
\begin{align*}
L_3(t)&= \sum_{j=1}^N\int_{t_0}^t\int_0^\infty  \int_0^1 (1-r)\left( D^2\theta \left( U_N(s-) \right) -D^2 \theta\left(u_{\theta_N}\right) \right)\left[\chi_j(s,z)\right]^2 drdsdz\\
&\leq  \dfrac{1}{2} \Vert D^2 \theta\Vert_{\text{lip}} \sum_{j=1}^N \int_{t_0}^t\int_0^\infty  \Vert U_N(s-)-u_{\theta_N}\Vert_{2} \Vert \chi_j(s,z)\Vert_{2}^2 dsdz\\
&\leq C N^{\eta-1/2} \sum_{j=1}^N \int_{t_0}^t\int_0^\infty  \dfrac{1}{N^2}\mathbf{1}_{z\leq \lambda_{N,j}(s)} dsdz \leq C \dfrac{T_f(N)}{N}  N^{\eta-1/2}.
\end{align*}

For $L_4$, we use the computation of $D^2\theta \left( u_{\theta_N}\right) \left[\chi_j(s,z)\right]^2$ given by Lemma \ref{lem:computations_Dtheta2_chi}: for some $C=C_{A,\gamma}$,
\begin{align*}
L_4(t)&= \frac{1}{2}\sum_{j=1}^N\int_{t_0}^t\int_0^\infty  D^2 \theta\left(u_{\theta_N}\right)\left[\chi_j(s,z)\right]^2 dsdz\\
&= \sum_{j=1}^N\int_{t_0}^t\int_0^\infty  \mathbf{1}_{z\leq \lambda_{N,j}(s)}   \left( \dfrac{C}{N^2}\cos(x_j+\theta)\sin(x_j+\theta)+ O(N^{-3}) \right) dsdz\\
&= \dfrac{C}{N^2}\sum_{j=1}^N\int_{t_0}^t\ \lambda_{N,j}(s)  \cos(x_j+\theta_N)\sin(x_j+\theta_N) ds +   O\left( \dfrac{T_f(N)}{N^2} \right)\\
&=\dfrac{C}{N^2}\sum_{j=1}^N\int_{t_0}^t \left( f(U_N(s-)(x_j) - f(u_{\theta_N}(x_j))\right)\cos(x_j+\theta_N)\sin(x_j+\theta_N) ds +   O\left( \dfrac{T_f(N)}{N^2}  \right)\\& \quad  + \dfrac{C}{N^2}\sum_{j=1}^N\int_{t_0}^t f(u_{\theta_N})\cos(x_j+\theta_N)\sin(x_j+\theta_N) ds.
\end{align*}
As done before for $D_N$ in \eqref{eq:aux_DN}, $\displaystyle\dfrac{2\pi}{N} \sum_{j=1}^N \left( f(U_N(s-)(x_j) - f(u_{\theta_N}(x_j))\right)\cos(x_j+\theta_N)\sin(x_j+\theta_N) = O(N^{\eta-1/2})$ and 
\begin{align*}
\dfrac{C}{N^2}\sum_{j=1}^N f(u_{\theta_N}(x_j))\cos(x_j+\theta_N)\sin(x_j+\theta_N)  &= \dfrac{C}{N} \left( \int_S f(u_{\theta_N}(x))\cos(x+\theta_N)\sin(x+\theta_N)dx + O(N^{-1}) \right)\\
&= O\left(\dfrac{1}{N^2} \right),
\end{align*}
hence as $T_f(N)\propto N$,  $\displaystyle L_4(t) =  O\left( N^{\eta-1/2} \right)$. Combining our results on $L_2, L_3, L_4$, we have then shown that  $\sup_{t\in [T_0(N),T_f(N)]} \left(L_2(t)+L_3(t)+L_4(t)\right) = O\left( N^{\eta-1/2}\right) \xrightarrow[N\to\infty]{}0$. We conclude with \eqref{eq:I4_aux}.

\appendix
\section{Appendix: on the stationary solutions to the Neural Field Equation}
\label{S:appendix_stat}

\subsection{When $f$ is the Heaviside function}\label{S:stat_H}

Here we study  the NFE equation \eqref{eq:NFE_gen} and its stationary solutions \eqref{eq:NFE_stat_gen} when $f=H_\varrho$. We recall the results from of \cite{kilpatrick2013} and \cite{veltzFaugeras2010}.

\begin{prop}\label{prop:stat_sol_H} There exist non-zero stationary solutions to \eqref{eq:NFE_gen} when $f=H_\varrho$, $\nu(dy)= \frac{\mathbf{1}_{[-\pi,\pi)}}{2\pi}dy$ and $w(x,y)=2\pi cos(x-y)$ if and only if $\varrho\in [-1,1]$, and in this case, the set of stationary solutions is  $\mathcal{U}_0~ \cup~  \mathcal{U}_{A_+(0)} ~ \cup~  \mathcal{U}_{A_-(0)}$, where $A_+(0)$ and $A_-(0)$ are defined in \eqref{eq:def_A0}.
\end{prop}

\begin{proof}
(following \cite{kilpatrick2013}) First, $u=0$ is an evident solution to \eqref{eq:NFE_stat_gen}. We focus now on the other solutions. To solve \eqref{eq:NFE_stat_gen}, we need to find $A$ solving \eqref{eq:A_self}.  As $A\cos(x)=-A\cos(x+\pi)$, $\mathcal{U}_A=\mathcal{U}_{-A}$ and we can focus on the case $A>0$. 

Let $A>0$ be a solution to \eqref{eq:A_self} with $f=H_\varrho$. Note that we necessarily need $A\geq \vert \varrho \vert$, because if $A<\varrho$, the threshold $\varrho$ is never reached in \eqref{eq:A_self} hence the unique solution is $A=0$ which is a contradiction (and similarly for $\varrho<-A$). Then as $\vert \varrho \vert \leq A$, $\text{Arccos} (\varrho/A)\in [0,\pi]$ is well defined and verifies $A\cos(y)\geq \varrho \Leftrightarrow \vert y \vert \leq \text{Arccos} (\varrho/A)$, hence \eqref{eq:A_self} becomes
\begin{equation}\label{eq:resolution_A_theta}
A=2 \int_0^{\text{Arccos}(\varrho/A)} \cos(y)dy = 2 \sin \left( \text{Arccos}(\varrho/A)\right)=2\sqrt{ 1-\left(\dfrac{\varrho}{A}\right)^2}.
\end{equation}
Equation \eqref{eq:resolution_A_theta} has two non-negative solutions $A_+(0)$ and $A_-(0)$ defined in \eqref{eq:def_A0} if and only if $\varrho \in [-1,1]$, which indeed verify $\varrho\in [-A,A]$, hence the result.
\end{proof}

\subsection{When $f$ is a sigmoid}
\label{S:stat_sol_sig}

Here we prove Proposition \ref{prop:stat_sol_sig}, following the previous result when $f=H_\varrho$ and using the fact that $f_{\kappa,\varrho}\xrightarrow[\kappa\to 0]{}H_\varrho$.

\begin{proof}[Proof of Proposition \ref{prop:stat_sol_sig}]
Define the function $g:\mathbb{R}\times (\vert\varrho\vert,+\infty) \to \mathbb{R}$ such that
\begin{equation}\label{eq:def_g(A,kappa)}
\left \{
\begin{array}{ll}
g(\kappa,a):=a-\int_{-\pi}^\pi \cos(y) f_{\kappa,\varrho}\left( a\cos(y)\right)dy, &\quad (\kappa,a)\in  \mathbb{R}_+^*\times(\vert\varrho\vert,+\infty),\\
g(\kappa,a):=a-\int_{-\pi}^\pi \cos(y) H_{\varrho}\left( a\cos(y)\right)dy=a-2\sqrt{1-\left(\dfrac{\varrho}{a}\right)^2},  &\quad (\kappa,a)\in \mathbb{R}_-\times (\vert\varrho\vert,+\infty).
\end{array}
\right.
\end{equation}
As $f_{\kappa,\varrho}\xrightarrow[\kappa\to 0]{} H_\varrho$, by dominated convergence, $g$ is continuous on  $\mathbb{R}\times(\varrho,+\infty)$. It is differentiable on $ \mathbb{R}_+^*\times(\varrho,+\infty)$ and on $ \mathbb{R}_-^*\times(\varrho,+\infty)$, we now focus on its differentiability in $(0,a)$ for any $a\in(\varrho,+\infty)$.  We first show the continuity of $\dfrac{dg}{da}$, that is showing
\begin{equation}\label{eq:continuity-d_Ag}
\lim_{\kappa\to 0} \dfrac{dg}{da}(\kappa,a)=\dfrac{dg}{da}(0,a)=1-\dfrac{2\varrho^2}{a^3 \sqrt{1-\left(\dfrac{\varrho}{a}\right)^2}}.
\end{equation}
For any $\kappa>0$, recalling the definition of $f_{\kappa,\varrho}$ in \eqref{eq:def_sigmoid},
$$\dfrac{dg}{da}(\kappa,a) = 1 -  \int_{-\pi}^\pi \dfrac{\cos(y)^2 e^{-(a\cos(y)-\varrho)/\kappa}}{\kappa \left(1+e^{-(a\cos(y)-\varrho)/\kappa}\right)^2}dy=1 - 2 \int_{0}^\pi \dfrac{\cos(y)^2 e^{-(a\cos(y)-\varrho)/\kappa}}{\kappa \left(1+e^{-(a\cos(y)-\varrho)/\kappa}\right)^2}dy,$$
and by the change of variables $a\cos(y)-\varrho=u$, we get
\begin{align*}
\int_{0}^\pi \dfrac{\cos(y)^2 e^{-(a\cos(y)-\varrho)/\kappa}}{\kappa \left(1+e^{-(a\cos(y)-\varrho)/\kappa}\right)^2}dy&=\int_{-a-\varrho}^{a-\varrho}\dfrac{(u+\varrho)^2 }{a^3 \sqrt{1-\left(\frac{u+\varrho}{a}\right)^2}} \dfrac{e^{-u/\kappa}}{\kappa \left( 1 + e^{-u/\kappa}\right)^2 }du\\
&= \dfrac{1}{a^3} \int_\mathbb{R} h(-u) \varphi_\kappa(u) du = \dfrac{1}{a^3} (h \ast \varphi_\kappa)(0)
\end{align*}
with $h(u):=\mathbf{1}_{(\varrho-a,a+\varrho)}(u) \dfrac{(-u+\varrho)^2}{ \sqrt{1-\left(\frac{-u+\varrho}{a}\right)^2}}$ and $\varphi_\kappa(u):=\dfrac{e^{-u/\kappa}}{\kappa \left( 1 + e^{-u/\kappa}\right)^2}$. By Lemma \ref{lem:approximate_id}, $ \left(h \ast \varphi_\kappa\right)(0)\xrightarrow[\kappa\to 0]{}h(0)=\dfrac{\varrho^2}{\sqrt{1-\left(\dfrac{\varrho}{a}\right)^2}}$ and \eqref{eq:continuity-d_Ag} follows.
We show now the continuity of $\dfrac{dg}{d\kappa}$, that is 
\begin{equation}\label{eq:continuity-dAkappa}
\lim_{\kappa\to 0} \dfrac{dg}{d\kappa}(\kappa,a)=0.
\end{equation}
For any $\kappa>0$, we obtain similarly
$$\dfrac{dg}{d\kappa}\left(\kappa,a\right)=2 \int_0^\pi \dfrac{\cos(y) \left(a\cos(y)-\varrho\right) e^{-(a\cos(y)-\varrho)/\kappa}}{\kappa^2 \left( 1+e^{-(a\cos(y)-\varrho)/\kappa}\right)^2}dy=  \dfrac{2}{a^2\kappa} \int_{(-a-\varrho)/\kappa}^{(a-\varrho)/\kappa} \tilde{h}(\kappa v) \dfrac{e^{-v}}{ \left( 1 + e^{-v}\right)^2}dv$$
with $\tilde{h}(u):=\dfrac{u(u+\varrho)}{\sqrt{1-\left(\frac{u+\varrho}{a}\right)^2}}$. Let $\displaystyle F(\kappa):=\int_0^{(a-\varrho)/\kappa} \tilde{h}(\kappa v) \dfrac{e^{-v}}{ \left( 1 + e^{-v}\right)^2}dv$, by dominated convergence $F(\kappa)\xrightarrow[\kappa\to 0]{} \tilde{h}(0) \int_0^\infty \dfrac{e^{-v}}{ \left( 1 + e^{-v}\right)^2}dv = \dfrac{\tilde{h}(0)}{2}=0$. Setting $F(0):=0$, $F$ is continuous on $[0,\infty)$ and differentiable on $(0,\infty)$ with 
$$F'(\kappa)=-\dfrac{(a-\varrho)}{\kappa^2} \tilde{h}(a-\varrho)\dfrac{e^{-(a-\varrho)/\kappa}}{ \left( 1 + e^{-(a-\varrho)/\kappa}\right)^2} + \int_0^{(a-\varrho)/\kappa} v\tilde{h}'(\kappa v) \dfrac{e^{-v}}{ \left( 1 + e^{-v}\right)^2}dv.$$
By dominated convergence, $F'(\kappa)\xrightarrow[\kappa\to 0]{} 0+\tilde{h}'(0)\int_0^\infty \dfrac{ve^{-v}}{ \left( 1 + e^{-v}\right)^2}dv=\tilde{h}'(0)\ln(2)=\dfrac{\varrho\ln(2)}{\sqrt{1+\left(\frac{\varrho}{a}\right)^2}}$. Hence by Taylor's theorem, $F(\kappa) = \kappa \dfrac{\varrho\ln(2)}{\sqrt{1+\left(\frac{\varrho}{a}\right)^2}}+o(\kappa)$ as $\kappa\to 0$. Similarly, let 
$$G(\kappa):=\int_{(-a-\varrho)/\kappa}^0 \tilde{h}(\kappa v) \dfrac{e^{-v}}{ \left( 1 + e^{-v}\right)^2}dv=\int_0^{(a+\varrho)/\kappa} \tilde{h}(-v\kappa) \dfrac{e^{v}}{ \left( 1 + e^{v}\right)^2}dv,$$
we also have $G(\kappa)\to 0$. Setting $G(0):=0$, $G$ is differentiable on $(0,\infty)$ with
$$G'(\kappa)=-\dfrac{a+\varrho}{\kappa^2}\tilde{h}(-a-\varrho)\dfrac{e^{(a+\varrho)/\kappa}}{ \left( 1 + e^{(a+\varrho)/\kappa}\right)^2} -\int_0^{(a+\varrho)/\kappa} v\tilde{h}'(\kappa v)\dfrac{e^{v}}{ \left( 1 + e^{v}\right)^2}dv \xrightarrow[\kappa\to 0]{} -\tilde{h}'(0) \ln(2).$$ Hence by Taylor's theorem $G(\kappa)=-\kappa \dfrac{\varrho\ln(2)}{\sqrt{1+\left(\frac{\varrho}{a}\right)^2}} + o(\kappa)$ as $\kappa\to 0$. We obtain then
$$\dfrac{dg}{d\kappa}\left(\kappa,a\right)= \dfrac{2}{a^2\kappa} \left( F(\kappa) + G(\kappa)\right) =  \dfrac{2}{a^2\kappa} \left( \kappa \tilde{h}'(0)\ln(2) - \kappa \tilde{h}'(0)\ln(2) + o(\kappa)\right)=o(1),$$ hence \eqref{eq:continuity-dAkappa} is true. We have shown that $g$ is indeed $\mathcal{C}^1$ on $\mathbb{R}\times (\vert\varrho\vert,+\infty)$. 

Our aim is to apply the implicit function theorem. With Proposition \ref{prop:stat_sol_H}, we have that $g(0,A_+(0))=0$. Let us show that $\dfrac{dg}{da}(0,A_+(0)) \neq 0$. Using \eqref{eq:def_A0}, we obtain
$$\dfrac{dg}{da}(0,A_+(0))= 1-\dfrac{2\varrho^2}{2\left(1 + \sqrt{1-\varrho^2} \right)\sqrt{2 + 2\sqrt{1-\varrho^2} -\varrho^2}},$$
we then need $\varrho^2\neq  \left(1+ \sqrt{1-\varrho^2} \right)\sqrt{2 + 2\sqrt{1-\varrho^2} -\varrho^2}$, which is true if and only if $\varrho \neq 1$. We conclude by implicit function theorem.

It remains now to prove  that there exists  $\kappa_1>0$ such that for any $\kappa \in (0,\kappa_1)$, $I(1,\kappa)=\int_S f_{\kappa,\varrho}'(A(\kappa)\cos(x))dx\in (1,2)$. We have \begin{align*}
\mathcal{I}(1,\kappa)&=2 \int_0^\pi  \dfrac{e^{-(A(\kappa)\cos(y)-\varrho)/\kappa}}{\kappa\left(1+e^{-(A(\kappa)\cos(y)-\varrho)/\kappa}\right)^2}dy\\ 
&=2\int_{-A(\kappa)-\varrho}^{A(\kappa)-\varrho} \dfrac{1}{A(\kappa)\sqrt{1-\left(\frac{u+\varrho}{A(\kappa)}\right)^2}} \dfrac{ e^{- u/\kappa}}{\kappa(1+e^{- u/\kappa})^2}= h\ast \phi_{\kappa}(0)\xrightarrow[\kappa\to 0]{} \dfrac{2}{A_+(0)\sqrt{1-\frac{\varrho^2}{A_+(0)^2}}},
\end{align*}
with $h(u)=\mathbf{1}_{(\varrho-A(\kappa), A(\kappa)+\varrho)}(u)\dfrac{2}{A(\kappa)\sqrt{1-\left(\frac{-u+\varrho}{A(\kappa)}\right)^2}}$ and using Lemma \ref{lem:approximate_id} and as $A(\kappa)\xrightarrow[\kappa\to 0]{} A_+(0)$ defined in \eqref{eq:def_A0}. As
$$\dfrac{1}{A_+(0)\sqrt{1-\frac{\varrho^2}{A_+(0)^2}}}=\dfrac{1}{\sqrt{2 + 2\sqrt{1-\varrho^2}-\varrho^2}}<1$$
when $\varrho\in (-1,1)$, by continuity of $\kappa\mapsto A(\kappa)$ there exists $\kappa_1>0$ such that for $\kappa<\kappa_1$, we have indeed $I(1,\kappa)<2$. Let us show know that for small $\kappa$ we have also $I(1,\kappa)>1$. We have
$$I(1,0)-1=\dfrac{2}{\sqrt{2+2\sqrt{1-\varrho^2}-\varrho^2}}-1=\dfrac{2-\sqrt{2+2\sqrt{1-\varrho^2}-\varrho^2}}{\sqrt{2+2\sqrt{1-\varrho^2}-\varrho^2}},$$
and as $2\sqrt{1-\varrho^2}-\varrho^2<2$ we have indeed $I(1,0)-1>0$. Similarly by continuity it implies that $I(1,\kappa)>1$ for $\kappa$ small enough.
\end{proof}

\section{Appendix: Some computations}
\subsection{Control of the noise perturbation}
\label{S:noise}

We prove here Proposition \ref{prop:control_noise}, which is a part of the \textbf{Step 2} of the proof of Theorem \ref{thm:UN_close_U} in Section \ref{S:long_time}. The proof relies on a adaptation of an argument given in \cite{Zhu2017} (Theorem 4.3), where a similar quantity to the following \eqref{eq:zeta_n_chi} is considered for $N=1$, and used in the proof of Proposition 4.2 of \cite{agathenerine_longtime_arxiv}. 

\begin{proof}[Proof of Proposition \ref{prop:control_noise}]  Recall the expression of $\left(Z_{N,j}\right)_{1\leq j \leq N}$ in \eqref{eq:def_ZiN}. Introduce the compensated measure $\tilde{\pi}_j(ds,dz):=\pi_j(ds,dz)-\lambda_{N,j}dsdz$, so that with the linearity of $(e^{t\mathcal{L}_{\phi_ {n-1}}})_{t\geq 0}$, we obtain that $\zeta_n$ can be written as
\begin{equation}\label{eq:zeta_n_chi}
 \zeta_n(t) = \sum_{j=1}^N \int_0^t\int_0^\infty e^{(t-s)\mathcal{L}_{\phi_{n-1}}}\chi_j(s,z) \tilde{\pi}_j(ds,dz),
\end{equation}
with 
\begin{equation}\label{eq:def_chi}
\chi_j(s,z):=\left( \sum_{i=1}^N \mathbf{1}_{B_{N,i}} \dfrac{w_{ij}^{(N)}}{N} \mathbf{1}_{z\leq \lambda_{N,j}(s)}\right).
\end{equation}
Fix $m\geq 1$. The functional $\phi:L^2(I)\to \mathbb{R}$ given by $\phi(v)=\Vert v \Vert_2^{2m}$ is of class $\mathcal{C}^2$ (recall that $\zeta_n(t) \in L^2(I)$) so that by Itô formula on the expression \eqref{eq:zeta_n_chi} we obtain
\begin{align}\label{eq:zeta_spatial_def}
\phi\left(\zeta_n(t)\right)&= \int_0^t \phi'\left(\zeta_n(s)\right) \mathcal{L}_{\phi_{n-1}}\left(\zeta_n(s)\right)ds + \sum_{j=1}^N \int_0^t \int_0^\infty \phi'\left(\zeta_n(s-)\right)\chi_j(s,z)\tilde{\pi}_j(ds,dz) \notag\\+ &\sum_{j=1}^N\int_0^t\int_0^\infty \left[ \phi\left( \zeta_n(s-)+\chi_j(s,z)\right) - \phi\left(\zeta_n(s-)\right) - \phi'\left(\zeta_n(s-)\right)\chi_j(s,z)\right]\pi_j(ds,dz)\notag\\
&:= I_0(t) + I_1(t) + I_2(t).
\end{align}
We also have that for any $v,h,k\in L^2(I)$, $\phi'(v)h=2m\Vert v \Vert_2^{2m-2}\left(\langle v,h\rangle\right)\in\mathbb{R}$ and $\phi''(v)(h,k)=2m(2m-1)\Vert v \Vert_2^{2m-4}\langle v,k \rangle \langle v,h \rangle +2m\Vert v \Vert^{2m-2} \langle h,k\rangle$.

We have $I_0(t)= \int_0^t 2m \Vert \zeta_N(s) \Vert_2^{2m-2}\left( \langle \zeta_N(s),\mathcal{L}(\zeta_N(s))\rangle \right)ds$.  From Proposition \ref{prop:spect_L_phi}, $\mathcal{L}_{\phi_{n-1}}$ has only three non-positive eigenvalues hence by Lumer-Philipps Theorem (see Section 1.4 of \cite{Pazy1974}), $\left( \langle \zeta_n(s),\mathcal{L}_{\phi_{n-1}}(\zeta_Nns))\right)\rangle\leq 0$. Then for any $t\geq 0$, $I_0(t)\leq 0$. 

Let some $\varepsilon>0$ to be chosen later. About $I_1$, using Burkholder-Davis-Gundy inequality, the independence of the family $\left(\pi_j\right)$ and  H{\"o}lder inequality with some well chosen parameter, one can show
\begin{multline}\label{eq:spatial_I1}
\mathbf{E}\left[ \sup_{s\leq T} \left|I_1(s)\right|\right] \leq C (2m-1) \varepsilon \mathbf{E}\left[ \sup_{0\leq s \leq T} \left( \Vert \zeta_n(s)\Vert_2^{2m}\right) \right]\\ + C \varepsilon^{-(2m-1)}\mathbf{E}\left[\left(\sum_{j=1}^N\int_0^T\int_0^\infty \Vert \chi_j(s,z)\Vert_2^2 \pi_j(ds,dz)\right)^m\right],
\end{multline}
as done  for Proposition 4.2 of \cite{agathenerine_longtime_arxiv}, with $C$ some deterministic constant. About $I_2$, using Taylor's Lagrange formula,  H{\"o}lder and Young's inequalities, one can show 
\begin{multline}\label{eq:spatial_I2}
\mathbf{E}\left[ \sup_{s\leq T} \left|I_2(s)\right|\right]  \leq m(2m-2) \varepsilon \mathbf{E}\left[\sup_{0\leq s \leq t} \left( \Vert \zeta_n(s)\Vert_2^{2m}\right)\right] \\+ 2m \varepsilon^{-(2m-2)}\mathbf{E}\left[\left(\sum_{j=1}^N\int_0^t\int_0^\infty \Vert \chi_j(s,z)\Vert_2^2 \pi_j(ds,dz)\right)^m\right].
\end{multline}

Taking the expectation in \eqref{eq:zeta_spatial_def} and fixing $\varepsilon$ such that $\varepsilon\left( C(2m-1)+m(2m-2) \right) \leq \frac{1}{2}$ , we get
$$\mathbf{E}\left[\sup_{s\leq T} \Vert \zeta_n(s) \Vert_2^{2m}\right] \leq 2 C\mathbf{E}\left[\left(\sum_{j=1}^N\int_0^T\int_0^\infty \Vert \chi_j(s,z)\Vert_2^2 \pi_j(ds,dz)\right)^m\right],
$$
where $C>0$ depends only on $m$. As $\sup_{i,j}w_{ij}^{(N)}\leq 2\pi$,
\begin{equation}\label{eq:norme_chi}
\Vert\chi_j(s,z)\Vert_2^2 = \dfrac{1}{N} \sum_{i=1}^N\left(\dfrac{w_{ij}^{(N)}}{N}\right)^2  \mathbf{1}_{z\leq \lambda_{N,j}(s)} \leq \dfrac{4\pi^2}{N^2}   \mathbf{1}_{z\leq \lambda_{N,j}(s)}.
\end{equation}
As $f$ is bounded by 1, we have that 
$ \displaystyle \mathbf{E}\left[\sup_{s\leq T} \Vert \zeta_n(s) \Vert_2^{2m}\right] \leq  \dfrac{C}{N^m} \mathbf{E}\left[\dfrac{1}{N}\sum_{j=1}^N \widetilde{Z}_{j}(T)^m\right],$
where $\left(\widetilde{Z}_j(t)\right)$ are i.i.d copies of a Poisson process on intensity 1. Hence for some constant $C=C(T,m,\kappa,\varrho)>0$, for any $1\leq n \leq n_f$,
$\displaystyle\mathbf{E}\left[ \sup_{0\leq t \leq T}  \left\Vert \zeta_{n}(t) \right\Vert_2^{2m} \right]\leq \dfrac{C}{N^m}$.
It implies 
$$\mathbf{P} \left( \sup_{t\in[0,T]} \left\Vert \zeta_{n}(t) \right\Vert_2  \geq \dfrac{N^\eta}{\sqrt{N}}\right) \leq \dfrac{\mathbf{E}\left[ \sup_{0\leq t \leq T}  \left\Vert \zeta_{n}(t) \right\Vert_2^{2m} \right] }{N^{2\eta m}} N^m \leq C N^{-2m\eta},$$
hence by a union bound $\mathbf{P}(A_N^C)\leq C n_f  N^{-2m\eta}=CN^{\alpha-2m\eta}$. We can then choose $m$ large enough to obtain the result of Proposition \ref{prop:control_noise}. 
\end{proof}

\subsection{Analysis complements}

\begin{lem}\label{lem:approximate_id}
Define $\varphi(u)=\dfrac{e^{-u}}{\left(1+e^{-u}\right)^2}$. For any $\kappa>0$, let $\varphi_\kappa(u):=\frac{1}{\kappa} \varphi\left(\frac{u}{\kappa}\right)$. Then $\left(\varphi_\kappa\right)_{\kappa>0}$ is an approximate identity and $\varphi_\kappa \ast h \xrightarrow[\kappa\to 0]{} h$ for any $h\in L^p$, with $1\leq p < \infty$.
\end{lem}

\begin{proof}
It suffices to check that 
$$ \int_\mathbb{R}\varphi(u)du= \int_\mathbb{R}\dfrac{e^{-u}}{(1+e^{-u})^2}du= \left[\dfrac{1}{1+e^{-u})}\right]_{-\infty}^{+\infty}=1.$$
\end{proof}

\begin{lem}\label{lem:approx_int}
Let $N\geq 1$, recall that $S=[-\pi,\pi)$ and its regular subdivision $x_{i}=\dfrac{2i\pi}{N}-\pi$ for $0\leq i \leq N$. For any function $g\in\mathcal{C}^2(I,\mathbb{R})$, we have
\begin{equation}\label{eq:approx_int_1}
\dfrac{2\pi}{N} \sum_{j=1}^N g(x_j) = \int_S g(y)dy - \dfrac{1}{2} \left( \dfrac{2\pi}{N} \right)^2 \sum_{j=1}^N g'(x_j) + o\left(\dfrac{1}{N}\right).
\end{equation}
Moreover, for any function $h\in\mathcal{C}^1(I,\mathbb{R})$, we have
\begin{equation}\label{eq:approx_int_2}
\sum_{i=1}^N h(x_i) \int_{x_{i-1}}^{x_i} g(y)dy = \int_S h(x)g(x)dx - \sum_{i=1}^N h'(x_i) \int_{x_{i-1}}^{x_i} (y-x_i)g(y)dy+ o\left(\dfrac{1}{N}\right).
\end{equation}
\end{lem}

\begin{proof}
Let $C_j=(x_{j-1},x_j)$ for  $1\leq j\leq N$. From Taylor's expansion,
$g(y)= g(x_j)+g'(x_j)(y-x_j)+\int_{x_j}^yg''(t)(y-t)dt$ hence the result \eqref{eq:approx_int_1} as $\int_S g(y)dy=\sum_{j=1}^N \int_{C_j} g(y)dy$. About  \eqref{eq:approx_int_2}, we proceed similarly as
$$ \int_S hg= \sum_j\int_{C_j}g(y) \left( h(x_j)+h'(x_i)(y-x_j)+\int_{x_j}^y h''(t)(y-t) dt\right)dy.$$
\end{proof}

\subsection{Auxilliary lemmas}\label{S:appendix_computation}
\subsubsection{About the derivatives of the isochron}

\begin{lem}\label{lem:computations_Dtheta2_chi} Let $\phi\in S$. There exists $C=C_{A,\gamma}$ such that
\begin{equation}
\label{eq:computations_Dtheta2_chi}
D^2\theta(u_{\phi})[\chi_j(s,z)]^2=  \mathbf{ 1}_{ z\leq \lambda_{ N, j}(s)} \left( \dfrac{ C}{ N^{ 2}} \cos \left(x_{ j}+ \phi\right) \sin(x_{ j}+ \phi) + O(N^{ -3})\right),
\end{equation}
where the notation $ O(N^{-3})$ is uniform in $(s,z,\phi)$.
\end{lem}

\begin{proof}
Recall \eqref{eq:diff2_theta}, we have
\begin{multline}\label{eq:diff2_theta_aux1}
D^2\theta(u_{\phi})[\chi_j(s,z)]^2= \dfrac{1}{2A^2} \left(  2\alpha_\phi^\circ(\chi_j(s,z))\beta_\phi(v_\phi,\chi_j(s,z))+\beta_{ \phi}(\chi_j(s,z),\chi_j(s,z))\right) \\
+ \frac{ 1+ \gamma}{ A^{ 2}(1- \gamma)} \alpha_\phi^\gamma(\chi_j(s,z))\beta_\phi(u_\phi,\chi_j(s,z))  - \frac{ (2- \gamma)(1+ \gamma)}{ 2(1-\gamma)} \left( \alpha_\phi^\circ(\chi_j(s,z))^2+\alpha_\phi^\gamma(\chi_j(s,z))^2\right),
\end{multline}
Let us compute each term. About $\alpha$, using some trigonometric formula and Lemma \ref{lem:approx_int} we have
\begin{align*}
& ~ \alpha_{\phi}^\circ(\chi_j(s,z))\\
 &= \dfrac{\langle \chi_j, v_{\phi} \rangle_{\phi} }{A}=\dfrac{1}{A} \int_S \chi_j v_\phi f'(u_\phi)= \dfrac{2\pi}{A}  \mathbf{1}_{z\leq \lambda_{N,j}(s)} \sum_{i=1}^N  \dfrac{\cos(x_i-x_j)}{N}  \int_S v_\phi f'(u_\phi) \mathbf{1}_{B_{N,i}}\\
&= \dfrac{2\pi}{AN}  \mathbf{1}_{z\leq \lambda_{N,j}(s)} \left(   \cos(x_j+\phi) \sum_{i=1}^N  \cos(x_i+\phi) \int_S v_\phi f'(u_\phi) \mathbf{1}_{B_{N,i}} +  \sin(x_j+\phi)\sum_{i=1}^N  \sin(x_i+\phi)\int_S v_\phi f'(u_\phi) \mathbf{1}_{B_{N,i}}\right)\\
&= \dfrac{2\pi}{AN}  \mathbf{1}_{z\leq \lambda_{N,j}(s)} \left(   \cos(x_j+\phi) \int_S\cos(x+\phi)  v_\phi(x) f'(u_\phi(x))dx  +  \sin(x_j+\phi)\int_S \sin(x+\phi) v_\phi(x) f'(u_\phi(x))dx + O(N^{-1})\right)\\
&= \dfrac{2\pi}{AN}  \mathbf{1}_{z\leq \lambda_{N,j}(s)}   \sin(x_j+\phi)\int_S \sin(x+\phi) v_\phi(x) f'(u_\phi(x))dx + \mathbf{1}_{z\leq \lambda_{N,j}(s)}O(N^{-2})\\
&= \mathbf{1}_{z\leq \lambda_{N,j}(s)}  \left(- \dfrac{2\pi}{N}  \sin(x_j+\phi) \mathcal{I}(\sin^2) +O(N^{-2})\right)= \mathbf{1}_{z\leq \lambda_{N,j}(s)}  \left(- \dfrac{2\pi}{N} \sin(x_j+\phi) +O(N^{-2})\right),
\end{align*}
using Lemma \ref{lem:computation_A}. We prove in a same way that $ \alpha_{\phi}^\gamma( \chi_{ j}(s,z))= \mathbf{ 1}_{ z\leq \lambda_{ N, j}(s)}  \left(\dfrac{ 2 \pi}{ \Vert u_\phi\Vert_\phi N} \cos(x_{ j}+ \phi) (\mathcal{I}(1)-1)+ O(N^{ -2})\right)$. About $\beta$, we have similarly using Lemma \ref{lem:approx_int} that
\begin{align*}
&~ \beta_{\phi}(v_\phi,\chi_j(s,z))\\
& =  \int_Sf''(u_{\phi}(y))v_{\phi}(y)^2 \chi_j(s,z)(y)dy
= \sum_{i=1}^N \dfrac{w_{ij}^{(N)}}{N} \mathbf{1}_{z\leq \lambda_{N,j}(s)} \int_S f''(u_{\phi}(y))v_{\phi}(y)^2 \mathbf{1}_{B_{N,i}}(y)dy\\
&= \mathbf{1}_{z\leq \lambda_{N,j}(s)}  \dfrac{2\pi}{N} \left( \cos(x_j+\phi)\sum_{i=1}^N \cos(x_i+\phi)\int_S f''(u_{\phi}(y))v_{\phi}(y)^2 \mathbf{1}_{B_{N,i}}(y)dy \right.\\&\quad \left.+  \sin(x_j+\phi)\sum_{i=1}^N \sin(x_i+\phi) \int_S f''(u_{\phi}(y))v_{\phi}(y)^2 \mathbf{1}_{B_{N,i}}(y)dy \right)\\
& =\mathbf{1}_{z\leq \lambda_{N,j}(s)}  \dfrac{2\pi}{N} \left( \cos(x_j+\phi) \int_S \cos(y+\phi)  f''(u_{\phi}(y))v_{\phi}(y)^2dy \right.\\ &\quad \left.+ \sin(x_j+\phi) \int_S \sin(y+\phi)f''(u_{\phi}(y))v_{\phi}(y)^2dy+O\left( \dfrac{1}{N} \right) \right)\\
&= \mathbf{1}_{z\leq \lambda_{N,j}(s)}  \dfrac{2\pi}{N} \left( \cos(x_j+\phi) \int_S \cos(y+\phi)  f''(u_{\phi}(y))v_{\phi}(y)^2dy+O\left( \dfrac{1}{N} \right) \right).
\end{align*}
With Lemma \ref{lem:computation_A} and an integration by parts, we obtain
\begin{align*}
\int_S \cos(y+\phi)  f''(u_{\phi}(y))v_{\phi}(y)^2dy &= A^2 \int_S  \cos(y+\phi)  f''(A\cos(y+\phi))\sin^2(y+\phi)dy\\
&= \int_S \left( -A\sin(y) f''(A\cos(y)\right)\left( -A\sin(y)\cos(y)\right)dy\\
&=  - \int_S f'(A\cos(y))\left( -A + 2A\sin^2\right)dy\\
&= A\left( \mathcal{I}(1) - 2\mathcal{I}(\sin^2)\right)=A\gamma
\end{align*}
recalling  \eqref{eq:def_lambda_stable}, hence $\displaystyle\beta_{\phi}(v_\phi,\chi_j(s,z)) = \mathbf{1}_{z\leq \lambda_{N,j}(s)} \left(\dfrac{2\pi}{N} A\gamma  \cos(x_j+\phi) + O(N^{-2})\right)$.  We prove in a same way that $ \beta_{ \phi}(u_{ \phi}, \chi_{ j}(s,z))= - \mathbf{ 1}_{ z\leq \lambda_{ N, j}(s)} \left(\dfrac{ 2 \pi }{ N}A \gamma \sin(x_{ j}+ \phi) +O(N^{ -2})\right) $. Finally we have
\begin{align*}
&\beta_{\phi}(\chi_j(s,z),\chi_j(s,z))= \int_S f''(u_{\phi}(y))v_{\phi}(y) \left(\sum_{i=1}^N \mathbf{1}_{B_{N,i}}(y) \dfrac{w_{ij}^{(N)}}{N} \mathbf{1}_{z\leq \lambda_{N,j}(s)}\right)^2dy\\
&=  \mathbf{1}_{z\leq \lambda_{N,j}(s)} \left( \dfrac{2\pi}{N} \right)^2 \sum_{i=1}^N \left( \cos(x_i+\phi)\cos(x_j+\phi) + \sin(x_i+\phi)\sin(x_j+\phi)\right)^2 \int_{B_{N,i}}(y) f''(u_{\phi}(y))v_{\phi}(y) dy\\
&=\mathbf{1}_{z\leq \lambda_{N,j}(s)} \left( \dfrac{2\pi}{N} \right)^2 
\left( \cos(x_j+\phi)^2 \int_S \cos(y+\phi)^2 f''(u_{\phi}(y))v_{\phi}(y) dy \right. \\ &\quad \left.+ \sin(x_j+\phi)^2 \int_S \sin(y+\phi)^2 f''(u_{\phi}(y))v_{\phi}(y) dy \right. \\ &\quad \left. +2 \cos(x_j+\phi)\sin(x_j+\phi) \int_S \cos(y+\phi)\sin(y+\phi) f''(u_{\phi}(y))v_{\phi}(y) dy+ O(N^{-1})\right)\\
&=\mathbf{1}_{z\leq \lambda_{N,j}(s)} \left[\left( \dfrac{2\pi}{N} \right)^2 
2 \cos(x_j+\phi)\sin(x_j+\phi) \int_S \cos(y+\phi)\sin(y+\phi) f''(u_{\phi}(y))v_{\phi}(y) dy+O(N^{-3})\right].
\end{align*}
With an integration by parts and recognising \eqref{eq:def_lambda_stable},
\begin{align*}
\int_S \cos(y+\phi)\sin(y+\phi) f''(u_{\phi}(y))v_{\phi}(y) dy &=-A \int_S \cos(y)\sin(y) f''(A\cos(y))\sin(y) dy\\
&= \int_S \left( -A\sin(y) f''(A\cos(y)) \right) \left( \cos(y)\sin(y)\right) dy= - \gamma,
\end{align*}
we obtain that $\displaystyle \beta_{\phi}(\chi_j(s,z),\chi_j(s,z)) = \mathbf{1}_{z\leq \lambda_{N,j}(s)} \left( - 2 \gamma \left( \dfrac{2\pi}{N} \right)^2 
\cos(x_j+\phi)\sin(x_j+\phi) + O(N^{-3})\right)$.
Putting all the previous estimates together in  \eqref{eq:diff2_theta_aux1}, we obtain \eqref{eq:computations_Dtheta2_chi} for some constant $C=C_{ A, \gamma}$.
\end{proof}

\subsubsection{About the fluctuations}

\begin{lem}[Some computations for the proof of Proposition \ref{prop:theta_hat_decomp}]\label{lem:computations_Ai} Let $\phi\in S$. Recall the definitions of $u_\phi$ and $v_\phi$ in \eqref{eq:def_U_circle} and \eqref{eq:def_vphi}. We have
\begin{align}
A_1 &:= \dfrac{2\pi}{N}\sum_{j=1}^N \cos(x_j+\phi)f(u_{\phi}(x_j))=A+  o\left(\frac{1}{N}\right) \label{eq:A1_eq}\\
A_2 &:= \sum_{i=1}^N \cos(x_i+\phi)  \langle v_{\phi}, \mathbf{1}_{B_{N,i}}\rangle_{\phi} =  \dfrac{A\pi}{N}+ o\left(\frac{1}{N}\right) \label{eq:A2_eq}\\
A_3 & := \dfrac{2\pi}{N}\sum_{j=1}^N \sin(x_j+\phi)f(u_{\phi}(x_j)) =  o\left(\frac{1}{N}\right) \label{eq:A3_eq}\\
A_4 &:= \sum_{i=1}^N \sin(x_i+\phi)  \langle v_{\phi}, \mathbf{1}_{B_{N,i}}\rangle_{\phi} = -A +  o\left(\frac{1}{N}\right), \label{eq:A4_eq}
\end{align}
where the notation $ o\left(\frac{1}{N}\right)$ is uniform in the choice of $\phi$.
\end{lem}

\begin{proof}
From Lemma \ref{lem:approx_int}, more especially \eqref{eq:approx_int_1} applied to $g(y)=\cos(y+\phi)f(u_{\phi}(y))$, we have that
\begin{align*}
A_1  &= \int_S \cos(x+\phi)f(u_{\phi}(x))dx + \dfrac{2\pi^2}{N^2} \sum_{j=1}^N \left( \sin(x_j+\phi)f(u_{\phi}(x_j))-\cos(x_j+\phi)f'(u_{\phi}(x_j))v_{\phi}(x_j)\right) + o\left(\frac{1}{N}\right)\\
&= A +\dfrac{2\pi^2}{N^2} \sum_{j=1}^N \left( \sin(x_j+\phi)f(u_{\phi}(x_j))-\cos(x_j+\phi)f'(u_{\phi}(x_j))v_{\phi}(x_j)\right) + o\left(\frac{1}{N}\right) = A+  o\left(\frac{1}{N}\right),
\end{align*}
using \eqref{eq:A_kappa&Isin} and as
\begin{multline*}
\dfrac{2\pi}{N} \sum_{j=1}^N \left( \sin(x_j+\phi)f(u_{\phi}(x_j))-\cos(x_j+\phi)f'(u_{\phi}(x_j))v_{\phi}(x_j)\right) \\= \int_S \sin (y+\phi)f(A\cos(y+\phi)dx +A \int_S \cos(y+\phi) f'(A\cos(y+\phi)) \sin(y+\phi) +  O\left(\frac{1}{N}\right) = O\left(\frac{1}{N}\right).
\end{multline*}
Similarly we can prove \eqref{eq:A3_eq} as
\begin{align*}
A_3 &= \int_S \sin(x+\phi)f(u_{\phi}(x))dx - \dfrac{2\pi^2}{N^2} \sum_{j=1}^N \left( \cos(x_j+\phi)f(u_{\phi}(x_j))+\sin(x_j+\phi)f'(u_{\phi}(x_j))v_{\phi}(x_j)\right) + o\left(\frac{1}{N}\right)\\
&=  - \dfrac{2\pi^2}{N^2} \sum_{j=1}^N \left( \cos(x_j+\phi)f(u_{\phi}(x_j))+\sin(x_j+\phi)f'(u_{\phi}(x_j))v_{\phi}(x_j)\right) + o\left(\frac{1}{N}\right)=  o\left(\frac{1}{N}\right),
\end{align*}
using that $\int_S \sin(x+\phi)f(u_{\phi}(x))dx=0$ by symmetry and
\begin{multline*}
\dfrac{2\pi}{N} \sum_{j=1}^N \left( \cos(x_j+\phi)f(u_{\phi}(x_j))+\sin(x_j+\phi)f'(u_{\phi}(x_j))v_{\phi}(x_j)\right)\\= \int_S \left( \cos f(A\cos) - \sin f'(A\cos) A\sin \right) +  O\left(\frac{1}{N}\right)=A-A+O\left(\frac{1}{N}\right)=O\left(\frac{1}{N}\right).
\end{multline*}

From Lemma \ref{lem:approx_int}, more especially \eqref{eq:approx_int_2} applied to $g(y)=v_{\phi}(y)f'(u_{\phi}(y))$ and $h(x)=\cos(x+\phi)$, we have that
\begin{align*}
A_2 &= \sum_{i=1}^N \cos(x_i+\phi)  \langle v_{\phi}, \mathbf{1}_{B_{N,i}}\rangle_{\phi}= \sum_{i=1}^N \cos(x_i+\phi)  \int_{B_{N,i}}  v_{\phi}(y) f'(u_{\phi}(y)) dy\\
&= \int_S \cos(x+\phi)v_{\phi}(x) f'(u_{\phi}(x))dx + \sum_{i=1}^N \sin(x_i+\phi) \int_{B_{N,i}}(y-x_i)v_{\phi}(y)f'(u_{\phi}(y))dy+o\left(\frac{1}{N}\right)\\
&=\sum_{i=1}^N \sin(x_i+\phi) \int_{B_{N,i}}(y-x_i)v_{\phi}(y)f'(u_{\phi}(y))dy+o\left(\frac{1}{N}\right)\\
&= \sum_{i=1}^N \sin(x_i+\phi)v_{\phi}(x_i)f'(u_{\phi}(x_i)) \int_{B_{N,i}}(y-x_i)dy+o\left(\frac{1}{N}\right)\\
&=- \sum_{i=1}^N \sin(x_i+\phi) v_{\phi}(x_i)f'(u_{\phi}(x_i)) \frac{1}{2}\left(\dfrac{2\pi}{N}\right)^2+o\left(\frac{1}{N}\right)\\
&=\dfrac{\pi}{N} \left(A \int \sin(x+\phi)^2 f'(A\cos(x+\phi))dx + O\left(\frac{1}{N}\right) \right)+o\left(\frac{1}{N}\right)= \dfrac{A\pi}{N}+ o\left(\frac{1}{N}\right)
\end{align*}
and similarly, for the choice $h(x)=\sin(x+\phi)$ and using \eqref{eq:A_kappa&Isin}
\begin{align*}
A_4 &= \sum_{i=1}^N \sin(x_i+\phi)  \langle v_{\phi}, \mathbf{1}_{B_{N,i}}\rangle_{\phi}= \sum_{i=1}^N \sin(x_i+\phi)  \int_{B_{N,i}}  v_{\phi}(y) f'(u_{\phi}(y)) dy\\
&= \int_S \sin(x+\phi)v_{\phi}(x) f'(u_{\phi}(x))dx - \sum_{i=1}^N \cos(x_i+\phi) \int_{B_{N,i}}(y-x_i)v_{\phi}(y)f'(u_{\phi}(y))dy+o\left(\frac{1}{N}\right)\\
&=-A  - \sum_{i=1}^N \cos(x_i+\phi) \int_{B_{N,i}}(y-x_i)v_{\phi}(y)f'(u_{\phi}(y))dy+o\left(\frac{1}{N}\right)\\
&=-A  +A \sum_{i=1}^N \cos(x_i+\phi) \int_{B_{N,i}}(y-x_i)\sin(y+\phi)f'(u_{\phi}(y))dy+o\left(\frac{1}{N}\right).
\end{align*}
As
\begin{align*}
&\sum_{i=1}^N \cos(x_i+\phi) \int_{B_{N,i}}(y-x_i)\sin(y+\phi)f'(u_{\phi}(y))dy\\
&=\sum_{i=1}^N \cos(x_i+\phi) \sin(x_i+\phi)f'(u_{\phi}(x_i))\int_{B_{N,i}}(y-x_i)dy+O\left(\frac{1}{N^2}\right)\\
&= -\dfrac{\pi}{N} \dfrac{2\pi}{N}\sum_{i=1}^N \cos(x_i+\phi) \sin(x_i+\phi)f'(u_{\phi}(x_i))+o\left(\frac{1}{N}\right)\\
&= -\dfrac{\pi}{N} \int_S\cos(x+\phi) \sin(x+\phi)f'(u_{\phi}(x))dx +O\left(\frac{1}{N^2}\right)+o\left(\frac{1}{N}\right)= o\left(\frac{1}{N}\right),
\end{align*}
we obtain \eqref{eq:A4_eq}.
\end{proof}

\bibliographystyle{abbrv}

\end{document}